\documentclass{article}

\usepackage[a4paper,bindingoffset=0.2in,left=1in,right=1in,top=1in,bottom=1in,footskip=.25in,
]{geometry}  \usepackage[utf8]{inputenc}
\usepackage[T1]{fontenc}
\usepackage{amsmath,amssymb}
\usepackage{amsthm}\usepackage{tikz}
\usepackage{subcaption}
\usepackage{mathtools}
\usepackage{booktabs}
\usepackage{microtype}
\usepackage{bm}
\usepackage[section]{placeins} \usepackage{relsize}  \usepackage{vruler}
\usepackage{etoolbox} 
\usepackage[numbers,sort&compress]{natbib}

\usepackage[colorlinks,citecolor=blue]{hyperref}
\usepackage{doi} 

\usetikzlibrary{shapes,positioning,fit,matrix,arrows,arrows.meta}
\usetikzlibrary{decorations.pathmorphing,decorations.pathreplacing,decorations.markings}
\usetikzlibrary{calc}

\DeclareMathSymbol{\shortminus}{\mathbin}{AMSa}{"39}

\tikzset{
dot/.style={draw,fill,circle,inner sep = 0pt,minimum size = 3pt},
bigdot/.style={dot,minimum size = 4pt},
terminal/.style={draw,circle, inner sep=2.5pt},
vcolour/.style={draw,inner sep=1.5pt,font=\scriptsize,label distance=2pt},
22box/.style={draw,minimum width=2cm,minimum height=1.5cm,font=\LARGE,node contents=\( \Downarrow \)},
22boxsmall/.style={22box,minimum width=1.5cm, minimum height=1cm},
Vset/.style={draw=black!30,ellipse,minimum width=1.5cm,minimum height=2.5cm},
petal/.style={
    decoration = {markings, mark=at position 1.0 with {
\coordinate (centre) at (0,0);
\path (centre)+(120:2) node(v)[bigdot]{};
\path (centre)+(240:2) node(u)[bigdot]{};
\path (centre)+(0:2) node(w)[bigdot]{};
\draw (v)--(u)--(w);\draw (v) arc (120:-120:2);
\path (v)--(u) node(v1)[pos=0.4][bigdot]{} node(u-v)[pos=0.6][bigdot]{};
\path (u)--(w) node(u-w)[pos=0.4][bigdot]{} node(w-u)[pos=0.6][bigdot]{};
\draw (w)--(v) node(w-v)[pos=0.4][bigdot]{} node(x1)[pos=0.6][bigdot]{};
\draw (u-w)--(u-v)--node[pos=0.3][bigdot]{} (w-v)--(w-u)--node[pos=0.3][bigdot]{} (v1)--(x1)--node[pos=0.3][bigdot]{} (u-w);
}},
    postaction = {decorate}
    },
}

\newtheorem{theorem}{Theorem}\newtheorem{corollary}{Corollary}\newtheorem{lemma}{Lemma}\newtheorem{observation}{Observation}\newtheorem{conjecture}{Conjecture}\newtheorem{problem}{Problem}

\newtheoremstyle{freethm}{3pt}{3pt}{}{}{\bfseries}{.\\}{.5em}{\thmname{#1}\thmnumber{ #2}\thmnote{ (#3)}}

\theoremstyle{freethm}
\newtheorem{construct}{Construction}

\newtheoremstyle{thmpart}{3pt}{3pt}{}{}{\bfseries}{:}{.5em}{\thmname{#1}\thmnumber{ #2}\thmnote{ (#3)}}

\theoremstyle{thmpart}
\newtheorem{claim}{Claim}\theoremstyle{plain}

\theoremstyle{freethm}

\newcommand{\pagetarget}[2]{\phantomsection \label{#1}\hypertarget{#1}{#2}}

\providetoggle{forThesis}  \togglefalse{forThesis}

\title{Hardness Transitions and Uniqueness of Acyclic Colouring}
\author{Shalu M.A., and Cyriac Antony\\
Indian Institute of Information Technology, Design \& Manufacturing\\ (IIITDM) Kancheepuram, Chennai, India}
\date{}

\begin{document}

\maketitle

\providetoggle{forThesis}

\providetoggle{extended}

\providetoggle{iiitdmV} 

\begin{abstract}
For \( k\in \mathbb{N} \), a \( k\)-acyclic colouring of a graph \( G \) is a function \( f\colon V(G)\to \{0,1,\dots,k-1\} \) such that (i)~\( f(u)\neq f(v) \) for every edge \( uv \) of \( G \), and (ii)~there is no cycle in \( G \) bicoloured by \( f \). 
For \( k\in \mathbb{N} \), the problem \textsc{\( k \)\nobreakdash-Acyclic Colourability} takes a graph \( G \) as input and asks whether \( G \) admits a \( k \)-acyclic colouring. 
Ochem (EuroComb 2005) proved that \textsc{3-Acyclic Colourability} is NP-complete for bipartite graphs of maximum degree~4. 
Mondal et al.\ (J.~Discrete Algorithms, 2013) proved that \textsc{4-Acyclic Colourability} is NP-complete for graphs of maximum degree five. 
We prove that for \( k\geq 3 \), \textsc{\( k \)-Acyclic Colourability} is NP-complete for bipartite graphs of maximum degree \( k+1 \), thereby generalising the NP-completeness result of Ochem, and adding bipartiteness to the NP-completeness result of Mondal et al. 
In contrast, \textsc{\( k \)-Acyclic Colourability} is polynomial-time solvable for graphs of maximum degree at most \( 0.38\, k^{\,3/4} \). 
Hence, for \( k\geq 3 \), the least integer \( d \) such that \textsc{\( k \)-Acyclic Colourability} in graphs of maximum degree \( d \) is NP-complete, denoted by \( L_a^{(k)} \), satisfies \( 0.38\, k^{\,3/4}<L_a^{(k)}\leq k+1 \). 
We prove that for \( k\geq 4 \), \textsc{\( k \)-Acyclic Colourability} in \( d \)-regular graphs is NP-complete if and only if \( L_a^{(k)}\leq d\leq 2k-3 \). 
We also show that it is coNP-hard to check whether an input graph \( G \) admits a unique \( k \)-acyclic colouring up to colour swaps (resp.\ up to colour swaps and automorphisms). 
\end{abstract}

\section{Introduction and Definitions}\label{sec:acyclic colouring intro}
Acyclic colouring is a variant of graph colouring introduced by Gr\"{u}nbaum~\cite{grunbaum} and widely studied for the class of planar graphs \cite{grunbaum,borodin1979} and its superclasses, such as 1-planar graphs \cite{borodin2001,yang2020} and graphs embeddable on surfaces \cite{albertson_berman,alon,kawarabayashi_mohar}. 
An acyclic colouring of a graph \( G \) is a (vertex) colouring of \( G \) without bicoloured cycles. 
It is used in the estimation of sparse Hessian matrices \cite{gebremedhin2005}. 
The algorithmic complexity of acyclic colouring is studied in various graph classes \cite{angelini_frati,alon_zaks,wang_zhang,shi_wang,lyons,fertin2003,bokPreprint}. 
Brause et al.~\cite{brause} investigated the complexity of 3-acyclic colouring with respect to the graph diameter. 
For \( k\geq 3 \), we study the complexity of \( k \)-acyclic colouring with respect to the maximum degree of the graph focusing on graphs of maximum degree \( d \) and \( d \)-regular graphs. 
Our interest is in the values of \( d \) for which the complexity of \( k \)-acyclic colouring in graphs of maximum degree \( d \) (resp.\  \( d \)-regular graphs) differ drastically from that in graphs of maximum degree \( d-1 \) (resp.\  \( (d-1) \)-regular graphs); we call such values of \( d \) as \emph{hardness transitions} of \( k \)-acyclic colouring (with respect to the maximum degree) in the class of graphs of maximum degree \( d \) (resp.\  \( d \)-regular graphs); see Section~\ref{sec:intro hardness transitions} for details. 
We also prove computational hardness results on unique acyclic colouring (see Section~\ref{sec:intro unique solution problems}). 

The paper is organised as follows. 
See Subsection~\ref{sec:def} for basic definitions. 
Subsection~\ref{sec:acyclic coloring known} discusses known results on the algorithmic complexity of acyclic colouring, and Subsections~\ref{sec:intro hardness transitions} and \ref{sec:intro unique solution problems} introduce the conventions and notations we use related to hardness transitions and unique solution problems, respectively. 
Subsection~\ref{sec:our results} lists the major contributions of this paper. 
Section~\ref{sec:hardness transitions of acyclic colouring} presents our results on the hardness transitions of acyclic colouring (with respect to the maximum degree). 
Section~\ref{sec:unique soln} discusses our results on unique acyclic colouring. 
We conclude with Section~\ref{sec:acyclic open} on open problems.

\iftoggle{forThesis}
{}{\subsection{Basic Definitions}\label{sec:def}
All graphs considered in this paper are finite, simple and undirected. 
We follow West~\cite{west} for graph theory terminology and notation. 
When the graph is clear from the context, we denote the number of edges of the graph by \( m \) and the number of vertices by \( n \). 
For a graph \( G \), we denote the maximum degree of \( G \) by \( \Delta(G) \). 
For a subset \( S \) of the vertex set of \( G \), the \emph{subgraph of \( G \) induced by \( S \)} is denoted by \( G[S] \). 
The \emph{girth} of a graph with a cycle is the length of its shortest cycle. 
A graph \( G \) is \emph{\( 2 \)-degenerate} if there exists a left-to-right ordering of its vertices such that every vertex has at most two neighbours to its left.
The \emph{maximum average degree} \( \text{mad}(G) \) of a graph \( G \) is the maximum over average degrees of subgraphs of~\( G \). That is, \( \text{mad}(G)=\max \{2|E(H)|/|V(H)|\colon H \text{ is a subgraph of } G\} \). 
The treewidth of \( G \) is denoted as \( \text{tw}(G) \). 
A 3-regular graph is also called a \emph{cubic graph}, and a graph of maximum degree~3 is called a \emph{subcubic graph}. 

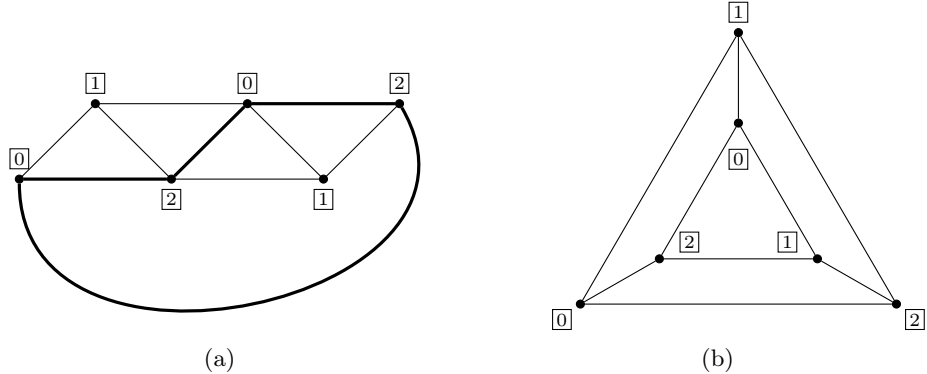
\begin{figure}[hbt]
\centering
\begin{subfigure}[b]{0.55\textwidth}
\centering
\iftoggle{iiitdmV}
{\scalebox{0.8}
}{\scalebox{1.0}
}{
\begin{tikzpicture}
\clip (-0.2,-2) rectangle (5.5,1.5);
\draw (0,0) node (00)[dot][label={[vcolour]0}]{} --++(1,1) node (01)[dot][label={[vcolour]1}]{} --++(1,-1) node (02)[dot][label={[vcolour]below:2}]{} --++(1,1) node (10)[dot][label={[vcolour]0}]{} --++(1,-1) node (11)[dot][label={[vcolour]below:1}]{} --++(1,1) node (12)[dot][label={[vcolour]2}]{};
\draw
(00)--(02)--(11)
(01)--(10)--(12);
\draw [very thick] (00) to[out=-90,in=-60,distance=3cm,looseness=0.8] (12);
\draw [very thick] (00)--(02)--(10)--(12);
\end{tikzpicture}
}
\caption{}
\end{subfigure}\begin{subfigure}[b]{0.3\textwidth}
\centering
\iftoggle{iiitdmV}
{\scalebox{0.9}
}{\scalebox{1.0}
}{
\begin{tikzpicture}[scale=1.2]
\path
(  90:1) node (u0)[dot][label={[vcolour,label distance=8pt]-90:0}]{}
( -30:1) node (u1)[dot][label={[vcolour,label distance=7pt,yshift=-2pt]150:1}]{}
(-150:1) node (u2)[dot][label={[vcolour,label distance=7pt,yshift=-2pt]30:2}]{}
(  90:2) node (v0)[dot][label={[vcolour]90:1}]{}
( -30:2) node (v1)[dot][label={[vcolour]-30:2}]{}
(-150:2) node (v2)[dot][label={[vcolour]-150:0}]{};
\draw
(u0)--(v0)
(u1)--(v1)
(u2)--(v2);
\draw
(u0)--(u1)--(u2)--(u0)
(v0)--(v1)--(v2)--(v0);
\end{tikzpicture}
}
\caption{}
\end{subfigure}\caption[Example of a 3-acyclic colouring]{(a) a 3-colouring of a graph, which is not a 3-acyclic colouring (bicoloured cycle highlighted), and (b) a 3-acyclic colouring of the circular ladder graph \( CL_3 \) (which is not 3-star colourable).}
\label{fig:acyclic colouring eg}
\end{figure}
A \( k \)-colouring of a graph \( G \) is a function \( f \) from the vertex set of \( G \) to a set of \( k \) colours, say \( \{0,1,\dots,k-1\} \), such that \( f \) maps every pair of adjacent vertices to different colours. Let us denote the \( i \)th colour class \( f^{-1}(i) \) by \( V_i \). 
A \( k \)-acyclic colouring of \( G \) is a \( k \)-colouring \( f \) of \( G \) such that every pair of colour classes induces an acyclic subgraph (i.e., \( G[V_i\cup V_j] \) is a forest for every pair of colour classes \( V_i \) and \( V_j \)). 
See Figure~\ref{fig:acyclic colouring eg} for an example. 
The \emph{acyclic chromatic number} \( \chi_a(G) \) of a graph \( G \) is the least integer \( k \) such that \( G \) is \( k \)-acyclic colourable. 
The problem \textsc{Acyclic Colourability} takes a graph \( G \) and a positive integer \( k \) as input and asks whether \( G \) is \( k \)-acyclic colourable. 
For \( k\in \mathbb{N} \), the decision problem \textsc{\( k \)-Colourability} takes a graph \( G \) as input and asks whether \( G \) is \( k \)-colourable. 
Similarly, for \( k\in \mathbb{N} \), the problem \textsc{\( k \)\nobreakdash-Acyclic Colourability} takes a graph \( G \) as input and asks whether \( G \) is \( k \)-acyclic colourable. 
To denote the restriction of a decision problem, we write the conditions in parenthesis. 
For instance, \textsc{\( 4 \)-Acyclic Colourability}\( (\text{bipartite}, \Delta=5) \) denotes the problem \textsc{\( 4 \)-Acyclic Colourability} restricted to the class of bipartite graphs \( G \) with \( \Delta(G)=5 \). 
The Exponential Time Hypothesis (ETH) asserts that \textsc{3-Sat} cannot be solved in \( 2^{o(n)} \) time, where \( n \) is the number of variables in the \textsc{3-Sat} formula \cite{niedermeier}.

An \emph{automorphism} of a graph \( G \) is a bijective function \( \psi \colon V(G)\to V(G) \) such that \( xy\in E(G) \) if and only if \( \psi(x)\psi(y)\in E(G) \). 
We say that two colourings \( f_1 \) and \( f_2 \) of \( G \) are the same \emph{up to colour swaps} if \( f_2 \) can be obtained from \( f_1 \) by merely swapping colours (that is, there exists a permutation \( \sigma \) of colours such that \( f_2(v)=\sigma(f_1(v)) \) for every vertex \( v \) of \( G \)). We say that two colourings \( f_1 \) and \( f_2 \) of \( G \) are the same \emph{up to colour swaps and automorphisms} if there exists a permutation \( \sigma \) of colours and an automorphism \( \psi \) of \( G \) such that \( f_2(\psi(v))=\sigma(f_1(v)) \) for every vertex \( v \) of \( G \).

}

\subsection{Acyclic colouring: Literature Survey}\label{sec:acyclic coloring known}
\iftoggle{forThesis}
{ Due to the vast literature on acyclic colouring, we focus on the algorithmic complexity aspect. 
We do not claim to mention all results on its complexity here; 
for surveys, see \cite[Section~9]{borodin2013} and \cite[Section~3.11]{jensen_toft}. 
} { Gr\"{u}nbaum~\cite{grunbaum} conjectured that every planar graph is 5-acyclic colourable. 
This conjecture and its proof by Borodin~\cite{borodin1979} attracted the attention of many, and as a result, acyclic colouring is widely studied for the class of planar graphs \cite[Section~9]{borodin2013} and its superclasses including 1-planar graphs \cite{borodin2001,yang2020} and graphs embeddable on surfaces \cite{albertson_berman,alon,kawarabayashi_mohar}. 
It is also studied for other graph classes such as regular graphs \cite{angelini_frati}, line graphs \cite{alon_zaks,wang_zhang,shi_wang}, \( H \)-free graphs for fixed \( H \)~\cite{bokPreprint}, co-bipartite graphs~\cite{bokPreprint}\footnote{See also \url{http://www.math.tau.ac.il/~nogaa/PDFS/multitaskadd.pdf}}, cographs \cite{lyons}, and grid graphs \cite{fertin2003}. 
Besides, acyclic colouring is studied for classes of graphs obtained by imposing bounds on parameters such as maximum degree, girth, maximum average degree and degeneracy \cite{alon1991,kostochka_stocker,wang2022,cai,muthu2007,borodin2010,basavaraju_chandran}. 
Gr\"{u}nbaum~\cite{grunbaum} proved that every graph of maximum degree 3 is 4-acyclic colourable. 
Burstein~\cite{burstein} proved that every graph of maximum degree 4 is 5-acyclic colourable. 
Due to the vast literature, we focus on the algorithmic complexity aspect of acyclic colouring. 
For surveys, see \cite[Section~3.11]{jensen_toft} and \cite[Section~9]{borodin2013}.

Fertin et al.~\cite{fertin2003} proved that \( \chi_a(G)>1+\frac{m}{n} \) for every non-empty graph \( G \). 
Alon et al.~\cite{alon1991} proved that \( \chi_a(G)\leq 50\, d^{4/3} \) for every graph \( G \) of maximum degree \( d \). 
This bound was improved by Ndreca et al.~\cite{ndreca} to \( 6.59\,d^{4/3}+3.3\,d \), and by Sereni and Volec~\cite{sereni_volec} to \( 2.835\,d^{4/3}+d \). 
See \cite{goncalves,kirousis_livieratos,alves} for related work.

}

\subsubsection*{Algorithms for Acyclic Colouring}
Every \( k \)-colouring of a chordal graph \( G \) is a \( k \)-acyclic colouring of \( G \), and hence the acyclic chromatic number of \( G \) can be computed in polynomial time \cite{gebremedhin2009}. 
Lyons~\cite{lyons} proved that the acyclic chromatic number of a cograph can be computed in linear time (i.e., \( O(m+n) \) time). 
Linhares-Sales et al.~\cite{linhares-sales} designed a linear-time algorithm to compute the acyclic chromatic number for two superclasses of cographs called \( P_4 \)-tidy graphs and \( (q,q-4) \)-graphs (for each fixed \( q \)). 
Skulrattanakulchai~\cite{skulrattanakulchai} designed a linear-time algorithm to acyclic colour a graph of maximum degree 3 with 4 colours. 
Cheng et al.~\cite{cheng} designed a polynomial-time algorithm to obtain optimal acyclic colourings of claw-free graphs of maximum degree 3.

\subsubsection*{Hardness Results on Acyclic Colouring}Kostochka~\cite{kostochka} and Coleman and Cai~\cite{coleman_cai} independently produced constructions that proved the NP-completeness of \textsc{Acyclic Colourability}. 
We restate the construction of Coleman and Cai~\cite{coleman_cai} in Section~\ref{sec:fpt acyclic} (see Construction~\ref{make:acyclic and star 2-degenerate bipartite}). 
From Construction~\ref{make:acyclic and star 2-degenerate bipartite}, it follows that (i)~for all \( k\geq 3 \), \textsc{\( k \)\nobreakdash-Acyclic Colourability} is NP-complete for 2\nobreakdash-degenerate bipartite graphs and (ii)~\textsc{\( 3 \)-Acyclic Colourability} is NP-complete for 2\nobreakdash-degenerate planar bipartite graphs. 
Borrowing ideas from Construction~\ref{make:acyclic and star 2-degenerate bipartite}, Gebremedhin et al.~\cite{gebremedhin2007} proved that for all \( \epsilon>0 \), it is NP-hard to approximate the acyclic chromatic number of a 2-degenerate bipartite graph within \( n^{\frac{1}{3}-\epsilon} \). 
In contrast, every 2-degenerate graph admits an acyclic colouring with \( n^{\frac{1}{2}} \) colours \cite[Theorem~6.2]{karpas} (note that every unique superior colouring is an acyclic colouring \cite{karpas}); hence, the acyclic chromatic number of a 2-degenerate graph is approximable within \( n^{\frac{1}{2}} \). 
Bok et al.~\cite{bok} studied the complexity of \textsc{Acyclic Colourability} and \textsc{\( k \)-Acyclic Colourability} in \( H \)-free graphs. 
They also proved that for \( k\geq 3 \), \textsc{\( k \)-Acyclic Colourability} is NP-complete for line graphs and graphs of arbitrarily large girth.
 
As mentioned, \textsc{3-Acyclic Colourability} is NP-complete for 2-degenerate planar bipartite graphs~\cite{coleman_cai}. 
Ochem~\cite{ochem} proved that the problem remains NP-complete when further restricted to graphs of maximum degree four. 
Alon and Zaks~\cite{alon_zaks} proved that \textsc{3-Acyclic Colourability} is NP-complete for line graphs of subcubic graphs. 
Mondal et al.~\cite{mondal} proved that \textsc{4-Acyclic Colourability} is NP-complete for graphs of maximum degree five. 
The problem \textsc{4-Acyclic Colourability} is also NP-complete for planar graphs of maximum degree seven \cite{mondal} and 2-degenerate planar bipartite graphs of maximum degree eight \cite{ochem}.

Consider a fixed \( k\geq 3 \). 
Emden-Weinert et al.~\cite{emden-weinert} proved that \textsc{\( k \)-Colourability} is NP-complete for graphs of maximum degree \( k-1+\raisebox{1.5pt}{\big\lceil}\sqrt{k}\raisebox{1.5pt}{\big\rceil} \). 
Since the Construction of Coleman and Cai~\cite{coleman_cai} (restated below as Construction~\ref{make:acyclic and star 2-degenerate bipartite}) establishes a reduction from \textsc{\( k \)-Colourability} in graphs of maximum degree \( k-1+\raisebox{1.5pt}{\big\lceil}\sqrt{k}\raisebox{1.5pt}{\big\rceil} \) to \textsc{\( k \)-Acyclic Colourability} in graphs of maximum degree \( k(k-1+\raisebox{1.5pt}{\big\lceil}\sqrt{k}\raisebox{1.5pt}{\big\rceil}) \), the latter problem is NP-complete as well. 
\begin{observation}\label{obs: almost k sq star acyclic npc}
For \( k\geq 3 \), \textsc{\( k \)-Acyclic Colourability} is NP-complete for graphs of maximum degree \( k(k-1+\raisebox{1.5pt}{\big\lceil}\sqrt{k}\raisebox{1.5pt}{\big\rceil}) \). 
\qed
\end{observation}

\subsubsection*{Fixed-Parameter Tractability of Acyclic Colouring}\label{sec:fpt acyclic}
For every positive integer \( k \), \textsc{\( k \)-Acyclic Colourability} can be expressed in Monadic Second Order (MSO) logic without edge set quantifiers (i.e., MSO\( _1 \)) as follows.
\[
 \exists V_0\ \exists V_1\ \dots \exists V_{k-1}\ k\text{-Colouring}(V_0,V_1,\dots,V_{k-1})\ \wedge
\]
\iftoggle{iiitdmV}
{\vspace*{-1.00cm}
}{}\[ \neg \text{ContainCycle}(V_0,V_1)\ \wedge \ \dots \wedge\ \neg \text{ContainCycle}(V_{k-2},V_{k-1})
\]
Here, ContainCycle\( (V_i,V_j) \) stands for \( \exists W\ \forall w\in W\ (w\in V_i \vee w\in V_j)\ \wedge \ (\exists w'\ \exists w''\ \allowbreak\mbox{\( w'\in W \)}\wedge \mbox{\( w''\in W \)} \wedge w'\neq w''\wedge \text{adj}(w,w')\wedge \text{adj}(w,w''))  \) (that is, there exists a set \( W\subseteq V_i\cup V_j \) such that every vertex \( w \) in \( W \) has two neighbours \( w' \) and \( w'' \) in \( W \)). 
Also, \( k \)-Colouring(\( V_0,V_1,\dots,V_{k-1} \)) denotes the MSO formula that says \( \{V_0,V_1,\dots,V_{k-1}\} \) is a partition of the vertex set of the graph into independent sets (see Section~1 in the supplementary material for the MSO formula).

Therefore, for each \( k \), the problem \textsc{\( k \)-Acyclic Colourability} admits FPT algorithms with parameter either treewidth or cliquewidth by Courcelle's theorem \cite{borie,courcelle}.
Ganian and Hlin\v{e}n\'{y}~\cite{ganian_hlineny} obtained an FPT algorithm for \textsc{\( k \)-Acyclic Colourability} with parameter rankwidth.

On the negative side, the following construction shows that unless coNP \( \subseteq \) NP/poly, \textsc{\( k \)-Acyclic Colourability} parameterized by treewidth does not admit a polynomial kernel, provided \( k\geq 3 \). 
\begin{construct}[Coleman and Cai~\cite{coleman_cai}]\label{make:acyclic and star 2-degenerate bipartite}
\emph{Parameter:} An integer \( k\geq 3 \).\\
\emph{Input:} A graph \( G \).\\
\emph{Output:} A 2-degenerate bipartite graph \( G' \).\\
\emph{Guarantee 1} \cite{coleman_cai}\,: \( G \) is \( k \)-colourable if and only if \( G' \) is \( k \)-acyclic colourable.\\
\emph{Guarantee 2:} \( \text{tw}(G')\leq \text{tw}(G)+1 \).\\ \emph{Steps:}\\
Replace each edge \( e=uv \) of \( G \) by a copy of the complete bipartite graph \( K_{2,k} \) with parts \( \{u,v\} \) and \( \{e_1,e_2,\dots,e_k\} \), where \( e_1,e_2,\dots,e_k \) are newly introduced vertices. 
To obtain a 2-degenerate ordering of \( V(G') \), list members of \( V(G) \) followed by the new vertices \( e_i \). 
\end{construct}
Guarantee~2 is easy to prove, especially using the game-theoretic definition of treewidth (see Section~2 in the supplementary material for details). 
By Guarantee~2, the transformation from \textsc{\( k \)-Colourability} to \textsc{\( k \)-Acyclic Colourability} established by Construction~\ref{make:acyclic and star 2-degenerate bipartite} is a Polynomial Parameter Transformation (PPT) \cite{fomin2019} when both problems are parameterized by treewidth. 
Thus, we have the following theorem since \textsc{\( k \)-Colourability} with parameter treewidth does not admit a polynomial kernel for \( k\geq 3 \). 

\begin{theorem}\label{thm:acyclic intractable w.r.t. treewidth}
For all \( k\geq 3 \), \textsc{\( k \)-Acyclic Colourability} parameterized by treewidth does not admit a polynomial kernel unless coNP \( \subseteq \) NP/poly. 
\qed
\end{theorem}

\subsection{Hardness Transitions}\label{sec:intro hardness transitions}
Analysing the boundary between easy (i.e., polynomial-time solvable) and hard (e.g., NP-complete) problems is a common theme in complexity theory \cite{garey_johnson}. 
Studying the change in the complexity of a problem in response to a change in a single parameter falls in this category. 
Brause et al.~\cite{brause} studied the complexity of \textsc{3-Acyclic Colourability} with the diameter of the graph as the parameter. 
For \( k\geq 3 \), we study the complexity of \( k \)-acyclic colouring with the maximum degree of the graph as the parameter. 
Recall that we write the conditions in parenthesis to denote the restriction of a decision problem; 
e.g.: \textsc{\( 4 \)-Acyclic Colourability}\( (\text{bipartite}, \Delta=5) \) denotes the problem \textsc{\( 4 \)-Acyclic Colourability} restricted to the class of bipartite graphs \( G \) with \( \Delta(G)=5 \).
We assume \mbox{P \( \neq \) NP} throughout this paper; thus, NP is partitioned into three classes: P, NPC and NPI~\cite{ladner}.
We emphasise that our interest is in the classification of NP-problems with respect to the P vs.\ NPC vs.\ NPI trichotomy: that is, the complexity classes dealt with in this paper are only P, NPC and NPI.

A decision problem \( \Pi \) in NP has a \emph{hardness transition} with respect to a discrete parameter \( d \) at a point \( d=x \) if \( \Pi(d=x) \) and \( \Pi(d=x-1) \) belong to different complexity classes among P, NPC and NPI (e.g.: \( \Pi(d=x)\in \)\,NPC whereas \( \Pi(d=x-1)\in \)\,P; see \cite{mikero} for a discussion). 
For example, \textsc{3-Colourability} of a graph of maximum degree \( d \) is polynomial-time solvable for \( d=3 \) (due to Brook's theorem) and NP-complete for \( d=4 \) \cite{garey_johnson}. 
That is, \textsc{3-Colourability}\( (\Delta=3)\in \) P and \textsc{3-Colourability}\( (\Delta=4)\in \) NPC. 
Hence, \textsc{3-Colourability} has a hardness transition with respect to the maximum degree \( d \) at the point \( d=4 \). 
Note that each hardness transition presumably deals with the P vs.\ NPC boundary since no `natural' problem is known to be NP-intermediate~\cite{arora_barak}.

The number of hardness transitions depends on the problem as well as the parameter under consideration. 
Interestingly, a decision problem can have infinitely many hardness transitions. 
Cseh and Kavitha~\cite{cseh_kavitha} proved that the popular matching problem on complete graph \( K_n \) is in P for odd \( n \) whereas it is NP-complete for even \( n \). 
Therefore, the popular matching problem on complete graph with respect to the number of vertices \( n \) has infinitely many hardness transitions. 

Let us consider the complexity of \( k \)-colouring and \( k \)-acyclic colouring in bounded degree graphs for fixed \( k\geq 3 \). 
Emden-Weinert et al.~\cite{emden-weinert} proved that \textsc{\( k \)-Colourability} is NP-complete for graphs of maximum degree \( k-1+\raisebox{1.5pt}{\big\lceil}\sqrt{k}\raisebox{1.5pt}{\big\rceil} \). 
By Observation~\ref{obs: almost k sq star acyclic npc}, \textsc{\( k \)-Acyclic Colourability} is NP-complete for graphs of maximum degree \( k(k-1+\raisebox{1.5pt}{\big\lceil}\sqrt{k}\raisebox{1.5pt}{\big\rceil}) \). 
Hence, \textsc{\( k \)\nobreakdash-Colourability} (resp.\ \textsc{\( k \)\nobreakdash-Acyclic Colourability}) in graphs of maximum degree \( d \) is NP-complete when \( d \) is sufficiently large. 
Observe that if \textsc{\( k \)\nobreakdash-Colourability} is NP-complete for graphs of maximum degree \( d \), then it is NP-complete for graphs of maximum degree \( d+1 \) (to produce a reduction, it suffices to add a disjoint copy of \( K_{1,d+1} \)). 
This suggests the following problem. 
\begin{problem}\label{prob:what is Lk}
For \( k\geq 3 \), what is the least integer \( d \) such that \textsc{\( k \)-Colourability} is NP-complete for graphs of maximum degree \( d \)?
\end{problem}
\noindent Observe that Problem~\ref{prob:what is Lk} deals with locating a point of hardness transition. 
By the same argument as in \textsc{\( k \)-Colourability}, if \textsc{\( k \)-Acyclic Colourability} is NP-complete for graphs of maximum degree~\( d \), then \textsc{\( k \)-Acyclic Colourability} is NP-complete for graphs of maximum degree \( d+1 \). 
Therefore, for each \( k\geq 3 \), there exists a unique integer \( d^* \) such that \textsc{\( k \)\nobreakdash-Colourability} (resp.\ \textsc{\( k \)\nobreakdash-Acyclic Colourability}) in graphs of maximum degree \( d \) is NP-complete if and only if \( d\geq d^* \). 
Thus, one can ask the counterpart of Problem~\ref{prob:what is Lk} for acyclic colouring. 
Let \( L^{(k)} \) and \( L_a^{(k)} \) denote the answers to Problem~\ref{prob:what is Lk} and its counterpart for acyclic colouring; that is, \( L^{(k)} \) (resp.\  \( L_a^{(k)} \)) is the least integer \( d \) such that \textsc{\( k \)\nobreakdash-Colourability} (resp.\ \textsc{\( k \)\nobreakdash-Acyclic Colourability}) is NP-complete for graphs of maximum degree \( d \).

Due to Brook's theorem, \textsc{\( k \)-Colourability} is polynomial-time solvable for graphs of maximum degree \( k \), and thus \( L^{(k)}\geq k+1 \). 
For \( k\geq 3 \), \textsc{\( k \)\nobreakdash-Colourability} is NP-complete for graphs of maximum degree \( k-1+\raisebox{1.5pt}{\big\lceil}\sqrt{k}\raisebox{1.5pt}{\big\rceil} \) \cite{emden-weinert}, and thus \( k+1\leq L^{(k)}\leq k-1+\raisebox{1.5pt}{\big\lceil}\sqrt{k}\raisebox{1.5pt}{\big\rceil} \). 
Hence, \( L^{(3)}=4 \), \( L^{(4)}=5 \), \( 6\leq L^{(5)}\leq 7 \), and so on. 
For sufficiently large \( k \) and \( d<k-1+\raisebox{1.5pt}{\big\lceil}\sqrt{k}\raisebox{1.5pt}{\big\rceil} \), the problem \textsc{\( k \)\nobreakdash-Colourability} is in P for graphs of maximum degree \( d \) \cite[Theorem~43]{molloy_reed}. 
Therefore, \( L^{(k)}=k-1+\raisebox{1.5pt}{\big\lceil}\sqrt{k}\raisebox{1.5pt}{\big\rceil} \) for sufficiently large \( k \). 
Yet, the exact value of \( L^{(k)} \) is unknown for small values of \( k \) such as \( k=5 \) eventhough we know that \( L^{(5)}\in \{6,7\} \) (the complexity of \textsc{5\nobreakdash-Colourability} in graphs of maximum degree~6 is open \cite{paulusma}).

\subsection{Unique Solution Problems}\label{sec:intro unique solution problems}
For a decision problem \( X \), the \emph{unique solution problem} associated with \( X \) takes the same input as \( X \) and asks whether problem \( X \) with the given input has exactly one solution. 
For instance, the unique solution problem associated with \textsc{Sat} takes a boolean formula \( B \) as input and asks whether \( B \) has exactly one satisfying truth assignment. 
For some decision problems, it is more natural to consider an equivalence relation and ask whether there is only one equivalence class in it. 
For example, for a graph \( G \), it is customary to call a \( k \)-colouring \( f \) of \( G \) as the unique \( k \)-colouring of \( G \) if each \( k \)-colouring of \( G \) can be obtained from \( f \) by merely swapping colours. 
That is, instead of the number of \( k \)-colourings of \( G \), we are interested in the number of equivalence classes under \( \mathcal{R}_{\text{swap}}(G,k) \) where \( \mathcal{R}_{\text{swap}}(G,k) \) is an equivalence relation on the set of \( k \)-colourings of \( G \) defined as \( (f_1,f_2)\in\mathcal{R}_{\text{swap}}(G,k) \) if \( f_1 \) and \( f_2 \) are the same up to colour swaps. 
Similarly, we define an equivalence relation \( \mathcal{R}_{\text{swap+auto}}(G,k) \) on the set of \( k \)-colourings of \( G \) as \( (f_1,f_2)\in \mathcal{R}_{\text{swap+auto}}(G,k) \) if \( f_1 \) and \( f_2 \) are the same up to colour swaps and automorphisms. 
Thus, we have two unique solution problems associated with \( k \)-colouring.\\

\noindent \textsc{Unique \( k \)-Colouring} [\( \mathcal{R}_{\text{swap}} \)]\\
\emph{Instance:} A graph \( G \).\\
\emph{Question:} Is the number of equivalence classes under \( \mathcal{R}_{\text{swap}}(G,k) \) exactly one?\\
\phantom{\emph{Question:}} (i.e., Is the number of \( k \)-colourings of \( G \) up to colour swaps exactly one?)\\

\noindent \textsc{Unique \( k \)-Colouring} [\( \mathcal{R}_{\text{swap+auto}} \)]\\
\emph{Instance:} A graph \( G \).\\
\emph{Question:} Is the number of equivalence classes under \( \mathcal{R}_{\text{swap+auto}}(G,k) \) exactly one?\\
\iftoggle{forThesis}
{\phantom{\emph{Question:}} (Is the number of \( k \)-colourings of \( G \) up to colour swaps and automorphisms\\
\phantom{\emph{Question:}} exactly one?)\\
}{\phantom{\emph{Question:}} (Is the number of \( k \)-colourings of \( G \) up to colour swaps and automorphisms exactly one?)\\
}

An \emph{another solution problem} is closely related to the unique solution problem. The another solution problem associated with \textsc{Sat} takes a boolean formula \( B \) and a satisfying truth assignment of \( B \) as input and asks whether \( B \) has another satisfying truth assignment. Similar to the unique solution problem, there are two another solution problems associated with \( k \)-colouring, namely \textsc{Another \( k \)-Colouring} [\( \mathcal{R}_{\text{swap}} \)] and \textsc{Another \( k \)-Colouring} [\( \mathcal{R}_{\text{swap+auto}} \)]. The former is defined below, and the latter is defined likewise.\\

\noindent \textsc{Another \( k \)-Colouring} [\( \mathcal{R}_{\text{swap}} \)]\\
\emph{Instance:} A graph \( G \), and a \( k \)-colouring \( f \) of \( G \).\\
\emph{Question:} Is there another \( k \)-colouring of \( G \) up to colour swaps?\\\phantom{\emph{Question:}} (i.e., a \( k \)-colouring \( f^* \) of \( G \) such that \( (f,f^*)\notin \mathcal{R}_{\text{swap}}(G,k) \))\\

Dailey~\cite{dailey} proved that for all \( k\geq 3 \), \textsc{Another \( k \)-Colouring}\,[\( \mathcal{R}_{\text{swap}} \)] is NP-hard. 
Hence, given a \( k \)-colourable graph \( G \), it is coNP-hard to check whether \( G \) is uniquely \( k \)-colourable up to colour swaps. 
That is, \textsc{Unique \( k \)-Colouring}\,[\( \mathcal{R}_{\text{swap}} \)] is coNP-hard even when restricted to the class of \( k \)-colourable graphs.
Dailey~\cite{dailey} produced a reduction from \textsc{\( k \)-Colourability} to \textsc{Another \( k \)-Colouring}\,[\( \mathcal{R}_{\text{swap}} \)]. 
A close look reveals that the same construction also establishes a reduction from \textsc{\( k \)\nobreakdash-Colourability} to \textsc{Another \( k \)-Colouring}\,[\( \mathcal{R}_{\text{swap+auto}} \)]. 
Thus, for all \( k\geq 3 \), problems \textsc{Another \( k \)-Colouring}\,[\( \mathcal{R}_{\text{swap}} \)] and \textsc{Another \( k \)-Colouring}\,[\( \mathcal{R}_{\text{swap+auto}} \)] are NP-complete. 
As a result, the problems \textsc{Unique \( k \)-Colouring}\,[\( \mathcal{R}_{\text{swap}} \)] and \textsc{Unique \( k \)-Colouring}\,[\( \mathcal{R}_{\text{swap+auto}} \)] are coNP-hard for \( k\geq 3 \).

It is easy to observe that Dailey's construction provides reductions from \textsc{\( k \)-Coloura\allowbreak bility}(\( \Delta=k-1+\raisebox{1.5pt}{\big\lceil}\sqrt{k}\raisebox{1.5pt}{\big\rceil} \)) to \textsc{Another \( k \)\nobreakdash-Colouring}\,[\( \mathcal{R}_{\text{swap}} \)]\,(\( \Delta=(k-1)(k-1+\raisebox{1.5pt}{\big\lceil}\sqrt{k}\raisebox{1.5pt}{\big\rceil}) \)) and to \textsc{Another \( k \)\nobreakdash-Colouring}\,[\( \mathcal{R}_{\text{swap+auto}} \)]\,(\( \Delta=(k-1)(k-1+\raisebox{1.5pt}{\big\lceil}\sqrt{k}\raisebox{1.5pt}{\big\rceil}) \)). 
In particular, \pagetarget{lnk:another 3-colouring}{\textsc{Another 3-Colouring}\,[\( \mathcal{R}_{\text{swap}} \)] and \textsc{Another 3-Colouring}\,[\( \mathcal{R}_{\text{swap+auto}} \)] are NP-complete for graphs of maximum degree~8}.

We establish hardness results for unique solution problems and another solution problems associated with acyclic colouring. 
These problems are defined in the obvious way. 
For instance, \textsc{Unique \( k \)\nobreakdash-Acyclic Colouring}\,[\( \mathcal{R}_{\text{swap}} \)] is defined like \textsc{Unique \( k \)-Colouring}\,[\( \mathcal{R}_{\text{swap}} \)]. 

\subsection{Our Results}\label{sec:our results}
Recall that for \( k\geq 3 \), \( L_a^{(k)} \) is the the least integer \( d \) such that \textsc{\( k \)-Acyclic Colourability} is NP-complete for graphs of maximum degree \( d \).
For \( k\geq 3 \), \textsc{\( k \)\nobreakdash-Acyclic Colourability} is NP-complete for graphs of maximum degree \( k(k-1+\raisebox{1.5pt}{\big\lceil}\sqrt{k}\raisebox{1.5pt}{\big\rceil}) \) by Observation~\ref{obs: almost k sq star acyclic npc}, and thus \( L_a^{(k)}\leq k(k-1+\raisebox{1.5pt}{\big\lceil}\sqrt{k}\raisebox{1.5pt}{\big\rceil}) \). 
Ochem~\cite{ochem} proved that \textsc{3\nobreakdash-Acyclic Colourability} is NP-complete for bipartite graphs of maximum degree~4; thus, \( L_a^{(3)}\leq 4 \). 
We generalise this result: for \mbox{\( k\geq 3 \)}, \textsc{\( k \)\nobreakdash-Acyclic Colourability} is NP-complete for bipartite graphs of maximum degree \( k+1 \) (and thus \( L_a^{(k)}\leq k+1 \)). 
Hence, \textsc{4-Acyclic Colourability} is NP-complete not only for graphs of maximum degree~5 \cite{mondal}, but also for bipartite graphs of maximum degree~5. 

For each \( k\geq 3 \), we prove the following.
\begin{enumerate}\item \textsc{\( k \)\nobreakdash-Acyclic Colourability} is NP-complete for bipartite graphs of maximum degree \( k+1 \), and the problem does not admit a \( 2^{o(n)} \)-time algorithm unless ETH fails.
\item \( 0.38\, k^{\,3/4}<L_a^{(k)}\leq k+1 \). 
\item Provided \( k\neq 3 \), \textsc{\( k \)-Acyclic Colourability} is NP-complete for \( d \)-regular graphs if and only if \( L_a^{(k)}\leq d\leq 2k-3 \). 
\item It is coNP-hard to check whether an input graph \( G \) admits a unique \( k \)-acyclic colouring up to colour swaps (resp.\ up to colour swaps and automorphisms). 
\end{enumerate}

\section{Hardness Transitions of Acyclic Colouring}\label{sec:hardness transitions of acyclic colouring}
In this section, we discuss hardness transitions of acyclic colouring with respect to the maximum degree. 
First, we show that (i)~for \( k\geq 3 \), \textsc{\( k \)-Acyclic Colourability} is NP-complete for graphs of maximum degree \( k+1 \), and (ii)~for \( k\geq 4 \) and \( d\leq 2k-3 \), \textsc{\( k \)-Acyclic Colourability} in graphs of maximum degree \( d \) is NP-complete if and only if \textsc{\( k \)-Acyclic Colourability} in \( d \)-regular graphs is NP-complete. 
The consequences of these results on the value of \( L_a^{(k)} \) are discussed later in Section~\ref{sec:acyclic colouring points of hardness transition}.

Let us start with a simple result due to Fertin et al.~\cite{fertin2003}, which has direct consequences for hardness transitions. 
We present a shorter proof, adapted from \cite{isaacg}, below. 
\begin{theorem}[\cite{fertin2003}]\label{thm:lb chi_a avg deg}
\( \chi_a(G)>1+\frac{m}{n} \) for every non-empty graph \( G \). \end{theorem}
\begin{proof}[Proof \( ( \)adapted from \cite{isaacg}\,\( ) \).]
Let \( f\colon V(G)\to\{0,1,\dots,k-1\} \) be a \( k \)-acyclic colouring of \( G \). 
Recall that we denote the \( i \)th colour class \( f^{-1}(i) \) by \( V_i \) for \( 0\leq i\leq k-1 \). 
For every pair of colour classes \( V_i \) and \( V_j \), let \( m_{ij} \) denote the number of edges from \( V_i \) to \( V_j \). 
Since \( G[V_i\cup V_j] \) is a forest, we have \( m_{ij}\leq |V_i|+|V_j| \), and equality holds only when \( |V_i|=|V_j|=0 \). 
We know that some colour class is non-empty because \( G \) is a non-empty graph. 
Summing up over all \( i \) and \( j \), we get \( m<(k-1)n \) (for instance, \( |V_0| \) appears exactly \( k-1 \) times on the right side).
Hence, \( k>1+(m/n) \). Since \( f \) is an arbitrary \( k \)-acyclic colouring of \( G \), we have \( \chi_a(G)>1+(m/n) \).
\end{proof}
Since \( \frac{m}{n}=\frac{d}{2} \) for every \( d \)-regular graph \( G \), we have the following corollary. 
\begin{corollary}\label{cor:lb chi_a}
\( \chi_a(G)\geq \raisebox{1pt}{\big\lceil}\frac{d+3}{2}\raisebox{1pt}{\big\rceil} \) for every \( d \)-regular graph \( G \) (provided \( d\neq 0 \)). 
\qed
\end{corollary}
\begin{corollary}
For every non-empty graph \( G^* \), \( \chi_a(G^*)>1+\frac{\textup{mad}(G^*)}{2} \).
\end{corollary}
\begin{proof}
Let \( G \) be a subgraph of \( G^* \) such that the average degree of \( G \) is equal to the maximum average degree of \( G^* \). 
Applying Theorem~\ref{thm:lb chi_a avg deg} on \( G \) proves the corollary since \( \chi_a(G^*)\geq \chi_a(G) \) and \( |E(G)|/|V(G)|=\text{mad}(G^*)/2 \).
\end{proof}

Next, we show that the lower bound in Corollary~\ref{cor:lb chi_a} is attained by a vertex-transitive \( d \)-regular graph for each \( d \) (the graphs defined in the proof are used later in a construction). 
\begin{theorem}\label{thm:chi_a attained}
For all \( d\geq 1 \), there exists a \( d \)-regular vertex-transitive graph \( G_d \) with \( \chi_a(G_d)=\raisebox{1pt}{\big\lceil}\frac{d+3}{2}\raisebox{1pt}{\big\rceil} \).
\end{theorem}
\begin{proof}
We first prove the theorem for every odd number \( d \).
To this end, we construct a \( (2p+1) \)-regular graph \( G_{2p+1} \) with acyclic chromatic number equal to \( p+2 \) for all \( p\geq 0 \). 
The vertex set of \( G_{2p+1} \) is \( \{ (i,j)\colon i,j\in\mathbb{Z},\ 0\leq i\leq p+1,\ 0\leq j\leq p+1,\ i\neq j \} \), and \( (i,j) \) is adjacent to \( (k,\ell) \) if \( j=k \) or \( i=\ell \) (or both). 
The graph \( G_5 \) is shown in Figure~\ref{fig: G5}. 
Note that \( G_{2p+1} \) is a \( (2p+1) \)-regular graph; for instance, the vertex \( (0,1) \) in \( G_{2p+1} \) has neighbours \( (1,2),\dots,(1,p+1),(1,0),(2,0),\dots,(p+1,0) \). 
Consider the function \( f\colon V(G_{2p+1})\to \{0,1,\dots,p+1\} \) defined as \( f( (i,j) )=i \) for all \( (i,j)\in V(G_{2p+1}) \). 
Let \( V_i=f^{-1}(i) \) for \( 0\leq i\leq p+1 \). 
It is easy to verify that \( f \) is a \( (p+2) \)-acyclic colouring of \( G_{2p+1} \) (the subgraph of \( G \) induced by \( V_i\cup V_j \) is two stars attached at endvertices of the edge \( \{(i,j),(j,i)\} \); observe that \( u\in V_i \) is adjacent to \( v\in V_j \) if and only if \( u=(i,j) \) or \( v=(j,i) \)). 
Thus, \( \chi_a(G_{2p+1})\leq p+2 \). 
Thanks to Corollary~\ref{cor:lb chi_a}, we have \( \chi_a(G_{2p+1})\geq \raisebox{1pt}{\big\lceil}\frac{(2p+1)+3}{2}\raisebox{1pt}{\big\rceil}= p+2 \).
Hence, \( \chi_a(G_{2p+1})=p+2 \). 
Note that by the definition of the graph \( G_{2p+1} \), whether vertex \( (i, j) \) is adjacent to vertex \( (k,\ell) \) depends only on equality and inequality between indices \( i, j, k, \ell \). 
This ensures that \( G_{2p+1} \) is vertex-transitive. 
We include a short proof for completeness. 
To construct an automorphism \( \psi \) that maps a vertex \( (i,j) \) to another vertex \( (k,\ell) \), first choose a bijection \( h \) from 
\( \{0,1,\dots,p+1\} \)
to itself such that \( h(i)=k \) and \( h(j)=\ell \), and then define \( \psi\big((x,y)\big)=\big(h(x),h(y)\big) \) for all \( (x,y)\in V(G_{2p+1}) \).

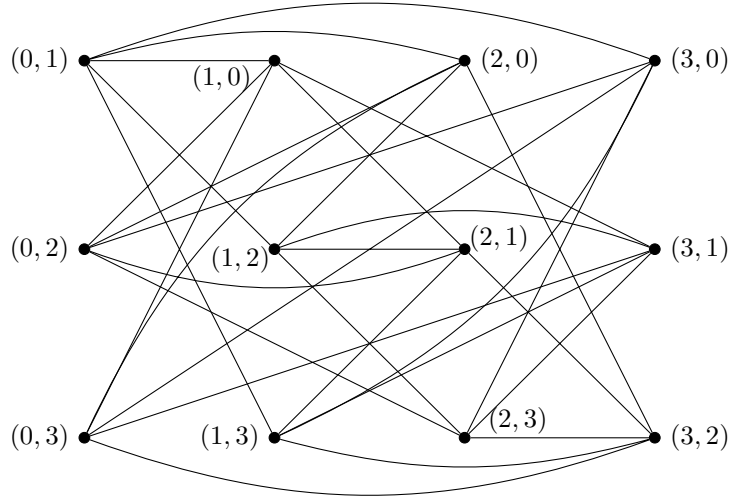
\begin{figure}[hbt]
\centering
\begin{tikzpicture}[scale=2.5]
\path (0,0) node(01)[bigdot][label=left:{\( (0,1) \)}]{}++(0,-1) node(02)[bigdot][label=left:{\( (0,2) \)}]{}++(0,-1) node(03)[bigdot][label=left:{\( (0,3) \)}]{};
\path (01)++(1,0) node(10)[bigdot][label={[xshift=-3pt,yshift=-6.5pt]left:{\( (1,0) \)}}]{}++(0,-1) node(12)[bigdot][label={[xshift=4pt,yshift=-4pt]left:{\( (1,2) \)}}]{}++(0,-1) node(13)[bigdot][label=left:{\( (1,3) \)}]{};
\path (10)++(1,0) node(20)[bigdot][label=right:{\( (2,0) \)}]{}++(0,-1) node(21)[bigdot][label={[xshift=-4pt,yshift=4pt]right:{\( (2,1) \)}}]{}++(0,-1) node(23)[bigdot][label={[xshift=3pt,yshift=6.5pt]right:{\( (2,3) \)}}]{};
\path (20)++(1,0) node(30)[bigdot][label=right:{\( (3,0) \)}]{}++(0,-1) node(31)[bigdot][label=right:{\( (3,1) \)}]{}++(0,-1) node(32)[bigdot][label=right:{\( (3,2) \)}]{};

\draw
(01)--(12)
(01)--(13)
(10)--(02)
(10)--(03)
(12)--(20)
(12)--(23)
(21)--(10)
(21)--(13)
(23)--(30)
(23)--(31)
(32)--(20)
(32)--(21);
\draw
(02)--(23)
(02)--(30)
(31)--(10)
(31)--(03);
\draw
(03) to[bend left=20] (20)
(13) to[bend right=20] (30)
(12) to[bend left=20] (31)
(02) to[bend right=20] (21)
(01) to[bend left=15] (20)
(01) to[bend left=20] (30)
(32) to[bend left=15] (13)
(32) to[bend left=20] (03);

\draw
(01)--(10)   (02)--(20)   (03)--(30)
(12)--(21)   (13)--(31)   
(23)--(32);
\end{tikzpicture}
\caption{The graph \( G_5 \).}
\label{fig: G5}
\end{figure}

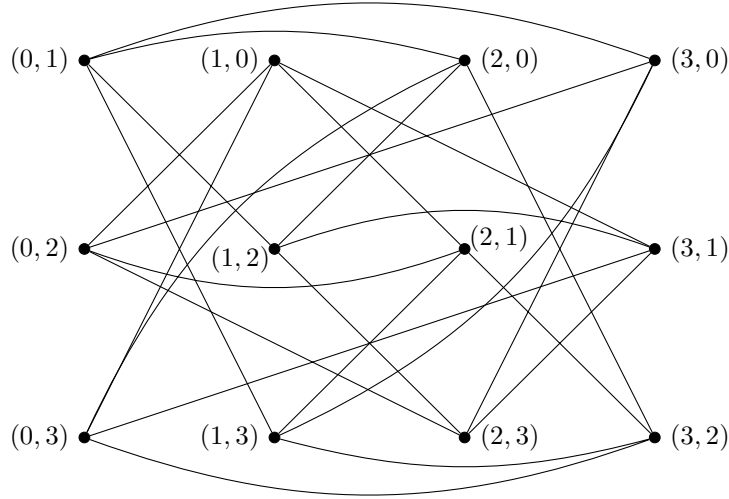
\begin{figure}[hbt]
\centering
\begin{tikzpicture}[scale=2.5]
\path (0,0) node(01)[bigdot][label=left:{\( (0,1) \)}]{}++(0,-1) node(02)[bigdot][label=left:{\( (0,2) \)}]{}++(0,-1) node(03)[bigdot][label=left:{\( (0,3) \)}]{};
\path (01)++(1,0) node(10)[bigdot][label=left:{\( (1,0) \)}]{}++(0,-1) node(12)[bigdot][label={[xshift=4pt,yshift=-4pt]left:{\( (1,2) \)}}]{}++(0,-1) node(13)[bigdot][label=left:{\( (1,3) \)}]{};
\path (10)++(1,0) node(20)[bigdot][label=right:{\( (2,0) \)}]{}++(0,-1) node(21)[bigdot][label={[xshift=-4pt,yshift=4pt]right:{\( (2,1) \)}}]{}++(0,-1) node(23)[bigdot][label=right:{\( (2,3) \)}]{};
\path (20)++(1,0) node(30)[bigdot][label=right:{\( (3,0) \)}]{}++(0,-1) node(31)[bigdot][label=right:{\( (3,1) \)}]{}++(0,-1) node(32)[bigdot][label=right:{\( (3,2) \)}]{};

\draw
(01)--(12)
(01)--(13)
(10)--(02)
(10)--(03)
(12)--(20)
(12)--(23)
(21)--(10)
(21)--(13)
(23)--(30)
(23)--(31)
(32)--(20)
(32)--(21);
\draw
(02)--(23)
(02)--(30)
(31)--(10)
(31)--(03);
\draw
(03) to[bend left=20] (20)
(13) to[bend right=20] (30)
(12) to[bend left=20] (31)
(02) to[bend right=20] (21)
(01) to[bend left=15] (20)
(01) to[bend left=20] (30)
(32) to[bend left=15] (13)
(32) to[bend left=20] (03);
\end{tikzpicture}
\caption{The graph \( G_4 \) obtained from \( G_5 \) by deleting all \( \{(i,j),(j,i)\} \) edges.}
\label{fig: G4}
\end{figure}

Next, we prove the theorem for every even number \( d \), by constructing a \( 2p \)-regular graph \( G_{2p} \) with the acyclic chromatic number equal to \( p+2 \) for all \( p\geq 1 \). 
Observe that for all \( p\geq 1 \), the set \( M \) of edges between vertex pairs \( (i,j) \) and \( (j,i) \) in \( G_{2p+1} \) is a perfect matching of \( G_{2p+1} \). 
We define \( G_{2p}=G_{2p+1}-M \). 
The graph \( G_4 \) is shown in Figure~\ref{fig: G4}. 
Clearly, \( G_{2p} \) is a \( 2p \)-regular subgraph of \( G_{2p+1} \), and thus \( \chi_a(G_{2p})\leq \chi_a(G_{2p+1})\leq p+2 \). 
By Corollary~\ref{cor:lb chi_a}, \( \chi_a(G_{2p})\geq \raisebox{1pt}{\big\lceil}\frac{2p+3}{2}\raisebox{1pt}{\big\rceil}= p+2 \), and thus \( \chi_a(G_{2p})=p+2 \). 
Moreover, the vertex set of \( G_{2p} \) is \( \{ (i,j)\colon i,j\in\mathbb{Z},\ 0\leq i\leq p+1,\ 0\leq j\leq p+1,\ i\neq j \} \), and there is an edge joining \( (i,j) \) to \( (k,\ell) \) if either \( j=k \) or \( i=\ell \) (not both). 
\iftoggle{forThesis}
{ Thus, by its definition, \( G_{2p} \) is vertex-transitive (see Theorem~\ref{thm:intro G2p} for a detailed proof). 
} { Thus, by its definition, \( G_{2p} \) is vertex-transitive (see \cite[Theorem 4]{shalu_cyriac3} for a detailed proof). 
}\end{proof}

\noindent \emph{Remark:} Note that \( G_1\subseteq G_3\subseteq G_5\subseteq \dots \) and \( G_2\subseteq G_4\subseteq G_6\subseteq \dots \) by the definition of the graphs \( G_d \). 

We are now ready to generalise the NP-completeness result of \textsc{\( 3 \)-Acyclic Colourability} in bipartite graphs of maximum degree \( 4 \) \cite{ochem}: we show that for each \( k\geq 3 \), \textsc{\( k \)-Acyclic Colourability} is NP-complete for bipartite graphs of maximum degree \( k+1 \). 
The following simple observations are employed to construct gadgets.

\begin{observation}\label{obs:K_2,k}
Let \( G \) be a graph, and let \( f \) be a \( q \)-acyclic colouring of \( G \), where \( q\leq k \). 
If two vertices \( x \) and \( y \) of \( G \) have \( k \) common neighbours, then \( f(x)\neq f(y) \).
\end{observation}
\begin{proof}
Let \( x \) and \( y \) be vertices in \( G \) that share \( k \) neighbours \( w_1,w_2,\dots,w_k \). 
Note that at least two among vertices \( w_1,w_2,\dots,w_k \) should get the same colour, say \( f(w_1)=f(w_2) \) (if not, we need \( k \) colours for vertices \( w_1,w_2,\dots,w_k \), and thus no colour is available for \( x \)). 
If \( f(x)=f(y) \), then \( (x,w_1,y,w_2) \) is a bicoloured cycle; a contradiction. 
Hence,~\( f(x)\neq f(y) \).
\end{proof}

\begin{observation}\label{obs:link in chain}
Let \( G(A\cup B,E) \cong K_{k-1,k} \) where \( A \) and \( B \) are independent sets in \( G \) with cardinality \( k-1 \) and \( k \), respectively. 
Then, \( \chi_a(G)=k \). 
Moreover, for every \( k \)-acyclic colouring \( f \) of \( G \), (i)~vertices in \( A \) should get pairwise distinct colours, and (ii)~all vertices in \( B \) should get the same colour  \( ( \)i.e., \( f(b)=f(b') \) for all \( b,b'\in B) \).
\end{observation}
\begin{proof}
Let \( f \) be a \( q \)-acyclic colouring of \( G \), where \( q=\chi_a(G) \). 
One can obtain a \( k \)\nobreakdash-acyclic colouring of \( G \) by assigning a permutation of colours \( 1,2,\dots,k-1 \) on vertices in \( A \) and colour~0 on vertices in \( B \). 
Hence, \( q\leq k \). 
Every pair of vertices from \( A \) have \( k \) common neighbours. 
Hence, vertices in \( A \) should get pairwise distinct colours by Observation~\ref{obs:K_2,k}. 
That is, \( k -1 \) distinct colours are used on vertices in \( A \). 
Let \( b\in B \). 
Since all vertices in \( A \) are adjacent to \( b \), at least one more colour, say colour~\( c_1 \), is needed (to colour \( b \)). 
This proves that \( q\geq k \), and thus \( \chi_a(G)=q=k \). 
Since \( k-1 \) distinct colours are used on vertices in \( A \), the remaining colour, namely colour~\( c_1 \), is the only available colour for each vertex in \( B \). 
Thus, all vertices in \( B \) should get colour \( c_1 \). 
\end{proof}

Note that by Observation~\ref{obs:link in chain}, the biclique \( K_{k-1,k} \) has a unique \( k \)-acyclic colouring up to colour swaps. 
\iftoggle{forThesis}
{ We construct a gadget called the \emph{chain gadget} by taking disjoint copies of the biclique \( K_{k-1,k} \) and then adding a matching (see Figure~\ref{fig:acyclic colouring chain gadget}). 
} { 

} 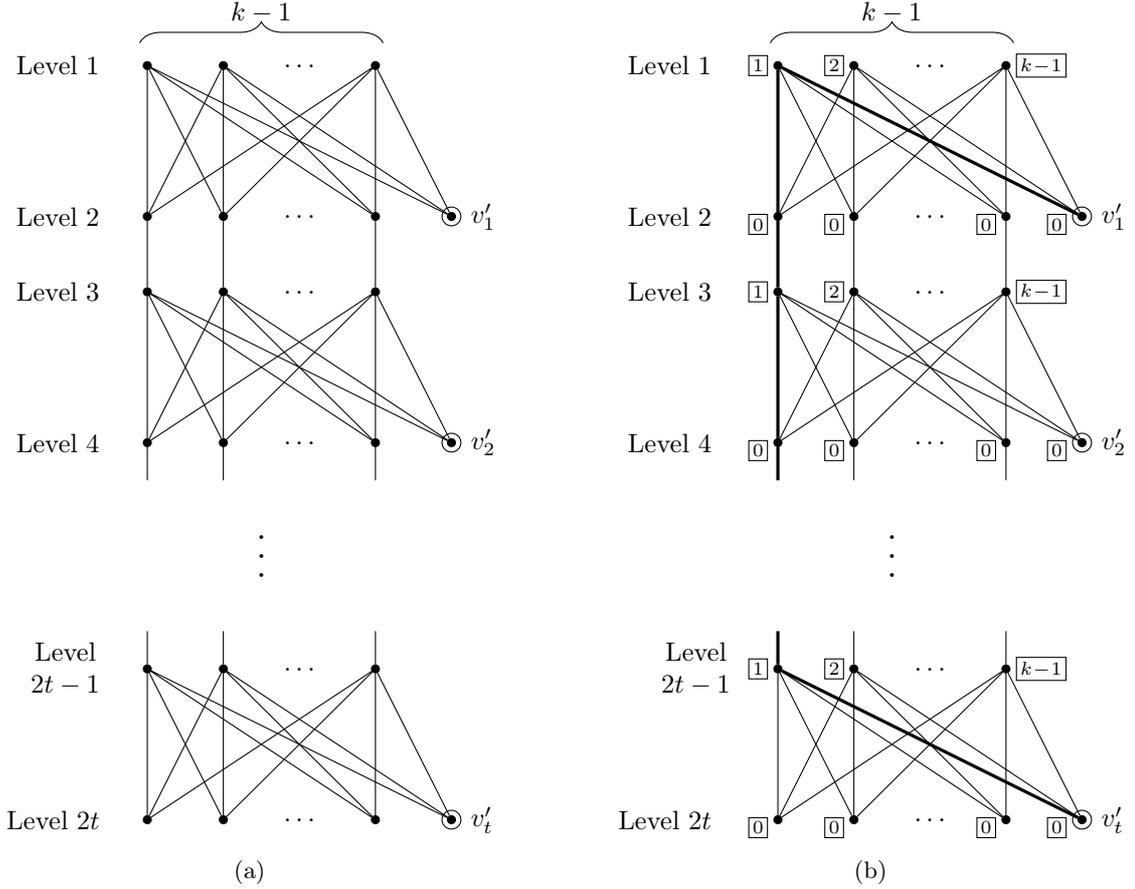
\begin{figure}[hbt]
\centering
\begin{subfigure}[b]{0.47\textwidth}
\centering
\begin{tikzpicture}
\coordinate (level1) node[left=0.5cm]{Level 1};
\path (level1)++(0,-2) coordinate (level2) node[left=0.5cm]{Level 2}++(0,-1) coordinate (level3) node[left=0.5cm]{Level 3} 
++(0,-2) coordinate (level4) node[left=0.5cm]{Level 4} 
++(0,-3) coordinate (level2n-1) node[left=0.5cm,align=center]{Level\\\( 2t-1 \)} 
++(0,-2) coordinate (level2n) node[left=0.5cm]{Level \( 2t \)};
\path (level4)-- node[pos=0.5,xshift=1.5cm,font=\LARGE,rotate=90]{.\,.\,.} (level2n-1);

\path (level1) node(11)[dot]{}++(1,0) node(12)[dot]{}++(2,0) node(1k-1)[dot]{};
\path (12)--node{\( \dots \)}(1k-1);
\path (level2) node(21)[dot]{}++(1,0) node(22)[dot]{}++(2,0) node(2k-1)[dot]{}++(1,0) node(2k)[dot]{} node[terminal][label=right:\( v'_1 \)]{};
\path (22)--node{\( \dots \)}(2k-1);
\path (level3) node(31)[dot]{}++(1,0) node(32)[dot]{}++(2,0) node(3k-1)[dot]{};
\path (32)--node{\( \dots \)}(3k-1);
\path (level4) node(41)[dot]{}++(1,0) node(42)[dot]{}++(2,0) node(4k-1)[dot]{}++(1,0) node(4k)[dot]{} node[terminal][label=right:\( v'_2 \)]{};
\path (42)--node{\( \dots \)}(4k-1);
\path (level2n-1) node(2n-11)[dot]{}++(1,0) node(2n-12)[dot]{}++(2,0) node(2n-1k-1)[dot]{};
\path (2n-12)--node{\( \dots \)}(2n-1k-1);
\path (level2n) node(2n1)[dot]{}++(1,0) node(2n2)[dot]{}++(2,0) node(2nk-1)[dot]{}++(1,0) node(2nk)[dot]{} node[terminal][label=right:\( v'_t \)]{};
\path (2n2)--node{\( \dots \)}(2nk-1);

\draw (11) -- (21)
      (11) -- (22)
      (11) -- (2k-1)
      (11) -- (2k);
\draw (12) -- (21)
      (12) -- (22)
      (12) -- (2k-1)
      (12) -- (2k);
\draw (1k-1)--(21)
      (1k-1)--(22)
      (1k-1)--(2k-1)
      (1k-1)--(2k);
\draw (21)--(31)  (22)--(32)  (2k-1)--(3k-1);
\draw (31) -- (41)
      (31) -- (42)
      (31) -- (4k-1)
      (31) -- (4k);
\draw (32) -- (41)
      (32) -- (42)
      (32) -- (4k-1)
      (32) -- (4k);
\draw (3k-1)--(41)
      (3k-1)--(42)
      (3k-1)--(4k-1)
      (3k-1)--(4k);
\draw (41)--+(0,-0.5)  (42)--+(0,-0.5)  (4k-1)--+(0,-0.5);
\draw (2n-11)--+(0,0.5)  (2n-12)--+(0,0.5)  (2n-1k-1)--+(0,0.5);
\draw (2n-11) -- (2n1)
      (2n-11) -- (2n2)
      (2n-11) -- (2nk-1)
      (2n-11) -- (2nk);
\draw (2n-12) -- (2n1)
      (2n-12) -- (2n2)
      (2n-12) -- (2nk-1)
      (2n-12) -- (2nk);
\draw (2n-1k-1)--(2n1)
      (2n-1k-1)--(2n2)
      (2n-1k-1)--(2nk-1)
      (2n-1k-1)--(2nk);

\path (1k-1)+(0.1,0) coordinate(braceEnd);
\draw [decorate,decoration={brace,amplitude=8pt,raise=3mm}] (11)+(-0.1,0)-- node[above=4.5mm](count){\( k-1 \)} (braceEnd);
\end{tikzpicture}
\caption{}\label{fig:acyclic colouring chain gadget}
\end{subfigure}\hfill
\begin{subfigure}[b]{0.47\textwidth}
\centering
\begin{tikzpicture}
\coordinate (level1) node[left=0.75cm]{Level 1};
\path (level1)++(0,-2) coordinate (level2) node[left=0.75cm]{Level 2}++(0,-1) coordinate (level3) node[left=0.75cm]{Level 3} 
++(0,-2) coordinate (level4) node[left=0.75cm]{Level 4} 
++(0,-3) coordinate (level2n-1) node[left=0.5cm,align=center]{Level\\\( 2t-1 \)} 
++(0,-2) coordinate (level2n) node[left=0.75cm]{Level \( 2t \)};
\path (level4)-- node[pos=0.5,xshift=1.5cm,font=\LARGE,rotate=90]{.\,.\,.} (level2n-1);

\path (level1) node(11)[dot][label={[vcolour]left:1}]{}++(1,0) node(12)[dot][label={[vcolour]left:2}]{}++(2,0) node(1k-1)[dot][label={[vcolour]right:\( k\!-\!1 \)}]{};
\path (12)--node{\( \dots \)}(1k-1);
\path (level2) node(21)[dot][label={[vcolour,yshift=-3pt]left:0}]{}++(1,0) node(22)[dot][label={[vcolour,yshift=-3pt]left:0}]{}++(2,0) node(2k-1)[dot][label={[vcolour,yshift=-3pt]left:0}]{}++(1,0) node(2k)[dot]{} node[terminal][label={[vcolour,yshift=-3pt]left:0}][label=right:\( v'_1 \)]{};
\path (22)--node{\( \dots \)}(2k-1);
\path (level3) node(31)[dot][label={[vcolour]left:1}]{}++(1,0) node(32)[dot][label={[vcolour]left:2}]{}++(2,0) node(3k-1)[dot][label={[vcolour]right:\( k\!-\!1 \)}]{};
\path (32)--node{\( \dots \)}(3k-1);
\path (level4) node(41)[dot][label={[vcolour,yshift=-3pt]left:0}]{}++(1,0) node(42)[dot][label={[vcolour,yshift=-3pt]left:0}]{}++(2,0) node(4k-1)[dot][label={[vcolour,yshift=-3pt]left:0}]{}++(1,0) node(4k)[dot]{} node[terminal][label={[vcolour,yshift=-3pt]left:0}][label=right:\( v'_2 \)]{};
\path (42)--node{\( \dots \)}(4k-1);
\path (level2n-1) node(2n-11)[dot][label={[vcolour]left:1}]{}++(1,0) node(2n-12)[dot][label={[vcolour]left:2}]{}++(2,0) node(2n-1k-1)[dot][label={[vcolour]right:\( k\!-\!1 \)}]{};
\path (2n-12)--node{\( \dots \)}(2n-1k-1);
\path (level2n) node(2n1)[dot][label={[vcolour,yshift=-3pt]left:0}]{}++(1,0) node(2n2)[dot][label={[vcolour,yshift=-3pt]left:0}]{}++(2,0) node(2nk-1)[dot][label={[vcolour,yshift=-3pt]left:0}]{}++(1,0) node(2nk)[dot]{} node[terminal][label={[vcolour,yshift=-3pt]left:0}][label=right:\( v'_t \)]{};
\path (2n2)--node{\( \dots \)}(2nk-1);

\draw (11) -- (22)
      (11) -- (2k-1);
\draw (12) -- (21)
      (12) -- (22)
      (12) -- (2k-1)
      (12) -- (2k);
\draw (1k-1)--(21)
      (1k-1)--(22)
      (1k-1)--(2k-1)
      (1k-1)--(2k);
\draw             (22)--(32)  (2k-1)--(3k-1);
\draw (31) -- (42)
      (31) -- (4k-1)
      (31) -- (4k);
\draw (32) -- (41)
      (32) -- (42)
      (32) -- (4k-1)
      (32) -- (4k);
\draw (3k-1)--(41)
      (3k-1)--(42)
      (3k-1)--(4k-1)
      (3k-1)--(4k);
\draw                  (42)--+(0,-0.5)  (4k-1)--+(0,-0.5);
\draw                    (2n-12)--+(0,0.5)  (2n-1k-1)--+(0,0.5);
\draw (2n-11) -- (2n1)
      (2n-11) -- (2n2)
      (2n-11) -- (2nk-1);
\draw (2n-12) -- (2n1)
      (2n-12) -- (2n2)
      (2n-12) -- (2nk-1)
      (2n-12) -- (2nk);
\draw (2n-1k-1)--(2n1)
      (2n-1k-1)--(2n2)
      (2n-1k-1)--(2nk-1)
      (2n-1k-1)--(2nk);

\draw [very thick]
      (11) -- (21)
      (11) -- (2k)
      (21)--(31)
      (31) -- (41)
      (41)--+(0,-0.5) 
      (2n-11)--+(0,0.5)
      (2n-11) -- (2nk);

\path (1k-1)+(0.1,0) coordinate(braceEnd);
\draw [decorate,decoration={brace,amplitude=8pt,raise=3mm}] (11)+(-0.1,0)-- node[above=4.5mm](count){\( k-1 \)} (braceEnd);
\end{tikzpicture}
\caption{}\label{fig:k-acyclic colouring of chain gadget}
\end{subfigure}\caption{(a) The chain gadget, and (b) a \( k \)-acyclic colouring of the chain gadget.}
\end{figure}
\iftoggle{forThesis}
{ } { For every construction in this paper, only selected vertices within each gadget are allowed to have neighbours outside the gadget. 
We call such vertices as the \emph{terminals} of the gadget, and highlight them in diagrams by drawing a circle around them.

} \iftoggle{forThesis}
{ The chain gadget plays a major role in our constructions. 
} { The graph displayed in Figure~\ref{fig:acyclic colouring chain gadget}, let us call it the chain gadget, plays a major role in our constructions. 
} In Figure~\ref{fig:acyclic colouring chain gadget}, the terminals of the chain gadget are labelled \( v_1', v_2', \dots, v_t' \), where \( t\in \mathbb{N} \). 
\iftoggle{forThesis}
{ Consider the \( k \)-colouring \( h \) of the chain gadget shown in Figure~\ref{fig:k-acyclic colouring of chain gadget}. 
For \( i\in \{1,2,\dots,t\} \), \( h \) assigns colour~0 on vertices in Level~\( 2i \) and a permutation of colours \( 1,2,\dots,k-1 \) on vertices in Level~\( 2i-1 \). 
Observe that each cycle \( C \) in the chain gadget must contain two or more vertices each, from some consecutive levels of the chain gadget: i.e., (i)~there exists an \( i\in \{1,2,\dots,t\} \) such that \( C \) contains at least two vertices in Level~\( 2i-1 \) and at least two vertices in Level~\( 2i \), or (ii)~there exists an \( i\in \{1,2,\dots,t-1\} \) such that \( C \) contains at least two vertices in Level~\( 2i \) and at least two vertices in Level~\( 2i+1 \). 
Since \( h \) assigns colour~0 on even-level vertices and pairwise distinct non-zero colours on vertices in Level~\( 2i-1 \) for \( i\in \{1,2,\dots,t\} \), at least three colours are used by \( h \) on \( C \). 
Since the arbitrary cycle \( C \) is not bicoloured by \( h \), \( h \) is a \( k \)-acyclic colouring of the chain gadget. 
} { The next lemma explains the importance of the chain gadget. 

} \begin{lemma}[Properties of the chain gadget]\label{lem:acyclic colouring chain gadget properties}
The colouring displayed in Figure~\ref{fig:k-acyclic colouring of chain gadget} is the unique \( k \)-acyclic colouring of the chain gadget up to colour swaps and automorphisms. 
In particular, the following hold for every \( k \)-acyclic colouring of the chain gadget:\\
(i)~there is a colour \( c_1 \) such that every terminal (of the chain gadget) gets colour \( c_1 \), and\\
(ii)~for every colour \( c_2 (\neq c_1) \) and every pair of terminals \( x \) and \( y \), there is an \( x,y \)-path coloured using only \( c_1 \) and \( c_2 \).
\end{lemma}
\begin{proof}
Let \( f \) be a \( k \)-acyclic colouring of the chain gadget. 
Observe that for each \( i\in\{1,2,\dots,t\} \), vertices of the chain gadget from Levels \( 2i-1 \) and \( 2i \) together form \( K_{k-1,k} \). 
Therefore, for fixed \( i \), vertices at Level~\( 2i-1 \) get pairwise distinct colours, and vertices at Level~\( 2i \) share the same colour by Observation~\ref{obs:link in chain}. 
Without loss of generality, assume that vertices at Level~2 are coloured~0. 
Since each vertex at Level~3 has a neighbour in Level~2, colour~0 is unavailable in Level~3. 
That means vertices in Level~3 are coloured \( 1,2,\dots,k-1 \) in some order. 
Therefore, the colour shared by vertices at Level~4 must be 0. 
Similarly, for all \( i\in\{1,2,\dots,t\} \), vertices at Level~\( 2i \) are coloured~0 and vertices at Level~\( 2i-1 \) are assigned a permutation of colours \( 1,2,\dots,k-1 \). 
This proves that the colouring displayed in Figure~\ref{fig:k-acyclic colouring of chain gadget} is the unique \( k \)-acyclic colouring of the chain gadget up to colour swaps and automorphisms (observe that applying a permutation of colours on the set of vertices at Level~\( 2j-1 \) for each \( j\in \{1,2,\dots,t\} \) corresponds to an automorphism of the chain gadget; see Section~4 in the supplementary material for a demonstration). 
Since all terminals are coloured~0, Property~(i) is proved.
Observe that in Figure~\ref{fig:k-acyclic colouring of chain gadget}, for every colour \( c_2\neq 0 \) and every pair of terminals \( x \) and \( y \), there is an \( x,y \)-path coloured using only 0 and \( c_2 \) (for \( x=v_1' \), \( y=v_t' \) and \( c_2=1 \), such an \( x,y \)-path is highlighted in Figure~\ref{fig:k-acyclic colouring of chain gadget}). 
This proves Property~(ii).
\end{proof}

The following construction is employed to prove NP-completeness of \textsc{\( k \)-Acyclic Colourability} in bipartite graphs of maximum degree \( k+1 \).

\begin{construct}\label{make:acyclic}
\emph{Parameter:} An integer \( k\geq 3\).\\
\emph{Input:} A graph \( G \) of maximum degree \( 2(k-1) \).\\
\emph{Output:} A bipartite graph \( G' \) of maximum degree \( k+1 \).\\
\emph{Guarantee 1:} \( G \) is \( k \)-colourable if and only if \( G' \) is \( k \)-acyclic colourable.\\
\emph{Guarantee 2:} \( G' \) has only \( O(n) \) vertices where \( n=|V(G)| \).\\
\emph{Steps:}\\
Replace each vertex \( v \) of \( G \) by a chain gadget with \( k\cdot \deg_G(v) \) terminals, and reserve \( k \) terminals each for every neighbour of \( v \) in \( G \). 
For each neighbour \( u \) of \( v \) in \( G \), label the \( k \) terminals of the chain gadget for \( v \) reserved for \( u \) as  \( v_{u1}, v_{u2}, \dots, v_{uk} \). 
For each edge \( e=uv \) of \( G \) and each \( j\in\{1,2,\dots,k\} \), introduce a new vertex \( e_j \) in \( G' \) and join \( e_j \) to \( u_{vj} \) as well as \( v_{uj} \) (see Figure~\ref{fig:colouring to acyclic colouring}). 
An example of the construction is exhibited in Figure~\ref{fig:eg make acyclic v2}. 
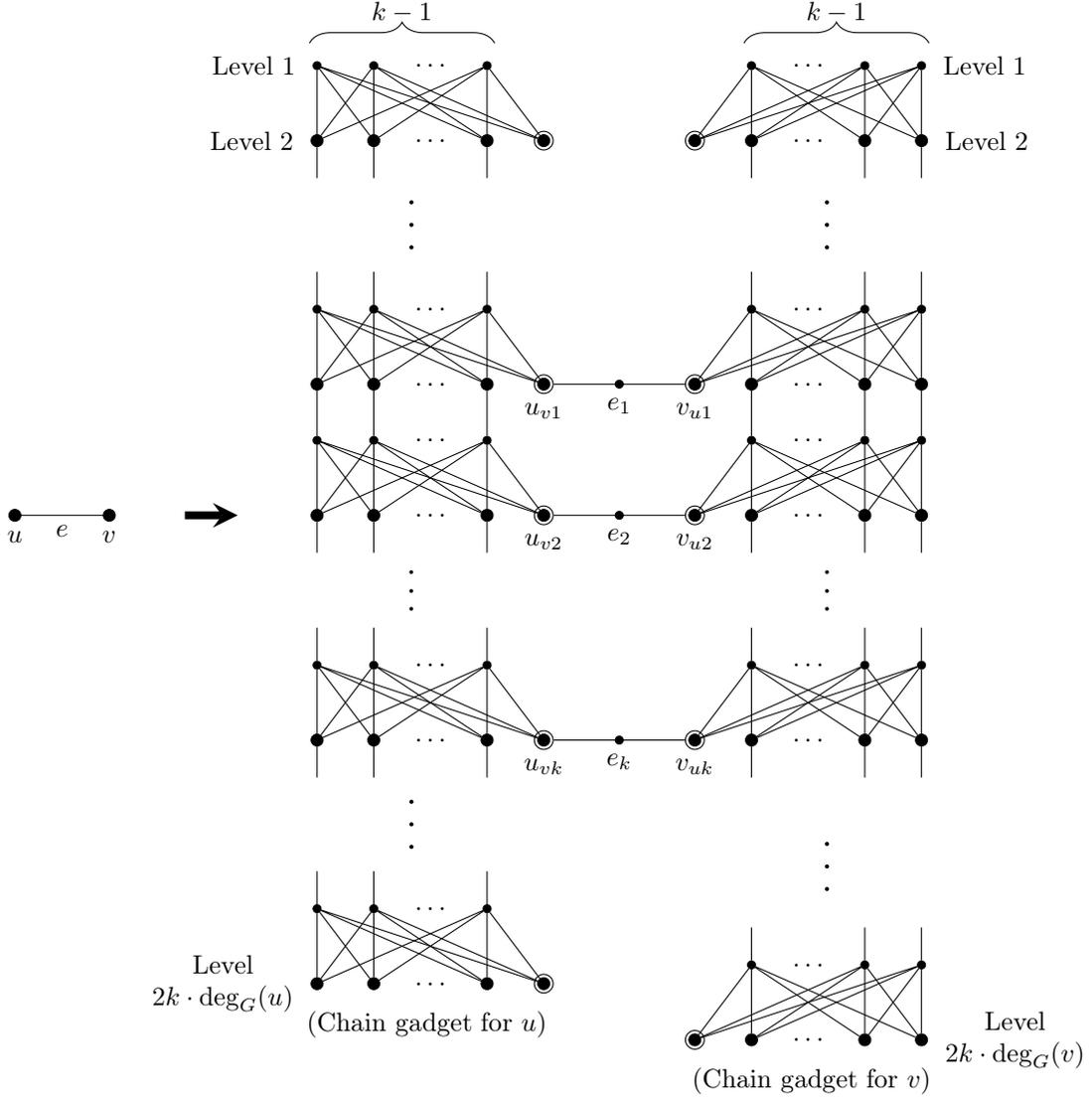
\begin{figure}[hbt]
\centering
\begin{tikzpicture}
\tikzset{
biggerdot/.style={dot,minimum size = 4.5pt}
}

\coordinate (level00); \path (level00)++(0,-1) coordinate (level01)++(0,-2.25) coordinate (level1)++(0,-1) coordinate (level2)++(0,-0.75) coordinate (level3) 
++(0,-1) coordinate (level4)
++(0,-2) coordinate (level2n-1) 
++(0,-1) coordinate (level2n)
++(0,-2.25) coordinate (level2n+1)
++(0,-1) coordinate (level2n+2);
\path (level01)-- node[pos=0.5,xshift=1.25cm,font=\huge,rotate=90]{.\,.\,.} (level1);
\path (level4)-- node[pos=0.5,xshift=1.25cm,font=\LARGE,rotate=90]{.\,.\,.} (level2n-1);
\path (level2n)-- node[pos=0.5,xshift=1.25cm,font=\huge,rotate=90]{.\,.\,.} (level2n+1);

\path (level01)-- node[pos=0.5,xshift=6.75cm,font=\huge,rotate=90]{.\,.\,.} (level1);
\path (level4)-- node[pos=0.5,xshift=6.75cm,font=\LARGE,rotate=90]{.\,.\,.} (level2n-1);
\path (level2n)-- node[pos=0.75,xshift=6.75cm,font=\huge,rotate=90]{.\,.\,.} (level2n+1);
\path (level4) coordinate (chainMid); 

\path (level00) node(001)[dot][label={[label distance=3pt]left:Level 1}]{}++(0.75,0) node(002)[dot]{}++(1.5,0) node(00k-1)[dot]{};
\path (002)--node{\( \dots \)}(00k-1);
\path (level01) node(011)[biggerdot][label={[label distance=3pt]left:Level 2}]{}++(0.75,0) node(012)[biggerdot]{}++(1.5,0) node(01k-1)[biggerdot]{}++(0.75,0) node(01k)[biggerdot]{} node[terminal]{};
\path (012)--node{\( \dots \)}(01k-1);

\path (level1) node(11)[dot]{}++(0.75,0) node(12)[dot]{}++(1.5,0) node(1k-1)[dot]{};
\path (12)--node{\( \dots \)}(1k-1);
\path (level2) node(21)[biggerdot]{}++(0.75,0) node(22)[biggerdot]{}++(1.5,0) node(2k-1)[biggerdot]{}++(0.75,0) node(2k)[biggerdot]{} node(uv1)[terminal][label=below:\( u_{v1} \)]{};
\path (22)--node{\( \dots \)}(2k-1);
\path (level3) node(31)[dot]{}++(0.75,0) node(32)[dot]{}++(1.5,0) node(3k-1)[dot]{};
\path (32)--node{\( \dots \)}(3k-1);
\path (level4) node(41)[biggerdot]{}++(0.75,0) node(42)[biggerdot]{}++(1.5,0) node(4k-1)[biggerdot]{}++(0.75,0) node(4k)[biggerdot]{} node(uv2)[terminal][label=below:\( u_{v2} \)]{};
\path (42)--node{\( \dots \)}(4k-1);
\path (level2n-1) node(2n-11)[dot]{}++(0.75,0) node(2n-12)[dot]{}++(1.5,0) node(2n-1k-1)[dot]{};
\path (2n-12)--node{\( \dots \)}(2n-1k-1);
\path (level2n) node(2n1)[biggerdot]{}++(0.75,0) node(2n2)[biggerdot]{}++(1.5,0) node(2nk-1)[biggerdot]{}++(0.75,0) node(2nk)[biggerdot]{} node(uvk)[terminal][label=below:\( u_{vk} \)]{};
\path (2n2)--node{\( \dots \)}(2nk-1);
\path (level2n+1) node(2n+11)[dot]{}++(0.75,0) node(2n+12)[dot]{}++(1.5,0) node(2n+1k-1)[dot]{};
\path (2n+12)--node{\( \dots \)}(2n+1k-1);
\path (level2n+2) node(2n+21)[biggerdot][label={[label distance=3pt,align=center]left:Level\\\( 2k\cdot \deg_G(u) \)}]{}++(0.75,0) node(2n+22)[biggerdot]{}++(1.5,0) node(2n+2k-1)[biggerdot]{}++(0.75,0) node(2n+2k)[biggerdot]{} node[terminal]{};
\path (2n+22)--node{\( \dots \)}(2n+2k-1);

\draw (001) -- (011)
      (001) -- (012)
      (001) -- (01k-1)
      (001) -- (01k);
\draw (002) -- (011)
      (002) -- (012)
      (002) -- (01k-1)
      (002) -- (01k);
\draw (00k-1)--(011)
      (00k-1)--(012)
      (00k-1)--(01k-1)
      (00k-1)--(01k);
\draw (011)--+(0,-0.5)  (012)--+(0,-0.5)  (01k-1)--+(0,-0.5);
\draw (11)--+(0,0.5)  (12)--+(0,0.5)  (1k-1)--+(0,0.5);
\draw (11) -- (21)
      (11) -- (22)
      (11) -- (2k-1)
      (11) -- (2k);
\draw (12) -- (21)
      (12) -- (22)
      (12) -- (2k-1)
      (12) -- (2k);
\draw (1k-1)--(21)
      (1k-1)--(22)
      (1k-1)--(2k-1)
      (1k-1)--(2k);
\draw (21)--(31)  (22)--(32)  (2k-1)--(3k-1);
\draw (31) -- (41)
      (31) -- (42)
      (31) -- (4k-1)
      (31) -- (4k);
\draw (32) -- (41)
      (32) -- (42)
      (32) -- (4k-1)
      (32) -- (4k);
\draw (3k-1)--(41)
      (3k-1)--(42)
      (3k-1)--(4k-1)
      (3k-1)--(4k);
\draw (41)--+(0,-0.5)  (42)--+(0,-0.5)  (4k-1)--+(0,-0.5);
\draw (2n-11)--+(0,0.5)  (2n-12)--+(0,0.5)  (2n-1k-1)--+(0,0.5);
\draw (2n-11) -- (2n1)
      (2n-11) -- (2n2)
      (2n-11) -- (2nk-1)
      (2n-11) -- (2nk);
\draw (2n-12) -- (2n1)
      (2n-12) -- (2n2)
      (2n-12) -- (2nk-1)
      (2n-12) -- (2nk);
\draw (2n-1k-1)--(2n1)
      (2n-1k-1)--(2n2)
      (2n-1k-1)--(2nk-1)
      (2n-1k-1)--(2nk);
\draw (2n1)--+(0,-0.5)  (2n2)--+(0,-0.5)  (2nk-1)--+(0,-0.5);
\draw (2n+11)--+(0,0.5)  (2n+12)--+(0,0.5)  (2n+1k-1)--+(0,0.5);
\draw (2n+11) -- (2n+21)
      (2n+11) -- (2n+22)
      (2n+11) -- (2n+2k-1)
      (2n+11) -- (2n+2k);
\draw (2n+12) -- (2n+21)
      (2n+12) -- (2n+22)
      (2n+12) -- (2n+2k-1)
      (2n+12) -- (2n+2k);
\draw (2n+1k-1)--(2n+21)
      (2n+1k-1)--(2n+22)
      (2n+1k-1)--(2n+2k-1)
      (2n+1k-1)--(2n+2k);

\path (00k-1)+(0.1,0) coordinate(braceEnd);
\draw [decorate,decoration={brace,amplitude=8pt,raise=3mm}] (001)+(-0.1,0)-- node[above=4.5mm](count){\( k-1 \)} (braceEnd);

\path (2n+22) node[below=7pt,xshift=20pt]{(Chain gadget for \( u \))};

\path (level00)++(8,0) node(001)[dot][label={[label distance=3pt]right:Level 1}]{}++(-0.75,0) node(002)[dot]{}++(-1.5,0) node(00k-1)[dot]{};
\path (002)--node{\( \dots \)}(00k-1);
\path (level01)++(8,0) node(011)[biggerdot][label={[label distance=3pt]right:Level 2}]{}++(-0.75,0) node(012)[biggerdot]{}++(-1.5,0) node(01k-1)[biggerdot]{}++(-0.75,0) node(01k)[biggerdot]{} node[terminal]{};
\path (012)--node{\( \dots \)}(01k-1);

\path (level1)++(8,0) node(11)[dot]{}++(-0.75,0) node(12)[dot]{}++(-1.5,0) node(1k-1)[dot]{};
\path (12)--node{\( \dots \)}(1k-1);
\path (level2)++(8,0) node(21)[biggerdot]{}++(-0.75,0) node(22)[biggerdot]{}++(-1.5,0) node(2k-1)[biggerdot]{}++(-0.75,0) node(2k)[biggerdot]{} node(vu1)[terminal][label=below:\( v_{u1} \)]{};
\path (22)--node{\( \dots \)}(2k-1);
\path (level3)++(8,0) node(31)[dot]{}++(-0.75,0) node(32)[dot]{}++(-1.5,0) node(3k-1)[dot]{};
\path (32)--node{\( \dots \)}(3k-1);
\path (level4)++(8,0) node(41)[biggerdot]{}++(-0.75,0) node(42)[biggerdot]{}++(-1.5,0) node(4k-1)[biggerdot]{}++(-0.75,0) node(4k)[biggerdot]{} node(vu2)[terminal][label=below:\( v_{u2} \)]{};
\path (42)--node{\( \dots \)}(4k-1);
\path (level2n-1)++(8,0) node(2n-11)[dot]{}++(-0.75,0) node(2n-12)[dot]{}++(-1.5,0) node(2n-1k-1)[dot]{};
\path (2n-12)--node{\( \dots \)}(2n-1k-1);
\path (level2n)++(8,0) node(2n1)[biggerdot]{}++(-0.75,0) node(2n2)[biggerdot]{}++(-1.5,0) node(2nk-1)[biggerdot]{}++(-0.75,0) node(2nk)[biggerdot]{} node(vuk)[terminal][label=below:\( v_{uk} \)]{};
\path (2n2)--node{\( \dots \)}(2nk-1);
\path (level2n+1)++(8,-0.75) node(2n+11)[dot]{}++(-0.75,0) node(2n+12)[dot]{}++(-1.5,0) node(2n+1k-1)[dot]{};
\path (2n+12)--node{\( \dots \)}(2n+1k-1);
\path (level2n+2)++(8,-0.75) node(2n+21)[biggerdot][label={[label distance=3pt,align=center]right:Level\\\( 2k\cdot \deg_G(v) \)}]{}++(-0.75,0) node(2n+22)[biggerdot]{}++(-1.5,0) node(2n+2k-1)[biggerdot]{}++(-0.75,0) node(2n+2k)[biggerdot]{} node[terminal]{};
\path (2n+22)--node{\( \dots \)}(2n+2k-1);

\draw (001) -- (011)
      (001) -- (012)
      (001) -- (01k-1)
      (001) -- (01k);
\draw (002) -- (011)
      (002) -- (012)
      (002) -- (01k-1)
      (002) -- (01k);
\draw (00k-1)--(011)
      (00k-1)--(012)
      (00k-1)--(01k-1)
      (00k-1)--(01k);
\draw (011)--+(0,-0.5)  (012)--+(0,-0.5)  (01k-1)--+(0,-0.5);
\draw (11)--+(0,0.5)  (12)--+(0,0.5)  (1k-1)--+(0,0.5);
\draw (11) -- (21)
      (11) -- (22)
      (11) -- (2k-1)
      (11) -- (2k);
\draw (12) -- (21)
      (12) -- (22)
      (12) -- (2k-1)
      (12) -- (2k);
\draw (1k-1)--(21)
      (1k-1)--(22)
      (1k-1)--(2k-1)
      (1k-1)--(2k);
\draw (21)--(31)  (22)--(32)  (2k-1)--(3k-1);
\draw (31) -- (41)
      (31) -- (42)
      (31) -- (4k-1)
      (31) -- (4k);
\draw (32) -- (41)
      (32) -- (42)
      (32) -- (4k-1)
      (32) -- (4k);
\draw (3k-1)--(41)
      (3k-1)--(42)
      (3k-1)--(4k-1)
      (3k-1)--(4k);
\draw (41)--+(0,-0.5)  (42)--+(0,-0.5)  (4k-1)--+(0,-0.5);
\draw (2n-11)--+(0,0.5)  (2n-12)--+(0,0.5)  (2n-1k-1)--+(0,0.5);
\draw (2n-11) -- (2n1)
      (2n-11) -- (2n2)
      (2n-11) -- (2nk-1)
      (2n-11) -- (2nk);
\draw (2n-12) -- (2n1)
      (2n-12) -- (2n2)
      (2n-12) -- (2nk-1)
      (2n-12) -- (2nk);
\draw (2n-1k-1)--(2n1)
      (2n-1k-1)--(2n2)
      (2n-1k-1)--(2nk-1)
      (2n-1k-1)--(2nk);
\draw (2n1)--+(0,-0.5)  (2n2)--+(0,-0.5)  (2nk-1)--+(0,-0.5);
\draw (2n+11)--+(0,0.5)  (2n+12)--+(0,0.5)  (2n+1k-1)--+(0,0.5);
\draw (2n+11) -- (2n+21)
      (2n+11) -- (2n+22)
      (2n+11) -- (2n+2k-1)
      (2n+11) -- (2n+2k);
\draw (2n+12) -- (2n+21)
      (2n+12) -- (2n+22)
      (2n+12) -- (2n+2k-1)
      (2n+12) -- (2n+2k);
\draw (2n+1k-1)--(2n+21)
      (2n+1k-1)--(2n+22)
      (2n+1k-1)--(2n+2k-1)
      (2n+1k-1)--(2n+2k);

\path (00k-1)+(-0.1,0) coordinate(braceEnd);
\draw [decorate,decoration={brace,amplitude=8pt,raise=3mm,mirror}] (001)+(0.1,0)-- node[above=4.5mm](count){\( k-1 \)} (braceEnd);

\path (2n+22) node[below=7pt,xshift=-20pt]{(Chain gadget for \( v \))};

\draw (uv1)--node[dot][label=below:\( e_1 \)]{} (vu1);
\draw (uv2)--node[dot][label=below:\( e_2 \)]{} (vu2);
\draw (uvk)--node[dot][label=below:\( e_k \)]{} (vuk);

\draw (chainMid)++(-4,0) node(u)[biggerdot][label=below:\( u \)]{}--++(1.25,0) node(v)[biggerdot][label=below:\( v \)]{}++(1.0,0) coordinate(from)++(0.7,0) coordinate(to);
\path (u)--node[below]{\( e \)} (v);
\draw [-stealth,draw=black,line width=3pt] (from)--(to);

\end{tikzpicture}
\caption[Construction of \( G' \) from \( G \)]{Construction of \( G' \) from \( G \) (only vertices \( u,v \) in \( G \) and edge \( uv \) in \( G \), and corresponding gadgets in \( G' \) are displayed).}
\label{fig:colouring to acyclic colouring}
\end{figure}

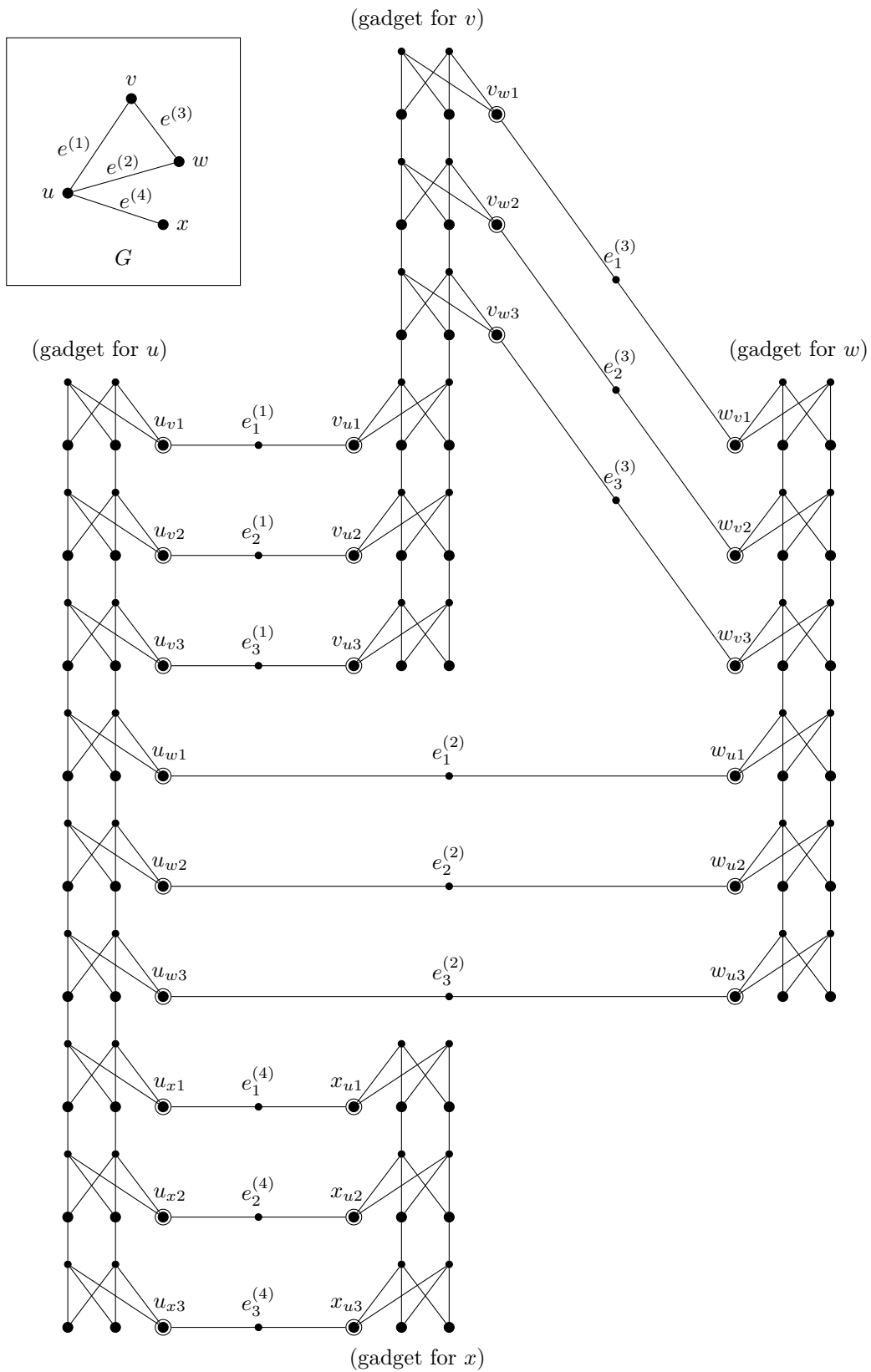
\begin{figure}[hbtp]
\centering 
\begin{tikzpicture}
\tikzset{
biggerdot/.style={dot,minimum size = 4.5pt}
}

\path (0,3) node (v1)[biggerdot][label=left:\( u \)]{};
\path (v1) --+(1,1.5) node (v2)[biggerdot][label=\( v \)]{};
\path (v1) --+(1.75,0.5) node (v3)[biggerdot][label=right:\( w \)]{}--(v2);
\path (v1) --+(1.5,-0.5) node (v4)[biggerdot][label=right:\( x \)]{};
\node [fit={(v1) (v2) (v3) (v4)},draw,inner sep=25pt][label={[yshift=20pt]-90:\( G \)}] {};

\draw (v1) --node[left]{\( e^{(1)} \)} (v2);
\draw (v1) --node[above=-2pt]{\( e^{(2)} \)} (v3);
\draw (v3) --node[above right=-2pt]{\( e^{(3)} \)} (v2);
\draw (v1) --node[pos=0.75,above]{\( e^{(4)} \)} (v4);

\coordinate (level1);
\path (level1)
++(0,-1) coordinate (level2)
++(0,-0.75) coordinate (level3)
++(0,-1) coordinate (level4)
++(0,-0.75) coordinate (level5)
++(0,-1) coordinate (level6)
++(0,-0.75) coordinate (level7)
++(0,-1) coordinate (level8)
++(0,-0.75) coordinate (level9)
++(0,-1) coordinate (level10)
++(0,-0.75) coordinate (level11)
++(0,-1) coordinate (level12)
++(0,-0.75) coordinate (level13)
++(0,-1) coordinate (level14)
++(0,-0.75) coordinate (level15)
++(0,-1) coordinate (level16)
++(0,-0.75) coordinate (level17)
++(0,-1) coordinate (level18);

\path (level1) node(11)[dot]{}++(0.75,0) node(12)[dot]{};
\path (level2) node(21)[biggerdot]{}++(0.75,0) node(22)[biggerdot]{}++(0.75,0) node(2k)[biggerdot]{} node(v11)[terminal][label={[xshift=3pt]\( u_{v1} \)}]{};
\path (level3) node(31)[dot]{}++(0.75,0) node(32)[dot]{};
\path (level4) node(41)[biggerdot]{}++(0.75,0) node(42)[biggerdot]{}++(0.75,0) node(4k)[biggerdot]{} node(v12)[terminal][label={[xshift=3pt]\( u_{v2} \)}]{};
\path (level5) node(51)[dot]{}++(0.75,0) node(52)[dot]{};
\path (level6) node(61)[biggerdot]{}++(0.75,0) node(62)[biggerdot]{}++(0.75,0) node(6k)[biggerdot]{} node(v13)[terminal][label={[xshift=3pt]\( u_{v3} \)}]{};
\path (level7) node(71)[dot]{}++(0.75,0) node(72)[dot]{};
\path (level8) node(81)[biggerdot]{}++(0.75,0) node(82)[biggerdot]{}++(0.75,0) node(8k)[biggerdot]{} node(v14)[terminal][label={[xshift=3pt]\( u_{w1} \)}]{};
\path (level9) node(91)[dot]{}++(0.75,0) node(92)[dot]{};
\path (level10) node(101)[biggerdot]{}++(0.75,0) node(102)[biggerdot]{}++(0.75,0) node(10k)[biggerdot]{} node(v15)[terminal][label={[xshift=3pt]\( u_{w2} \)}]{};
\path (level11) node(111)[dot]{}++(0.75,0) node(112)[dot]{};
\path (level12) node(121)[biggerdot]{}++(0.75,0) node(122)[biggerdot]{}++(0.75,0) node(12k)[biggerdot]{} node(v16)[terminal][label={[xshift=3pt]\( u_{w3} \)}]{};
\path (level13) node(131)[dot]{}++(0.75,0) node(132)[dot]{};
\path (level14) node(141)[biggerdot]{}++(0.75,0) node(142)[biggerdot]{}++(0.75,0) node(14k)[biggerdot]{} node(v17)[terminal][label={[xshift=3pt]\( u_{x1} \)}]{};
\path (level15) node(151)[dot]{}++(0.75,0) node(152)[dot]{};
\path (level16) node(161)[biggerdot]{}++(0.75,0) node(162)[biggerdot]{}++(0.75,0) node(16k)[biggerdot]{} node(v18)[terminal][label={[xshift=3pt]\( u_{x2} \)}]{};
\path (level17) node(171)[dot]{}++(0.75,0) node(172)[dot]{};
\path (level18) node(181)[biggerdot]{}++(0.75,0) node(182)[biggerdot]{}++(0.75,0) node(18k)[biggerdot]{} node(v19)[terminal][label={[xshift=3pt]\( u_{x3} \)}]{};

\draw
(11) -- (21)
(11) -- (22)
(11) -- (2k);
\draw
(12) -- (21)
(12) -- (22)
(12) -- (2k);
\draw (21)--(31)  (22)--(32);
\draw
(31) -- (41)
(31) -- (42)
(31) -- (4k);
\draw
(32) -- (41)
(32) -- (42)
(32) -- (4k);
\draw (41)--(51)  (42)--(52);
\draw
(51) -- (61)
(51) -- (62)
(51) -- (6k);
\draw
(52) -- (61)
(52) -- (62)
(52) -- (6k);
\draw (61)--(71)  (62)--(72);
\draw
(71) -- (81)
(71) -- (82)
(71) -- (8k);
\draw
(72) -- (81)
(72) -- (82)
(72) -- (8k);
\draw (81)--(91)  (82)--(92);
\draw
(91) -- (101)
(91) -- (102)
(91) -- (10k);
\draw
(92) -- (101)
(92) -- (102)
(92) -- (10k);
\draw (101)--(111)  (102)--(112);
\draw
(111) -- (121)
(111) -- (122)
(111) -- (12k);
\draw
(112) -- (121)
(112) -- (122)
(112) -- (12k);
\draw (121)--(131)  (122)--(132);
\draw
(131) -- (141)
(131) -- (142)
(131) -- (14k);
\draw
(132) -- (141)
(132) -- (142)
(132) -- (14k);
\draw (141)--(151)  (142)--(152);
\draw
(151) -- (161)
(151) -- (162)
(151) -- (16k);
\draw
(152) -- (161)
(152) -- (162)
(152) -- (16k);
\draw (161)--(171)  (162)--(172);
\draw 
(171) -- (181)
(171) -- (182)
(171) -- (18k);
\draw 
(172) -- (181)
(172) -- (182)
(172) -- (18k);

\path (level1) +(0.5,0.5) node{(gadget for \( u \))};

\path (level1) ++(6,5.25) coordinate (level1);
\path (level1)
++(0,-1) coordinate (level2) 
++(0,-0.75) coordinate (level3) 
++(0,-1) coordinate (level4) 
++(0,-0.75) coordinate (level5) 
++(0,-1) coordinate (level6) 
++(0,-0.75) coordinate (level7) 
++(0,-1) coordinate (level8) 
++(0,-0.75) coordinate (level9) 
++(0,-1) coordinate (level10) 
++(0,-0.75) coordinate (level11) 
++(0,-1) coordinate (level12);

\path (level1) node(11)[dot]{}++(-0.75,0) node(12)[dot]{};
\path (level2) node(21)[biggerdot]{}++(-0.75,0) node(22)[biggerdot]{}++(1.5,0) node(2k)[biggerdot]{} node(v21)[terminal][label={[xshift=3pt]\( v_{w1} \)}]{};
\path (level3) node(31)[dot]{}++(-0.75,0) node(32)[dot]{};
\path (level4) node(41)[biggerdot]{}++(-0.75,0) node(42)[biggerdot]{}++(1.5,0) node(4k)[biggerdot]{} node(v22)[terminal][label={[xshift=3pt]\( v_{w2} \)}]{};
\path (level5) node(51)[dot]{}++(-0.75,0) node(52)[dot]{};
\path (level6) node(61)[biggerdot]{}++(-0.75,0) node(62)[biggerdot]{}++(1.5,0) node(6k)[biggerdot]{} node(v23)[terminal][label={[xshift=3pt]\( v_{w3} \)}]{};
\path (level7) node(71)[dot]{}++(-0.75,0) node(72)[dot]{};
\path (level8) node(81)[biggerdot]{}++(-0.75,0) node(82)[biggerdot]{}++(-0.75,0) node(8k)[biggerdot]{} node(v24)[terminal][label={[xshift=-3pt]\( v_{u1} \)}]{};
\path (level9) node(91)[dot]{}++(-0.75,0) node(92)[dot]{};
\path (level10) node(101)[biggerdot]{}++(-0.75,0) node(102)[biggerdot]{}++(-0.75,0) node(10k)[biggerdot]{} node(v25)[terminal][label={[xshift=-3pt]\( v_{u2} \)}]{};
\path (level11) node(111)[dot]{}++(-0.75,0) node(112)[dot]{};
\path (level12) node(121)[biggerdot]{}++(-0.75,0) node(122)[biggerdot]{}++(-0.75,0) node(12k)[biggerdot]{} node(v26)[terminal][label={[xshift=-3pt]\( v_{u3} \)}]{};

\draw
(11) -- (21)
(11) -- (22)
(11) -- (2k);
\draw
(12) -- (21)
(12) -- (22)
(12) -- (2k);
\draw (21)--(31)  (22)--(32);
\draw
(31) -- (41)
(31) -- (42)
(31) -- (4k);
\draw
(32) -- (41)
(32) -- (42)
(32) -- (4k);
\draw (41)--(51)  (42)--(52);
\draw
(51) -- (61)
(51) -- (62)
(51) -- (6k);
\draw
(52) -- (61)
(52) -- (62)
(52) -- (6k);
\draw (61)--(71)  (62)--(72);
\draw
(71) -- (81)
(71) -- (82)
(71) -- (8k);
\draw
(72) -- (81)
(72) -- (82)
(72) -- (8k);
\draw (81)--(91)  (82)--(92);
\draw
(91) -- (101)
(91) -- (102)
(91) -- (10k);
\draw
(92) -- (101)
(92) -- (102)
(92) -- (10k);
\draw (101)--(111)  (102)--(112);
\draw
(111) -- (121)
(111) -- (122)
(111) -- (12k);
\draw
(112) -- (121)
(112) -- (122)
(112) -- (12k);

\path (level1) +(-0.5,0.5) node{(gadget for \( v \))};

\path (level1) ++(6,-5.25) coordinate (level1);
\path (level1)
++(0,-1) coordinate (level2) 
++(0,-0.75) coordinate (level3) 
++(0,-1) coordinate (level4) 
++(0,-0.75) coordinate (level5) 
++(0,-1) coordinate (level6) 
++(0,-0.75) coordinate (level7) 
++(0,-1) coordinate (level8) 
++(0,-0.75) coordinate (level9) 
++(0,-1) coordinate (level10) 
++(0,-0.75) coordinate (level11) 
++(0,-1) coordinate (level12);

\path (level1) node(11)[dot]{}++(-0.75,0) node(12)[dot]{};
\path (level2) node(21)[biggerdot]{}++(-0.75,0) node(22)[biggerdot]{}++(-0.75,0) node(2k)[biggerdot]{} node(v31)[terminal][label={[yshift=3pt]\( w_{v1} \)}]{};
\path (level3) node(31)[dot]{}++(-0.75,0) node(32)[dot]{};
\path (level4) node(41)[biggerdot]{}++(-0.75,0) node(42)[biggerdot]{}++(-0.75,0) node(4k)[biggerdot]{} node(v32)[terminal][label={[yshift=3pt]\( w_{v2} \)}]{};
\path (level5) node(51)[dot]{}++(-0.75,0) node(52)[dot]{};
\path (level6) node(61)[biggerdot]{}++(-0.75,0) node(62)[biggerdot]{}++(-0.75,0) node(6k)[biggerdot]{} node(v33)[terminal][label={[yshift=3pt]\( w_{v3} \)}]{};
\path (level7) node(71)[dot]{}++(-0.75,0) node(72)[dot]{};
\path (level8) node(81)[biggerdot]{}++(-0.75,0) node(82)[biggerdot]{}++(-0.75,0) node(8k)[biggerdot]{} node(v34)[terminal][label={[xshift=-3pt]\( w_{u1} \)}]{};
\path (level9) node(91)[dot]{}++(-0.75,0) node(92)[dot]{};
\path (level10) node(101)[biggerdot]{}++(-0.75,0) node(102)[biggerdot]{}++(-0.75,0) node(10k)[biggerdot]{} node(v35)[terminal][label={[xshift=-3pt]\( w_{u2} \)}]{};
\path (level11) node(111)[dot]{}++(-0.75,0) node(112)[dot]{};
\path (level12) node(121)[biggerdot]{}++(-0.75,0) node(122)[biggerdot]{}++(-0.75,0) node(12k)[biggerdot]{} node(v36)[terminal][label={[xshift=-3pt]\( w_{u3} \)}]{};

\draw
(11) -- (21)
(11) -- (22)
(11) -- (2k);
\draw
(12) -- (21)
(12) -- (22)
(12) -- (2k);
\draw (21)--(31)  (22)--(32);
\draw
(31) -- (41)
(31) -- (42)
(31) -- (4k);
\draw
(32) -- (41)
(32) -- (42)
(32) -- (4k);
\draw (41)--(51)  (42)--(52);
\draw
(51) -- (61)
(51) -- (62)
(51) -- (6k);
\draw
(52) -- (61)
(52) -- (62)
(52) -- (6k);
\draw (61)--(71)  (62)--(72);
\draw
(71) -- (81)
(71) -- (82)
(71) -- (8k);
\draw
(72) -- (81)
(72) -- (82)
(72) -- (8k);
\draw (81)--(91)  (82)--(92);
\draw
(91) -- (101)
(91) -- (102)
(91) -- (10k);
\draw
(92) -- (101)
(92) -- (102)
(92) -- (10k);
\draw (101)--(111)  (102)--(112);
\draw
(111) -- (121)
(111) -- (122)
(111) -- (12k);
\draw
(112) -- (121)
(112) -- (122)
(112) -- (12k);

\path (level1) +(-0.5,0.5) node{(gadget for \( w \))};

\path (level13) ++(6,0) coordinate (level1);
\path (level1)
++(0,-1) coordinate (level2) 
++(0,-0.75) coordinate (level3) 
++(0,-1) coordinate (level4) 
++(0,-0.75) coordinate (level5) 
++(0,-1) coordinate (level6);

\path (level1) node(11)[dot]{}++(-0.75,0) node(12)[dot]{};
\path (level2) node(21)[biggerdot]{}++(-0.75,0) node(22)[biggerdot]{}++(-0.75,0) node(2k)[biggerdot]{} node(v41)[terminal][label={[xshift=-3pt]\( x_{u1} \)}]{};
\path (level3) node(31)[dot]{}++(-0.75,0) node(32)[dot]{};
\path (level4) node(41)[biggerdot]{}++(-0.75,0) node(42)[biggerdot]{}++(-0.75,0) node(4k)[biggerdot]{} node(v42)[terminal][label={[xshift=-3pt]\( x_{u2} \)}]{};
\path (level5) node(51)[dot]{}++(-0.75,0) node(52)[dot]{};
\path (level6) node(61)[biggerdot]{}++(-0.75,0) node(62)[biggerdot]{}++(-0.75,0) node(6k)[biggerdot]{} node(v43)[terminal][label={[xshift=-3pt]\( x_{u3} \)}]{};

\draw
(11) -- (21)
(11) -- (22)
(11) -- (2k);
\draw
(12) -- (21)
(12) -- (22)
(12) -- (2k);
\draw (21)--(31)  (22)--(32);
\draw
(31) -- (41)
(31) -- (42)
(31) -- (4k);
\draw
(32) -- (41)
(32) -- (42)
(32) -- (4k);
\draw (41)--(51)  (42)--(52);
\draw
(51) -- (61)
(51) -- (62)
(51) -- (6k);
\draw
(52) -- (61)
(52) -- (62)
(52) -- (6k);

\path (level6) +(-0.5,-0.5) node{(gadget for \( x \))};

\draw (v11) --node[dot][label=\( e^{(1)}_1 \)]{} (v24);
\draw (v12) --node[dot][label=\( e^{(1)}_2 \)]{} (v25);
\draw (v13) --node[dot][label=\( e^{(1)}_3 \)]{} (v26);
\draw (v14) --node[dot][label=\( e^{(2)}_1 \)]{} (v34);
\draw (v15) --node[dot][label=\( e^{(2)}_2 \)]{} (v35);
\draw (v16) --node[dot][label=\( e^{(2)}_3 \)]{} (v36);
\draw (v17) --node[dot][label=\( e^{(4)}_1 \)]{} (v41);
\draw (v18) --node[dot][label=\( e^{(4)}_2 \)]{} (v42);
\draw (v19) --node[dot][label=\( e^{(4)}_3 \)]{} (v43);
\draw (v21) --node[dot][label={[xshift=2pt]\( e^{(3)}_1 \)}]{} (v31);
\draw (v22) --node[dot][label={[xshift=2pt]\( e^{(3)}_2 \)}]{} (v32);
\draw (v23) --node[dot][label={[xshift=2pt]\( e^{(3)}_3 \)}]{} (v33);
\end{tikzpicture}
\caption[Example of Construction~\ref{make:acyclic} with \( k=3 \).]{Example of Construction~\ref{make:acyclic} with \( k=3 \). Graph \( G' \) is displayed large, and graph \( G \) is shown inset (for convenience, a graph of maximum degree 3 rather than 4 is used as \( G \)).}
\label{fig:eg make acyclic v2}
\end{figure}

\( G' \) is clearly bipartite (small dots form one part and big dots form the other part; see Figure~\ref{fig:eg make acyclic v2}). 
\end{construct}
\begin{proof}[Proof of Guarantee 1]
\iftoggle{forThesis}
{ Suppose that \( G \) admits a \( k \)-colouring \( f:V(G)\to\{0,1,\dots,k\text{-}1\} \). 
} { Suppose that \( G \) admits a \( k \)-colouring \( f:V(G)\to\{0,1,\dots,k-1\} \). 
} We produce a \( k \)-colouring \( f' \) of \( G' \) as follows, where \( f':V(G')\to\{0,1,\dots,k-1\} \). 
For each vertex \( v \) of \( G \), colour the chain gadget for \( v \) by the scheme obtained from Figure~\ref{fig:k-acyclic colouring of chain gadget} by swapping colour~0 with \( f(v) \). 
Now, the terminals of the chain gadget for \( v \) have colour \( f(v) \) under \( f' \). 
For each edge \( e=uv \) of \( G \), choose a colour \( c\in\{0,1,\dots,k-1\}\setminus\{f(u),f(v)\} \) and assign \( f'(e_j)=c \) for \( 1\leq j\leq k \). 
Since the paths of the form \( u_{vj},e_j,v_{uj} \) are tricoloured by \( f' \), any cycle in \( G' \) bicoloured by \( f' \) must be entirely within a chain gadget. 
Since chain gadgets are coloured by an acyclic colouring scheme, they do not contain any bicoloured cycle. 
\iftoggle{forThesis}
{ Therefore, \( f' \) is a \( k \)-acyclic colouring of \( G' \).\\[-10pt]
} { Therefore, \( f' \) is a \( k \)-acyclic colouring of \( G' \).
} 

\iftoggle{forThesis}
{ \noindent Conversely, suppose that \( G' \) admits a \( k \)-acyclic colouring \( f':V(G')\to\{0,1,\dots,k\text{-}1\} \). 
} { Conversely, suppose that \( G' \) admits a \( k \)-acyclic colouring \( f':V(G')\to\{0,1,\dots,k-1\} \). 
} By Property~(i) of the chain gadget (see Lemma~\ref{lem:acyclic colouring chain gadget properties}), all terminals of a chain gadget get the same colour. 
\iftoggle{forThesis}
{ For brevity, let us call this colour as ``the colour of the chain gadget''.\\[10pt]
\noindent \parbox{\textwidth}{\textit{Claim 1:} For every edge \( uv \) of \( G \), the colour of the chain gadget for \( u \) differs from the colour of the chain gadget for \( v \) (i.e., \( f'(u_{v1})\neq f'(v_{u1}) \)\,).}\\[10pt]
} { For brevity, let us call this colour as ``the colour of the chain gadget''.\\

\noindent \textit{Claim 1:} For every edge \( uv \) of \( G \), the colour of the chain gadget for \( u \) differs from the colour of the chain gadget for \( v \) (i.e., \( f'(u_{v1})\neq f'(v_{u1}) \)\,).\\

} \noindent Contrary to the claim, assume that \( uv \) is an edge in \( G \), but there is a colour \( c_1 \) such that \( f'(u_{v1})=f'(v_{u1})=c_1 \). 
Clearly, colour \( c_1 \) is unavailable for vertices \( e_j \) (\( 1\leq j\leq k \)). 
Hence, by pigeon-hole principle, at least two vertices among \( e_1,e_2,\dots,e_k \) have the same colour, say \( f(e_p)=f(e_q)=c_2 \) where \( p\neq q \) and \( c_2\neq c_1 \). 
By Property~(ii) of the chain gadget (see Lemma~\ref{lem:acyclic colouring chain gadget properties}), the chain gadget for \( u \) contains a path \( R_1 \) from \( u_{vp} \) to \( u_{vq} \) which is coloured using only \( c_1 \) and \( c_2 \). 
Similarly, the chain gadget for \( v \) contains a path \( R_2 \) from \( v_{uq} \) to \( v_{up} \) which is coloured using only \( c_1 \) and \( c_2 \). 
These paths together with the three-vertex paths \( u_{vq},e_q,v_{uq} \) and \( v_{up},e_p,u_{vp} \) form a cycle in \( G' \) bicoloured by \( f' \) (namely, the cycle \( (u_{vp};R_1;u_{vq},e_q,v_{uq};R_2;v_{up},e_p) \)\,). 
This contradiction proves Claim~1.

Producing a \( k \)-colouring \( f \) of \( G \) from \( f' \) is easy. 
For each vertex \( v \) of \( G \), assign at~\( v \), the colour of the chain gadget for \( v \). 
The function \( f \) is a \( k \)-colouring of \( G \) due to Claim~1.
\end{proof}
\begin{proof}[Proof of Guarantee 2]
Suppose that \( G \) has \( n \) vertices and \( m \) edges. 
For each vertex~\( v \) of~\( G \), the chain gadget for \( v \) has \( (2k-1)\cdot k\cdot \deg_G(v) \) vertices. 
Also, there are only \( km \) vertices outside of chain gadgets. 
Therefore, \( G' \) has \( \left(\sum_{v\in V(G)}(2k^2-k)\deg_G(v)\right)+km=\allowbreak (2k^2-k)\,2m+km=O(m) \) vertices. 
Since \( m\leq n\Delta(G)/2 \) and \( \Delta(G)\leq 2(k-1) \), we have \( m=O(n) \), and thus the number of vertices in \( G' \) is \( O(n) \).
\end{proof}

Since \( G' \) has only \( O(n) \) vertices and the maximum degree of \( G' \) is \( k+1 \), the graph \( G' \) has only \( O(n) \) edges as well (\( \because 2|E(G')|=(k+1)|V(G')| \)\,). 
Hence, Construction~\ref{make:acyclic} requires only time polynomial in the input size. 

\begin{theorem}\label{thm:k-acyclic npc}
For \( k\geq 3 \), \textsc{\( k \)-Acyclic Colourability}\( (\textup{bipartite, }\Delta=k+1) \) is NP-complete, and the problem does not admit a \( 2^{o(n)} \)-time algorithm unless ETH fails.
\end{theorem}
\begin{proof}
Fix an integer \( k\geq 3 \). 
We employ Construction~\ref{make:acyclic} to establish a reduction from \textsc{\( k \)-Colourability}\( (\Delta=2(k-1)) \) to \textsc{\( k \)-Acyclic Colourability}\( (\text{bipartite, }\Delta=k+1) \). 
The problem \textsc{\( k \)-Colourability}\( (\Delta=2(k-1)) \) is NP-complete (in fact, \textsc{\( k \)-Colourability} is NP-complete for line graphs of \( k \)-regular graphs \cite{leven_galil}) and does not admit a \( 2^{o(n)} \)-time algorithm unless ETH fails (the latter can be observed from a reduction of Emden-Weinert et al.~\cite{emden-weinert}).

Let \( G \) be an instance of \textsc{\( k \)-Colourability}\( (\Delta=2(k-1)) \). 
Produce a graph \( G' \) from \( G \) by Construction~\ref{make:acyclic}. 
Note that Construction~\ref{make:acyclic} requires only time polynomial in the size of \( G \). 
By Guarantee~1 of Construction~\ref{make:acyclic}, \( G \) is \( k \)-colourable if and only if \( G' \) is \( k \)-acyclic colourable. 
Besides, the number of vertices in \( G' \) is \( O(n) \) where \( n=|V(G)| \). 
Therefore, the problem \textsc{\( k \)-Acyclic Colourability}\( (\text{bipartite, }\Delta=k+1) \) is NP-complete, and it does not admit a \( 2^{o(n)} \)-time algorithm unless ETH fails.
\end{proof}

For \( k\geq 3 \), \textsc{\( k \)-Acyclic Colourability} is NP-complete for graphs of maximum degree \( k+1 \)  by Theorem~\ref{thm:k-acyclic npc}. 
If \textsc{\( k \)-Acyclic Colourability} is NP-complete for graphs of maximum degree \( d \), then it is NP-complete for graphs of maximum degree \( d+1 \). 
Thus, for \( k\geq 3 \) and \( d\geq k+1 \), \textsc{\( k \)-Acyclic Colourability} in graphs of maximum degree \( d \) is NP-complete. 
On the other hand, for all \( d\geq 2 \) and every \( d \)-regular graph \( G \), we have \( \chi_a(G)\geq \raisebox{1pt}{\big\lceil}\frac{d+3}{2}\raisebox{1pt}{\big\rceil} \), and thus \( G \) is not \( k \)-acyclic colourable for \( k<\raisebox{1pt}{\big\lceil}\frac{d+3}{2}\raisebox{1pt}{\big\rceil} \). 
Hence, for \( k\geq 3 \) and \( d\geq 2k-2 \), \textsc{\( k \)-Acyclic Colourability} in \( d \)-regular graphs is polynomial-time solvable (because \( k<\raisebox{1pt}{\big\lceil}\frac{(2k-2)+3}{2}\raisebox{1pt}{\big\rceil}\leq \raisebox{1pt}{\big\lceil}\frac{d+3}{2}\raisebox{1pt}{\big\rceil} \)). 
As a result, for \( k\geq 3 \) and \( d\geq 2k-2 \), \textsc{\( k \)-Acyclic Colourability} in graphs of maximum degree \( d \) is NP-complete (because \( 2k-2\geq k+1 \)) whereas \textsc{\( k \)-Acyclic Colourability} in \( d \)-regular graphs is polynomial-time solvable. 
In contrast, we use Construction~\ref{make:acyclic regular} below to show that for all \( k\geq 3 \) and \( d\leq 2k-3 \), the complexity of \textsc{\( k \)-Acyclic Colourability} is the same for graphs of maximum degree~\( d \) and \( d \)-regular graphs. 
First, we construct a gadget called the \emph{filler gadget} using the graph \( G_d \). 
Note that for every graph \( H \) with an edge \( xy \) and a \( k \)-acyclic colouring~\( h \), there is no \( x,y \)-path in \( H-xy \) bicoloured by~\( h \) (if not, that path together with the edge \( xy \) forms a cycle in \( H \) bicoloured by \( h \)). 
In particular, the graph \( G_d \) defined in the proof of Theorem~\ref{thm:chi_a attained} is a \( d \)-regular graph with a \( k \)-acyclic colouring \( h \) (because \( \chi_a(G_d)=\raisebox{1pt}{\big\lceil}\frac{d+3}{2}\raisebox{1pt}{\big\rceil}\leq \raisebox{1pt}{\big\lceil}\frac{(2k-3)+3}{2}\raisebox{1pt}{\big\rceil}=k \)), and for each edge \( xy \) of \( G_d \), there is no \( x,y \)-path in \( G_d-xy \) bicoloured by \( h \). 
We choose an edge \( xy \) of \( G_d \), and make the filler gadget using \( G_d-xy \) as shown in Figure~\ref{fig:acyclic regular filler gadget}. 
Note that every non-terminal vertex of the filler gadget has degree~\( d \). 

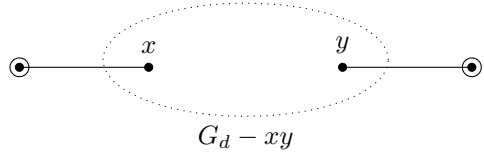
\begin{figure}[hbt]
\centering
\begin{tikzpicture}[scale=0.85]
\path (0,0) node(x)[dot][label=\( x \)]{} ++(3,0) node(y)[dot][label=\( y \)]{};
\path (x) --node[draw,ellipse,dotted,minimum width=3.75cm,minimum height=01.5cm,yshift=0.1cm][label=below:\( G_d-xy \)]{}  (y);

\draw
(x)--+(-2,0) node[dot]{} node[terminal]{}
(y)--+(2,0) node[dot]{} node[terminal]{};
\end{tikzpicture}
\caption[The filler gadget in Construction~\ref{make:acyclic regular}]{The filler gadget in Construction~\ref{make:acyclic regular} (except for vertices \( x \) and \( y \), vertices and edges within the copy of \( G_d-xy \) are not displayed).}
\label{fig:acyclic regular filler gadget}
\end{figure}

We use the vertex identification operation in Construction~\ref{make:acyclic regular}  (see Section~3 in the supplementary material for the definition of vertex idenitification).

\begin{construct}\label{make:acyclic regular}
\emph{Parameters:} Integers \( k\geq 3 \) and \( d\leq 2k-3 \).\\
\emph{Input:} A graph \( G \) of maximum degree \( d \).\\
\emph{Output:} A \( d \)-regular graph \( G' \).\\
\emph{Guarantee:} \( G \) is \( k \)-acyclic colourable if and only if \( G' \) is \( k \)-acyclic colourable.\\
\emph{Steps:}\\
Introduce two copies of \( G \), say \( G^{(1)} \) and \( G^{(2)} \). 
For each \( v\in V(G) \) and \( i\in \{1,2\} \), let \( v^{(i)} \) denote the copy of \( v \) in \( G^{(i)} \). 
For each \( v\in V(G) \), introduce \( d-\deg_G(v) \) filler gadgets, and for each of these filler gadgets, identify its two terminals with \( v^{(1)} \) and \( v^{(2)} \), respectively. 
See Figure~\ref{fig:acyclic regular} for an example. 

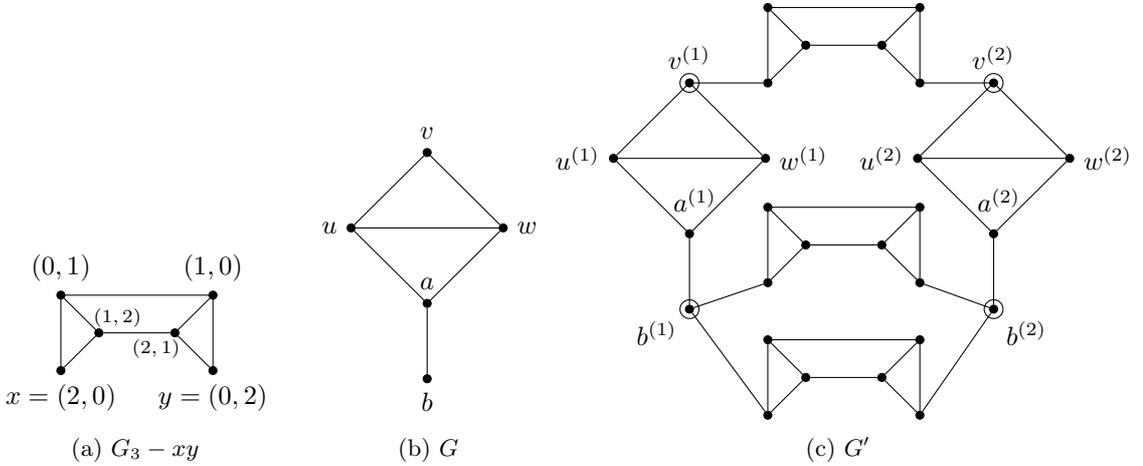
\begin{figure}[hbt]
\centering
\begin{subfigure}[b]{0.3\textwidth}
\centering
\begin{tikzpicture}
\path (0,0) node (x)[dot][label={below:\( x=(2,0) \)}]{};
\path (x) --++(2,0) node (y)[dot][label={below:\( y=(0,2) \)}]{};
\draw (x) --++(0,1) node (01)[dot][label={\( (0,1) \)}]{} --++(2,0) node (10)[dot][label={\( (1,0) \)}]{} --(y);
\draw (x) --+(0.5,0.5) node (12)[dot][label={[font=\scriptsize,xshift=7pt,yshift=-3pt]\( (1,2) \)}]{} -- (01);
\draw (y) --+(-0.5,0.5) node (21)[dot][label={[font=\scriptsize,xshift=-7pt,yshift=3pt]below:\( (2,1) \)}]{} -- (10);
\draw (12)--(21);
\end{tikzpicture}
\caption{\( G_3-xy \)}\end{subfigure}\begin{subfigure}[b]{0.2\textwidth}
\centering
\begin{tikzpicture}
\path (0,0) node (v)[dot][label=\( v \)]{} +(-1,-1) node (u)[dot][label=left:\( u \)]{} ++(1,-1) node (w)[dot][label=right:\( w \)]{} ++(-1,-1) node (a)[dot][label=\( a \)]{} ++(0,-1) node (b)[dot][label=below:\( b \)]{};
\draw (u)--(v)--(w)--(a)--(u)--(w)  (a)--(b);
\end{tikzpicture}
\caption{\( G \)}
\end{subfigure}\begin{subfigure}[b]{0.5\textwidth}
\centering
\begin{tikzpicture}
\path (0,0) node (v1)[dot][label=\( v^{(1)} \)]{} node[terminal]{} +(-1,-1) node (u1)[dot][label=left:\( u^{(1)} \)]{} ++(1,-1) node (w1)[dot][label=right:\( w^{(1)} \)]{} ++(-1,-1) node (a1)[dot][label={[yshift=3pt,xshift=2pt]\( a^{(1)} \)}]{} ++(0,-1) node (b1)[dot][label=below left:\( b^{(1)} \)]{} node[terminal]{};
\draw (u1)--(v1)--(w1)--(a1)--(u1)--(w1)  (a1)--(b1);

\path (4,0) node (v2)[dot][label=\( v^{(2)} \)]{} node[terminal]{} +(-1,-1) node (u2)[dot][label=left:\( u^{(2)} \)]{} ++(1,-1) node (w2)[dot][label=right:\( w^{(2)} \)]{} ++(-1,-1) node (a2)[dot][label={[yshift=3pt,xshift=2pt]\( a^{(2)} \)}]{} ++(0,-1) node (b2)[dot][label=below right:\( b^{(2)} \)]{} node[terminal]{};
\draw (u2)--(v2)--(w2)--(a2)--(u2)--(w2)  (a2)--(b2);

\path (v1) -- node (x)[dot][pos=0.25]{} (v2);
\path (x) ++(2,0) node (y)[dot]{};
\draw (x) --++(0,1) node (01)[dot]{} --++(2,0) node (10)[dot]{} --(y);
\draw (x) --+(0.5,0.5) node (12)[dot]{} -- (01);
\draw (y) --+(-0.5,0.5) node (21)[dot]{} -- (10);
\draw (12)--(21);
\draw (x)--(v1)  (y)--(v2);

\path (b1) -- node (x)[dot][pos=0.25,yshift=10pt]{} (b2);
\path (x) ++(2,0) node (y)[dot]{};
\draw (x) --++(0,1) node (01)[dot]{} --++(2,0) node (10)[dot]{} --(y);
\draw (x) --+(0.5,0.5) node (12)[dot]{} -- (01);
\draw (y) --+(-0.5,0.5) node (21)[dot]{} -- (10);
\draw (12)--(21);
\draw (x)--(b1)  (y)--(b2);

\path (b1) -- node (x)[dot][pos=0.25,yshift=-40pt]{} (b2);
\path (x) ++(2,0) node (y)[dot]{};
\draw (x) --++(0,1) node (01)[dot]{} --++(2,0) node (10)[dot]{} --(y);
\draw (x) --+(0.5,0.5) node (12)[dot]{} -- (01);
\draw (y) --+(-0.5,0.5) node (21)[dot]{} -- (10);
\draw (12)--(21);
\draw (x)--(b1)  (y)--(b2);

\end{tikzpicture}
\caption{\( G' \)}
\end{subfigure}\caption{Example of Construction~\ref{make:acyclic regular} with \( d=3 \).} \label{fig:acyclic regular}
\end{figure}

For every \( v\in V(G) \) and \( i\in \{1,2\} \), the vertex \( v^{(i)} \) has (i)~\( \deg_G(v) \) neighbours in \( G^{(i)} \), and (ii)~one neighbour in each of the \( d-\deg_G(v) \) filler gadgets attached at \( v^{(i)} \); and thus \( v^{(i)} \) has degree \( d \) in \( G' \). 
Recall that every non-terminal vertex of a filler gadget has degree \( d \) in \( G' \). 
Thus, \( G' \) is \( d \)-regular. 
\end{construct}
\noindent \emph{Remark:} In this construction, one can use any \( k \)-acyclic colourable \( d \)-regular graph in place of \( G_d \) to construct a filler gadget. 
We chose a fixed graph, namely \( G_d \), for definiteness.

\begin{proof}[Proof of Guarantee]
If \( G' \) is \( k \)-acyclic colourable, then its subgraph \( G \) is \( k \)-acyclic colourable. 
Conversely, suppose that \( G \) admits a \( k \)-acyclic colouring \( f \). 
We produce a \( k \)-colouring \( f' \) of \( G' \) as follows. 
First, colour both copies of \( G \) using \( f \). 
Next, we colour the filler gadgets. 
For every \( v\in V(G) \), (i)~choose two distinct colours \( c_1,c_2\in\{0,1,\dots,k-1\}\setminus \{f(v)\} \) and a \( k \)-acyclic colouring \( h \) of \( G_d-xy \) such that \( h(x)=c_1 \) and \( h(y)=c_2 \), and (ii)~use \( h \) to complete the colouring of the filler gadgets with terminals \( v^{(1)} \) and \( v^{(2)} \). 
If a filler gadget contains a path \( Q \) between its terminals \( v^{(1)} \) and \( v^{(2)} \), then \( f' \) uses at least three colours (namely, \( c_1, c_2 \) and \( f(v) \)) on \( Q \), and thus \( Q \) is not bicoloured by \( f' \). 
Therefore, paths such as \( Q \) cannot be part of any cycle in \( G' \) bicoloured by \( f' \). 
This ensures that \( f' \) is a \( k \)-acyclic colouring of \( G' \) since \( f' \) colours the copies of \( G \) and the copies of \( G_d-xy \) in \( G' \) using acyclic colouring schemes. 
\end{proof}

The graph \( G' \) contains two copies of \( G \) and at most \( dn \) copies of the filler gadget. 
Since \( G_d \) is a fixed graph, \( G' \) has at most \( 2n+dn\cdot O(1)=O(n) \) vertices and \( 2m+\mbox{\( dn\cdot O(1) \)}=O(m+n) \) edges, where \( m=|E(G)| \) and \( n=|V(G)| \). 
Thus, Construction~\ref{make:acyclic regular} requires only time polynomial in \( m+n \). 

Due to Theorem~\ref{thm:k-acyclic npc}, for all \( k\geq 3 \) and \( d\geq k+1 \), \textsc{\( k \)-Acyclic Colourability} is NP-complete for graphs of maximum degree \( d \). 
For \( k\geq 3 \) and \( d\leq 2k-3 \), Construction~\ref{make:acyclic regular} establishes a reduction from \textsc{\( k \)-Acyclic Colourability}(\( \Delta=d \)) to \textsc{\( k \)-Acyclic Colourability}(\( d \)-regular). 
Hence, for \( k\geq 3 \) and \( d\leq 2k-3 \), if \textsc{\( k \)-Acyclic Colourability} is NP-complete for graphs of maximum degree \( d \), then \textsc{\( k \)-Acyclic Colourability} is NP-complete for \( d \)-regular graphs. 
Clearly, if \textsc{\( k \)-Acyclic Colourability} is NP-complete for \( d \)-regular graphs, then \textsc{\( k \)-Acyclic Colourability} is NP-complete for graphs of maximum degree \( d \). 
Thus, we have the following theorem.
\begin{theorem}\label{thm:acyclic bdd degree to regular}
For all \( k\geq 3 \) and \( d\leq 2k-3 \), \textsc{\( k \)-Acyclic Colourability} is NP-complete for graphs of maximum degree \( d \) if and only if \textsc{\( k \)-Acyclic Colourability} is NP-complete for \( d \)-regular graphs. 
In particular, for \( k\geq 4 \) and \( k+1\leq d\leq 2k-3 \), \textsc{\( k \)-Acyclic Colourability} is NP-complete for \( d \)-regular graphs. 
\qed
\end{theorem}

A modification of Construction~\ref{make:acyclic regular} gives the following theorem. 
\begin{theorem}\label{thm:k-acyclic bdd deg to regular}
For all \( k\geq 3 \) and \( d\leq 2k-3 \), \textsc{\( k \)-Acyclic Colourability} is NP-complete for bipartite graphs of maximum degree \( d \) if and only if \textsc{\( k \)-Acyclic Colourability} is NP-complete for \( d \)-regular bipartite graphs. 
In particular, for \( k\geq 4 \) and \( k+1\leq d\leq 2k-3 \), \textsc{\( k \)-Acyclic Colourability} is NP-complete for \( d \)-regular bipartite graphs. 
\end{theorem}
\begin{proof}
First, we prove that for \( k\geq 3 \) and \( d\leq 2k-3 \), if \textsc{\( k \)-Acyclic Colourability} is NP-complete for bipartite graphs of maximum degree \( d \), then \textsc{\( k \)-Acyclic Colourability} is NP-complete for \( d \)-regular bipartite graphs. 
As remarked earlier, any \( k \)-acyclic colourable \( d \)-regular graph \( H \) can replace the graph \( G_d \) in Construction~\ref{make:acyclic regular} without affecting the guarantee in the construction. 
Hence, it is enough to prove the following claim to complete the proof.\\[3pt]
\noindent \textit{Claim:} \( G' \) is bipartite if \( G \) and \( H \) are bipartite.\\[3pt]
To prove the claim, assume that \( G \) and \( H \) are bipartite. 
Let \( f \) be a 2-colouring of \( G \) and \( h \) be a 2-colouring of \( H \). 
Since \( xy \) is an edge in \( H \), we have \( h(x)\neq h(y) \). 
Without loss of generality, assume that \( h(x)=0 \) and \( h(y)=1 \). 
It suffices to produce a 2-colouring \( f' \) of \( G' \). 
For each \( v\in V(G) \), assign \( f'(v^{(1)})=f(v) \) and \( f'(v^{(2)})=1-f(v) \). 
\begin{figure}[hbt]
\centering
\begin{subfigure}[b]{0.5\textwidth}
\centering
\begin{tikzpicture}[scale=0.85]
\path (0,0) node(x)[dot][label=below:\( x \)][label={[vcolour]0}]{} ++(3,0) node(y)[dot][label=below:\( y \)][label={[vcolour]1}]{};
\path (x) --node[draw,ellipse,dotted,minimum width=3.75cm,minimum height=01.5cm,yshift=0.1cm][label=below:\( H-xy \)]{}  (y);

\draw
(x)--+(-2,0) node[dot]{} node[terminal][label=below:\( v^{(1)} \)][label={[vcolour]1}]{}
(y)--+(2,0) node[dot]{} node[terminal][label=below:\( v^{(2)} \)][label={[vcolour]0}]{};
\end{tikzpicture}
\caption{When \( f(v)=1 \)}
\label{fig:regular filler gadget acyclic colouring scheme1}
\end{subfigure}\begin{subfigure}[b]{0.5\textwidth}
\centering
\begin{tikzpicture}[scale=0.85]
\path (0,0) node(x)[dot][label=below:\( x \)][label={[vcolour]1}]{} ++(3,0) node(y)[dot][label=below:\( y \)][label={[vcolour]0}]{};
\path (x) --node[draw,ellipse,dotted,minimum width=3.75cm,minimum height=01.5cm,yshift=0.1cm][label=below:\( H-xy \)]{}  (y);

\draw
(x)--+(-2,0) node[dot]{} node[terminal][label=below:\( v^{(1)} \)][label={[vcolour]0}]{}
(y)--+(2,0) node[dot]{} node[terminal][label=below:\( v^{(2)} \)][label={[vcolour]1}]{};
\end{tikzpicture}
\caption{When \( f(v)=0 \)}
\label{fig:regular filler gadget acyclic colouring scheme2}
\end{subfigure}\caption[2-colouring schemes for filler gadget when \( H \) is bipartite]{2-colouring schemes for filler gadget when \( H \) is bipartite. For each vertex \( z \) in \( H-xy \), scheme~(a) assigns colour \( h(z) \) whereas scheme~(b) assigns colour \( 1-h(z) \).}
\label{fig:regular filler gadget acyclic colouring}
\end{figure}
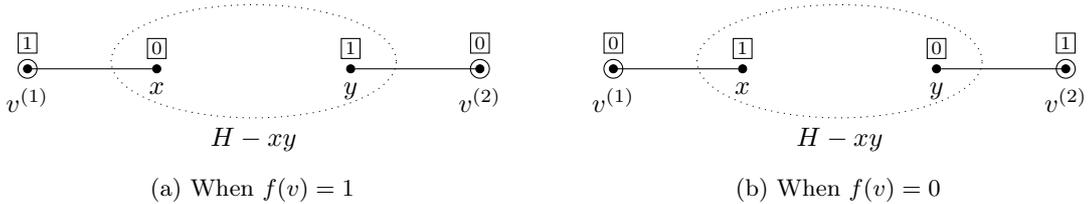
Next, for each \( v\in V(G) \), we colour the filler gadgets with terminals \( v^{(1)} \) and \( v^{(2)} \). 
If \( f'(v^{(1)})=1 \), colour the filler gadget by the scheme in Figure~\ref{fig:regular filler gadget acyclic colouring scheme1}; otherwise, use the scheme in Figure~\ref{fig:regular filler gadget acyclic colouring scheme2}. 
It is easy to verify that \( f' \) is indeed a 2-colouring of \( G' \).

We know that for \( k\geq 4 \) and \( d\geq k+1 \), \textsc{\( k \)-Acyclic Colourability} is NP-complete for bipartite graphs of maximum degree \( d \). 
As a result, for \( k\geq 4 \) and \( k+1\leq d\leq 2k-3 \), \textsc{\( k \)-Acyclic Colourability} is NP-complete for \( d \)-regular bipartite graphs. 
\end{proof}

\subsection[Results on \( L_a^{(k)} \)]{\boldmath Results on \( L_a^{(k)} \)\label{sec:acyclic colouring points of hardness transition}
} Recall that for \( k\geq 3 \), \( L_a^{(k)} \) is the least integer \( d \) such that \textsc{\( k \)-Acyclic Colourability} in graphs of maximum degree \( d \) is NP-complete. 
Bear in mind that we assume P \( \neq \) NP throughout this paper; thus, NP is partitioned into three classes: P, NPC and NPI~\cite{paschos}.
If a problem in NP is not NP-complete (i.e., not in NPC), then it is either in P or in NPI. 
By the definition of \( L_a^{(k)} \), \textsc{\( k \)-Acyclic Colourability}(\( \Delta=d \)) is not NP-complete for \( d<L_a^{(k)} \), which means that the problem is either in P or in NPI (we do not know which is the case).

Theorem~\ref{thm:k-acyclic npc} proved that for \( k\geq 3 \), \textsc{\( k \)-Acyclic Colourability} is NP-complete for graphs of maximum degree \( k+1 \), and thus \( L_a^{(k)}\leq k+1 \). 
It is easy to observe that for \( d\leq 2 \), the acyclic chromatic number of a graph of maximum degree \( d \) can be computed in polynomial time. 
Hence, \( L_a^{(k)}\geq 3 \) for all \( k\geq 3 \). 
Next, we show that \( 0.38\, k^{\,3/4}<L_a^{(k)}\leq k+1 \) for all \( k\geq 3 \). 
\begin{observation}
For \( d\leq 0.38\, k^{\,3/4} \), \textsc{\( k \)-Acyclic Colourability} is polynomial-time solvable for graphs of maximum degree \( d\). 
Hence, \( L_a^{(k)}>0.38\, k^{\,3/4} \) for all \( k\geq 3 \). 
\end{observation}
\begin{proof}
\iftoggle{extended}
{The observation is trivially true for \( d\leq 2 \). 
Suppose that \( d\geq 3 \). 
Hence, \( d<0.693\, d^{4/3} \). 
Substituting this in the bound of Sereni and Volec~\cite{sereni_volec} gives \( \chi_a(G)<(2.835+0.693)d^{4/3}=3.528\,d^{4/3} \) for every graph \( G \) with maximum degree \( d \). 
Thus, when \( k\geq 3.528\,d^{4/3} \), a graph of maximum degree \( d \) is always \( k \)-acyclic colourable. 
Therefore, for \( d\leq (3.528)^{-3/4}k^{3/4}\leq 0.38\,k^{3/4} \), \textsc{\( k \)-Acyclic Colourability} is polynomial-time solvable for graphs of maximum degree \( d\). 
}{The observation is trivially true for \( d\leq 2 \). 
It suffices to prove the observation for \( d\geq 3 \). 
Suppose that \( d\geq 3 \). 
Sereni and Volec~\cite{sereni_volec} proved that \( \chi_a(G)<2.835\,d^{\,4/3}+d \) for every graph \( G \) of maximum degree \( d \). 
Since \( d\geq 3 \), we have \( d^{\,1/3}\geq 3^{1/3}>1/0.694 \), and thus \( d<0.694\, d^{\,4/3} \). 
Thus, \( \chi_a(G)<(2.835+0.694)d^{\,4/3}=3.529\,d^{\,4/3} \) for every graph \( G \) of maximum degree \( d \). 
Hence, when \( k\geq 3.529\,d^{\,4/3} \), every graph of maximum degree \( d \) is \( k \)-acyclic colourable. 
In other words, if \( d\leq (3.529)^{-3/4}k^{\,3/4} \), then every graph of maximum degree \( d \) is \( k \)-acyclic colourable. 
Note that \( 0.38<(3.529)^{-3/4} \). 
Hence, if \( d\leq 0.38 k^{3/4} \), then \( d\leq (3.529)^{-3/4}k^{3/4} \). 
Therefore, for \( d\leq 0.38\,k^{3/4} \), every graph of maximum degree \( d \) is \( k \)-acyclic colourable, and thus \textsc{\( k \)-Acyclic Colourability} is polynomial-time solvable for graphs of maximum degree \( d \). 
}As a result, \( L_a^{(k)}>0.38\,k^{3/4} \) for all \( k\geq 3 \). 
\end{proof}

By Corollary~\ref{cor:lb chi_a}, \( \chi_a(G)\geq \lceil (d+3)/2\rceil \) for every \( d \)-regular graph \( G \). 
Hence, for \( k\geq 3 \), \textsc{\( k \)\nobreakdash-Acyclic Colourability} in \( d \)-regular graphs is polynomial-time solvable for each \mbox{\( d\geq  2k-2 \)} (because the answer is always `no').

Fix an integer \( k\geq 4 \). 
Theorem~\ref{thm:k-acyclic bdd deg to regular} proved that for \( d\leq 2k-3 \), \textsc{\( k \)-Acyclic Colourability} in graphs of maximum degree \( d \) is NP-complete if and only if \textsc{\( k \)-Acyclic Colourability} in \( d \)\nobreakdash-regular graphs is NP-complete. 
By the definition of \( L_a^{(k)} \), \textsc{\( k \)-Acyclic Colourability} in graphs of maximum degree \( d \) is NP-complete for \( d=L_a^{(k)} \), and not NP-complete for \( d<L_a^{(k)} \). 
Hence, for \( d<L_a^{(k)} \), \textsc{\( k \)-Acyclic Colourability} in \( d \)-regular graphs is not NP-complete by Theorem~\ref{thm:k-acyclic bdd deg to regular}. 
We know that \textsc{\( k \)-Acyclic Colourability} in graphs of maximum degree \( d \) is NP-complete for \( d\geq L_a^{(k)} \). 
As a result, for \( d \) in the range \( L_a^{(k)}\leq d\leq 2k-3 \), \textsc{\( k \)-Acyclic Colourability} in \( d \)-regular graphs is also NP-complete by Theorem~\ref{thm:k-acyclic bdd deg to regular}. 
Moreover, for \( d\geq  2k-2 \), \textsc{\( k \)\nobreakdash-Acyclic Colourability} in \( d \)-regular graphs is polynomial-time solvable (see the previous paragraph). 
Thus, we have the following theorem. 
\begin{theorem}\label{thm:k acyclic regular}
For \( k\geq 4 \), \textsc{\( k \)-Acyclic Colourability} in \( d \)-regular graphs is NP-complete if and only if \( L_a^{(k)}\leq d\leq 2k-3 \).
\qed
\end{theorem}

\section{Unique Acyclic Colouring}\label{sec:unique soln}
In this section, we borrow gadgets from Construction~\ref{make:acyclic} to obtain results on unique acyclic colouring. 
See Section~\ref{sec:intro unique solution problems} for definitions of problems related to unique colouring and unique acyclic colouring. 
We prove that for all \( k\geq 3 \), \textsc{Another \( k \)-Acyclic Colouring}\,[\( \mathcal{R}_{\text{swap}} \)] and \textsc{Another \( k \)-Acyclic Colouring}\,[\( \mathcal{R}_{\text{swap+auto}} \)] are NP-complete and thus the corresponding unique solution problems are coNP-hard. 
We also show that \textsc{Unique 3-Acyclic Colouring}\,[\( \mathcal{R}_{\text{swap+auto}} \)] is coNP-hard for the class of bipartite graphs of maximum degree~4. 

We start with a simple construction that enables us to transform \textsc{Unique 3-Colouring}\,[\( \mathcal{R}_{\text{swap}} \)] to \textsc{Unique 3-Acyclic Colouring}\,[\( \mathcal{R}_{\text{swap}} \)]. 
Observation~\ref{obs:link in chain} proved that the biclique \( K_{k-1,k} \) has a unique \( k \)-acyclic colouring up to colour swaps. In particular, \( K_{2,3} \) has a unique 3-acyclic colouring up to colour swaps. 
The following construction makes use of this. 

\begin{construct}\label{make:3-acyclic R_swap max degree 24}
\emph{Input:} A graph \( G \) of maximum degree 8.\\
\emph{Output:} A 2-degenerate bipartite graph \( G' \) of maximum degree~24.\\
\emph{Guarantee:} The number of \( 3 \)-colourings of \( G \) up to colour swaps equals the number of \( 3 \)-acyclic colourings of \( G' \) up to colour swaps.\\
\emph{Steps:}\\
Replace each edge \( e=uv \) of \( G \) by a copy of the complete bipartite graph \( K_{2,3} \) with parts \( \{u,v\} \) and \( \{e_1,e_2,e_3\} \) where \( e_1,e_2 \) and \( e_3 \) are newly introduced vertices. 
To produce a 2-degenerate ordering of \( V(G') \), list the new vertices \( e_i \) followed by the members of \( V(G) \).
\end{construct}
\begin{proof}[Proof of Guarantee]
For each 3-colouring \( f \) of \( G \) that uses colours 0,1 and 2, there exists a unique 3-colouring extension of \( f \) into \( V(G') \). 
The extension is unique because for each edge \( e=uv \) of \( G \), exactly one colour, namely the unique colour in \( \{0,1,2\}\setminus \{f(u),f(v)\} \), is available for \( e_1,e_2 \) and \( e_3 \). 
Let \( \phi \) be the function that maps each 3\nobreakdash-colouring of \( G \) to its unique 3-colouring extension into \( V(G') \). 
Clearly, \( \phi \) is a function from the set of 3-colourings of \( G \) to the set of 3-colourings of \( G' \), and \( \phi \) is one-one.

We claim that \( \text{Range}(\phi) \) is precisely the set of 3-acyclic colourings of \( G' \). 
For every 3-colouring \( f \) of \( G \), we know that \( \phi(f) \) is a 3-acyclic colouring of \( G' \) because every cycle in \( G' \) contains a path of the form \( u,e_i,v \) where \( u,v\in V(G) \) and \( uv\in E(G) \), and such paths are tricoloured by \( \phi(f) \) since \( \phi(f)(u)=f(u)\neq f(v)=\phi(f)(v) \). 
Moreover, each 3-acyclic colouring of the complete bipartite graph \( K_{2,3} \) with parts \( \{u,v\} \) and \( \{e_1,e_2,e_3\} \) assigns different colours to \( u \) and \( v \) by the special case \( k=3 \) of Observation~\ref{obs:K_2,k}. 
Thus, every 3-acyclic colouring \( f' \) of \( G' \) has a preimage under \( \phi \), namely the restriction of \( f' \) to \( V(G) \). 
This proves that \( \phi \) is onto. 
Thus, there exists a one-one function \( \phi \) from the set of 3-colourings of \( G \) onto the set of 3-acyclic colourings of \( G' \). 
Furthermore, two 3-colourings \( f_1 \) and \( f_2 \) of \( G \) are non-equivalent under colour swaps if and only if the 3-acyclic colourings \( \phi(f_1) \) and \( \phi(f_2) \) of \( G' \) are non-equivalent under colour swaps.
\end{proof}

\begin{theorem}\label{thm:3-acyclic R_swap max degree 24}
For 2-degenerate bipartite graphs of maximum degree~24, \textsc{Another 3-Acyclic Colouring}\,\( [\mathcal{R}_{\text{swap}}] \) is NP-complete and \textsc{Unique 3-Acyclic Colouring}\,\( [\mathcal{R}_{\text{swap}}] \) is coNP-hard.
\end{theorem}
\begin{proof}
The reduction is from \textsc{Another 3-Colouring}\,[\( \mathcal{R}_{\text{swap}} \)]\,(\( \Delta=8 \)). 
Let \( (G,f) \) be an instance of the source problem. 
From \( G \), produce a graph \( G' \) by Construction~\ref{make:3-acyclic R_swap max degree 24}. 
In Construction~\ref{make:3-acyclic R_swap max degree 24}, it is established that there is a bijection \( \phi \) from the set of 3-colourings of \( G \) to the set of 3-acyclic colourings of \( G' \). 
In particular, \( f'=\phi(f) \) is a 3-acyclic colouring of \( G' \). 
By the guarantee in Construction~\ref{make:3-acyclic R_swap max degree 24}, the number of 3-colourings of \( G \) up to colour swaps is equal to the number of 3-acyclic colourings of \( G' \) up to colour swaps. 
Therefore, \( (G,f) \) is a yes instance of \textsc{Another 3-Colouring}\,[\( \mathcal{R}_{\text{swap}} \)] if and only if \( (G',f') \) is a yes instance of \textsc{Another 3-Acyclic Colouring}\,[\( \mathcal{R}_{\text{swap}} \)]. 
This proves that \textsc{Another 3-Acyclic Colouring}\,[\( \mathcal{R}_{\text{swap}} \)] is NP-complete for 2\nobreakdash-degenerate bipartite graphs of maximum degree~24, and thus the problem \textsc{Unique 3-Acyclic Colouring}\,[\( \mathcal{R}_{\text{swap}} \)] is coNP-hard for the same class.
\end{proof}

Next, we show that \textsc{Another \( k \)-Acyclic Colouring}\,[\( \mathcal{R}_{\text{swap}} \)] is NP-complete and \textsc{Unique \( k \)-Acyclic Colouring}\,[\( \mathcal{R}_{\text{swap}} \)] in coNP-hard for all \( k\geq 3 \). 
Let \( G \) be a graph, and let \( G' \) be the graph obtained by adding a universal vertex to \( G \); that is, \( G' \) is the graph join of \( G \) and \( K_1 \). 
Fix an integer \( k\geq 3 \). 
Clearly, \( G \) is \( k \)-acyclic colourable if and only if \( G' \) is \( (k+1) \)-acyclic colourable. 
Moreover, \( G \) admits two \( k \)-acyclic colourings \( f_1 \) and \( f_2 \) non-equivalent up to colour swaps (i.e., \( (f_1,f_2)\notin \mathcal{R}_{\text{swap}}(G,k) \)) if and only if \( G' \) admits two \( (k+1) \)-acyclic colourings \( f_1' \) and \( f_2' \) non-equivalent up to colour swaps (i.e., \( (f_1',f_2')\notin \mathcal{R}_{\text{swap}}(G',k+1) \)). 
Thus, \( G \) admits a unique \( k \)-acyclic colouring up to colour swaps if and only if \( G' \) admits a unique \( (k+1) \)-acyclic colouring up to colour swaps. 
Hence, for \( k\geq 3 \), the transformation from \( (G,k) \) to \( (G',k+1) \) establishes a reduction from \textsc{Another \( k \)-Acyclic Colouring}\,[\( \mathcal{R}_{\text{swap}} \)] to \textsc{Another \( (k+1) \)-Acyclic Colouring}\,[\( \mathcal{R}_{\text{swap}} \)]. 
Thus, we have the following theorem by Theorem~\ref{thm:3-acyclic R_swap max degree 24} and induction. 
\begin{theorem}\label{thm:k-acyclic R_swap}
For \( k\geq 3 \), \textsc{Another \( k \)-Acyclic Colouring}\,\( [\mathcal{R}_{\text{swap}}] \) is NP-complete and \textsc{Unique \( k \)-Acyclic Colouring}\,\( [\mathcal{R}_{\text{swap}}] \) is coNP-hard. 
\qed
\end{theorem}

Finally, we prove that \textsc{Another \( k \)-Acyclic Colouring}\,[\( \mathcal{R}_{\text{swap+auto}} \)] is NP-complete for all \( k\geq 3 \). 
We prove this for \( k=3 \) first by Construction~\ref{make:3-acyclic R_swap_n_auto} below, which establishes a reduction from \textsc{Another 3-Colouring}\,[\( \mathcal{R}_{\text{swap}} \)]\,(\( \Delta=8 \)) to the problem \textsc{Another 3-Acyclic Colouring}\,[\( \mathcal{R}_{\text{swap+auto}} \)]\,(\( \Delta=4 \)). 
Construction~\ref{make:3-acyclic R_swap_n_auto} is a slight modification of Construction~\ref{make:acyclic}.

\begin{construct}\label{make:3-acyclic R_swap_n_auto}
\emph{Input:} A graph \( G \) of maximum degree~8.\\
\emph{Output:} A bipartite graph \( G^* \) of maximum degree~4.\\
\emph{Guarantee:} \( G \) has a unique 3-colouring up to colour swaps if and only if \( G^* \) has a unique 3-acyclic colouring up to colour swaps and automorphisms.\\
\emph{Steps:}\\
Replace each vertex \( v \) of \( G \) by a chain gadget with \( 3\deg_G(v)+\lambda(v) \) terminals, where \( \lambda\colon V(G)\to \mathbb{N} \) is defined in such a way that no two chain gadgets have the same number of terminals (one way to ensure this is to choose an ordering \( v_1,v_2,\dots,v_n \) of the vertex set of \( G \) such that \( \deg_G(v_1)\leq \deg_G(v_2)\leq \dots \leq \deg_G(v_n) \), and define \( \lambda(v_i)=i \) for \( 1\leq i\leq n \)). 
For each \( v\in V(G) \) and each neighbour \( u \) of \( v \), the chain gadget for \( v \) (let us call it \( \text{chain}(v) \)) has three terminals reserved for \( u \), which we shall call as \( v_{u1} \), \( v_{u2} \) and \( v_{u3} \). 
For each edge \( e=uv \) of \( G \) and each \( j\in\{1,2,3\} \), introduce a new vertex \( e_j \) in \( G^* \) and join \( e_j \) to \( u_{vj} \) as well as \( v_{uj} \). 
An example of the construction is shown in Figure~\ref{fig:eg 3-acyclic R_swap_n_auto}. 
\end{construct}

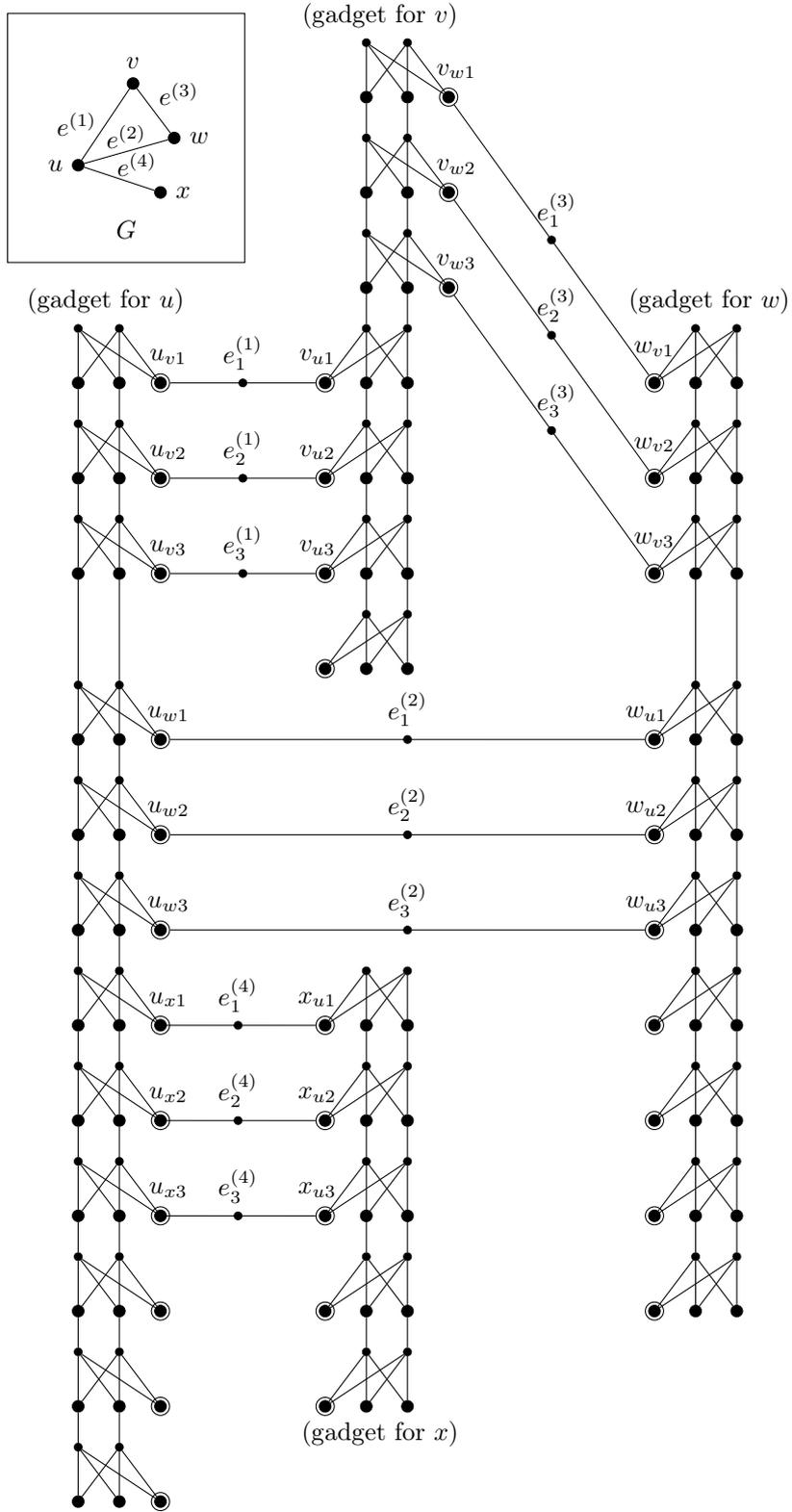
\begin{figure}[hbtp]
\centering 
\begin{tikzpicture}[scale=0.75]
\tikzset{
biggerdot/.style={dot,minimum size = 4.5pt}
}

\path (0,3) node (v1)[biggerdot][label=left:\( u \)]{};
\path (v1) --+(1,1.5) node (v2)[biggerdot][label=\( v \)]{};
\path (v1) --+(1.75,0.5) node (v3)[biggerdot][label=right:\( w \)]{}--(v2);
\path (v1) --+(1.5,-0.5) node (v4)[biggerdot][label=right:\( x \)]{};
\node [fit={(v1) (v2) (v3) (v4)},draw,inner sep=25pt][label={[yshift=20pt]-90:\( G \)}] {};

\draw (v1) --node[left]{\( e^{(1)} \)} (v2);
\draw (v1) --node[above=-2pt]{\( e^{(2)} \)} (v3);
\draw (v3) --node[above right=-2pt]{\( e^{(3)} \)} (v2);
\draw (v1) --node[pos=0.75,above]{\( e^{(4)} \)} (v4);

\coordinate (level1);
\path (level1)
++(0,-1) coordinate (level2)
++(0,-0.75) coordinate (level3)
++(0,-1) coordinate (level4)
++(0,-0.75) coordinate (level5)
++(0,-1) coordinate (level6)
++(0,-2.05) coordinate (level7)
++(0,-1) coordinate (level8)
++(0,-0.75) coordinate (level9)
++(0,-1) coordinate (level10)
++(0,-0.75) coordinate (level11)
++(0,-1) coordinate (level12)
++(0,-0.75) coordinate (level13) coordinate (ogLevel13)
++(0,-1) coordinate (level14)
++(0,-0.75) coordinate (level15)
++(0,-1) coordinate (level16)
++(0,-0.75) coordinate (level17)
++(0,-1) coordinate (level18)
++(0,-0.75) coordinate (level19)
++(0,-1) coordinate (level20)
++(0,-0.75) coordinate (level21)
++(0,-1) coordinate (level22)
++(0,-0.75) coordinate (level23)
++(0,-1) coordinate (level24)
;

\path (level1) node(11)[dot]{}++(0.75,0) node(12)[dot]{};
\path (level2) node(21)[biggerdot]{}++(0.75,0) node(22)[biggerdot]{}++(0.75,0) node(2k)[biggerdot]{} node(v11)[terminal][label={[xshift=3pt]\( u_{v1} \)}]{};
\path (level3) node(31)[dot]{}++(0.75,0) node(32)[dot]{};
\path (level4) node(41)[biggerdot]{}++(0.75,0) node(42)[biggerdot]{}++(0.75,0) node(4k)[biggerdot]{} node(v12)[terminal][label={[xshift=3pt]\( u_{v2} \)}]{};
\path (level5) node(51)[dot]{}++(0.75,0) node(52)[dot]{};
\path (level6) node(61)[biggerdot]{}++(0.75,0) node(62)[biggerdot]{}++(0.75,0) node(6k)[biggerdot]{} node(v13)[terminal][label={[xshift=3pt]\( u_{v3} \)}]{} coordinate(ogv13);
\path (level7) node(71)[dot]{}++(0.75,0) node(72)[dot]{};
\path (level8) node(81)[biggerdot]{}++(0.75,0) node(82)[biggerdot]{}++(0.75,0) node(8k)[biggerdot]{} node(v14)[terminal][label={[xshift=3pt]\( u_{w1} \)}]{} coordinate(ogv14);
\path (level9) node(91)[dot]{}++(0.75,0) node(92)[dot]{};
\path (level10) node(101)[biggerdot]{}++(0.75,0) node(102)[biggerdot]{}++(0.75,0) node(10k)[biggerdot]{} node(v15)[terminal][label={[xshift=3pt]\( u_{w2} \)}]{} coordinate(ogv15);
\path (level11) node(111)[dot]{}++(0.75,0) node(112)[dot]{};
\path (level12) node(121)[biggerdot]{}++(0.75,0) node(122)[biggerdot]{}++(0.75,0) node(12k)[biggerdot]{} node(v16)[terminal][label={[xshift=3pt]\( u_{w3} \)}]{} coordinate(ogv16);
\path (level13) node(131)[dot]{}++(0.75,0) node(132)[dot]{};
\path (level14) node(141)[biggerdot]{}++(0.75,0) node(142)[biggerdot]{}++(0.75,0) node(14k)[biggerdot]{} node(v17)[terminal][label={[xshift=3pt]\( u_{x1} \)}]{} coordinate(ogv17);
\path (level15) node(151)[dot]{}++(0.75,0) node(152)[dot]{};
\path (level16) node(161)[biggerdot]{}++(0.75,0) node(162)[biggerdot]{}++(0.75,0) node(16k)[biggerdot]{} node(v18)[terminal][label={[xshift=3pt]\( u_{x2} \)}]{} coordinate(ogv18);
\path (level17) node(171)[dot]{}++(0.75,0) node(172)[dot]{};
\path (level18) node(181)[biggerdot]{}++(0.75,0) node(182)[biggerdot]{}++(0.75,0) node(18k)[biggerdot]{} node(v19)[terminal][label={[xshift=3pt]\( u_{x3} \)}]{} coordinate(ogv19);

\path (level19) node(191)[dot]{}++(0.75,0) node(192)[dot]{};
\path (level20) node(201)[biggerdot]{}++(0.75,0) node(202)[biggerdot]{}++(0.75,0) node(20k)[biggerdot]{} node(v110)[terminal]{} coordinate(ogv20);
\path (level21) node(211)[dot]{}++(0.75,0) node(212)[dot]{};
\path (level22) node(221)[biggerdot]{}++(0.75,0) node(222)[biggerdot]{}++(0.75,0) node(22k)[biggerdot]{} node(v110)[terminal]{} coordinate(ogv22);
\path (level23) node(231)[dot]{}++(0.75,0) node(232)[dot]{};
\path (level24) node(241)[biggerdot]{}++(0.75,0) node(242)[biggerdot]{}++(0.75,0) node(24k)[biggerdot]{} node(v110)[terminal]{} coordinate(ogv24);

\draw
(11) -- (21)
(11) -- (22)
(11) -- (2k);
\draw
(12) -- (21)
(12) -- (22)
(12) -- (2k);
\draw (21)--(31)  (22)--(32);
\draw
(31) -- (41)
(31) -- (42)
(31) -- (4k);
\draw
(32) -- (41)
(32) -- (42)
(32) -- (4k);
\draw (41)--(51)  (42)--(52);
\draw
(51) -- (61)
(51) -- (62)
(51) -- (6k);
\draw
(52) -- (61)
(52) -- (62)
(52) -- (6k);
\draw (61)--(71)  (62)--(72);
\draw
(71) -- (81)
(71) -- (82)
(71) -- (8k);
\draw
(72) -- (81)
(72) -- (82)
(72) -- (8k);
\draw (81)--(91)  (82)--(92);
\draw
(91) -- (101)
(91) -- (102)
(91) -- (10k);
\draw
(92) -- (101)
(92) -- (102)
(92) -- (10k);
\draw (101)--(111)  (102)--(112);
\draw
(111) -- (121)
(111) -- (122)
(111) -- (12k);
\draw
(112) -- (121)
(112) -- (122)
(112) -- (12k);
\draw (121)--(131)  (122)--(132);
\draw
(131) -- (141)
(131) -- (142)
(131) -- (14k);
\draw
(132) -- (141)
(132) -- (142)
(132) -- (14k);
\draw (141)--(151)  (142)--(152);
\draw
(151) -- (161)
(151) -- (162)
(151) -- (16k);
\draw
(152) -- (161)
(152) -- (162)
(152) -- (16k);
\draw (161)--(171)  (162)--(172);
\draw 
(171) -- (181)
(171) -- (182)
(171) -- (18k);
\draw 
(172) -- (181)
(172) -- (182)
(172) -- (18k);
\draw (181)--(191)  (182)--(192);
\draw
(191) -- (201)
(191) -- (202)
(191) -- (20k);
\draw
(192) -- (201)
(192) -- (202)
(192) -- (20k);
\draw (201)--(211)  (202)--(212);
\draw
(211) -- (221)
(211) -- (222)
(211) -- (22k);
\draw
(212) -- (221)
(212) -- (222)
(212) -- (22k);
\draw (221)--(231)  (222)--(232);
\draw
(231) -- (241)
(231) -- (242)
(231) -- (24k);
\draw
(232) -- (241)
(232) -- (242)
(232) -- (24k);

\path (level1) +(0.5,0.5) node{(gadget for \( u \))};

\path (level1) ++(6,5.25) coordinate (level1);
\path (level1)
++(0,-1) coordinate (level2) 
++(0,-0.75) coordinate (level3) 
++(0,-1) coordinate (level4) 
++(0,-0.75) coordinate (level5) 
++(0,-1) coordinate (level6) 
++(0,-0.75) coordinate (level7) 
++(0,-1) coordinate (level8) 
++(0,-0.75) coordinate (level9) 
++(0,-1) coordinate (level10) 
++(0,-0.75) coordinate (level11) 
++(0,-1) coordinate (level12)
++(0,-0.75) coordinate (level13)
++(0,-1) coordinate (level14);

\path (level1) node(11)[dot]{}++(-0.75,0) node(12)[dot]{};
\path (level2) node(21)[biggerdot]{}++(-0.75,0) node(22)[biggerdot]{}++(1.5,0) node(2k)[biggerdot]{} node(v21)[terminal][label={[xshift=3pt]\( v_{w1} \)}]{};
\path (level3) node(31)[dot]{}++(-0.75,0) node(32)[dot]{};
\path (level4) node(41)[biggerdot]{}++(-0.75,0) node(42)[biggerdot]{}++(1.5,0) node(4k)[biggerdot]{} node(v22)[terminal][label={[xshift=3pt]\( v_{w2} \)}]{};
\path (level5) node(51)[dot]{}++(-0.75,0) node(52)[dot]{};
\path (level6) node(61)[biggerdot]{}++(-0.75,0) node(62)[biggerdot]{}++(1.5,0) node(6k)[biggerdot]{} node(v23)[terminal][label={[xshift=3pt]\( v_{w3} \)}]{};
\path (level7) node(71)[dot]{}++(-0.75,0) node(72)[dot]{};
\path (level8) node(81)[biggerdot]{}++(-0.75,0) node(82)[biggerdot]{}++(-0.75,0) node(8k)[biggerdot]{} node(v24)[terminal][label={[xshift=-3pt]\( v_{u1} \)}]{};
\path (level9) node(91)[dot]{}++(-0.75,0) node(92)[dot]{};
\path (level10) node(101)[biggerdot]{}++(-0.75,0) node(102)[biggerdot]{}++(-0.75,0) node(10k)[biggerdot]{} node(v25)[terminal][label={[xshift=-3pt]\( v_{u2} \)}]{};
\path (level11) node(111)[dot]{}++(-0.75,0) node(112)[dot]{};
\path (level12) node(121)[biggerdot]{}++(-0.75,0) node(122)[biggerdot]{}++(-0.75,0) node(12k)[biggerdot]{} node(v26)[terminal][label={[xshift=-3pt]\( v_{u3} \)}]{};

\path (level13) node(131)[dot]{}++(-0.75,0) node(132)[dot]{};
\path (level14) node(141)[biggerdot]{}++(-0.75,0) node(142)[biggerdot]{}++(-0.75,0) node(14k)[biggerdot]{} node(v27)[terminal]{}; 

\draw
(11) -- (21)
(11) -- (22)
(11) -- (2k);
\draw
(12) -- (21)
(12) -- (22)
(12) -- (2k);
\draw (21)--(31)  (22)--(32);
\draw
(31) -- (41)
(31) -- (42)
(31) -- (4k);
\draw
(32) -- (41)
(32) -- (42)
(32) -- (4k);
\draw (41)--(51)  (42)--(52);
\draw
(51) -- (61)
(51) -- (62)
(51) -- (6k);
\draw
(52) -- (61)
(52) -- (62)
(52) -- (6k);
\draw (61)--(71)  (62)--(72);
\draw
(71) -- (81)
(71) -- (82)
(71) -- (8k);
\draw
(72) -- (81)
(72) -- (82)
(72) -- (8k);
\draw (81)--(91)  (82)--(92);
\draw
(91) -- (101)
(91) -- (102)
(91) -- (10k);
\draw
(92) -- (101)
(92) -- (102)
(92) -- (10k);
\draw (101)--(111)  (102)--(112);
\draw
(111) -- (121)
(111) -- (122)
(111) -- (12k);
\draw
(112) -- (121)
(112) -- (122)
(112) -- (12k);
\draw (121)--(131)  (122)--(132);
\draw
(131) -- (141)
(131) -- (142)
(131) -- (14k);
\draw
(132) -- (141)
(132) -- (142)
(132) -- (14k);

\path (level1) +(-0.5,0.5) node{(gadget for \( v \))};

\path (level1) ++(6,-5.25) coordinate (level1);
\path (level1)
++(0,-1) coordinate (level2) 
++(0,-0.75) coordinate (level3) 
++(0,-1) coordinate (level4) 
++(0,-0.75) coordinate (level5) 
++(0,-1) coordinate (level6) 
++(0,-2.05) coordinate (level7) 
++(0,-1) coordinate (level8) 
++(0,-0.75) coordinate (level9) 
++(0,-1) coordinate (level10) 
++(0,-0.75) coordinate (level11) 
++(0,-1) coordinate (level12)
++(0,-0.75) coordinate (level13)
++(0,-1) coordinate (level14)
++(0,-0.75) coordinate (level15)
++(0,-1) coordinate (level16)
++(0,-0.75) coordinate (level17)
++(0,-1) coordinate (level18)
++(0,-0.75) coordinate (level19)
++(0,-1) coordinate (level20);

\path (level1) node(11)[dot]{}++(-0.75,0) node(12)[dot]{};
\path (level2) node(21)[biggerdot]{}++(-0.75,0) node(22)[biggerdot]{}++(-0.75,0) node(2k)[biggerdot]{} node(v31)[terminal][label={[yshift=3pt]\( w_{v1} \)}]{};
\path (level3) node(31)[dot]{}++(-0.75,0) node(32)[dot]{};
\path (level4) node(41)[biggerdot]{}++(-0.75,0) node(42)[biggerdot]{}++(-0.75,0) node(4k)[biggerdot]{} node(v32)[terminal][label={[yshift=3pt]\( w_{v2} \)}]{};
\path (level5) node(51)[dot]{}++(-0.75,0) node(52)[dot]{};
\path (level6) node(61)[biggerdot]{}++(-0.75,0) node(62)[biggerdot]{}++(-0.75,0) node(6k)[biggerdot]{} node(v33)[terminal][label={[yshift=3pt]\( w_{v3} \)}]{};
\path (level7) node(71)[dot]{}++(-0.75,0) node(72)[dot]{};
\path (level8) node(81)[biggerdot]{}++(-0.75,0) node(82)[biggerdot]{}++(-0.75,0) node(8k)[biggerdot]{} node(v34)[terminal][label={[xshift=-3pt]\( w_{u1} \)}]{};
\path (level9) node(91)[dot]{}++(-0.75,0) node(92)[dot]{};
\path (level10) node(101)[biggerdot]{}++(-0.75,0) node(102)[biggerdot]{}++(-0.75,0) node(10k)[biggerdot]{} node(v35)[terminal][label={[xshift=-3pt]\( w_{u2} \)}]{};
\path (level11) node(111)[dot]{}++(-0.75,0) node(112)[dot]{};
\path (level12) node(121)[biggerdot]{}++(-0.75,0) node(122)[biggerdot]{}++(-0.75,0) node(12k)[biggerdot]{} node(v36)[terminal][label={[xshift=-3pt]\( w_{u3} \)}]{};

\path (level13) node(131)[dot]{}++(-0.75,0) node(132)[dot]{};
\path (level14) node(141)[biggerdot]{}++(-0.75,0) node(142)[biggerdot]{}++(-0.75,0) node(14k)[biggerdot]{} node(v17)[terminal]{};
\path (level15) node(151)[dot]{}++(-0.75,0) node(152)[dot]{};
\path (level16) node(161)[biggerdot]{}++(-0.75,0) node(162)[biggerdot]{}++(-0.75,0) node(16k)[biggerdot]{} node(v18)[terminal]{};
\path (level17) node(171)[dot]{}++(-0.75,0) node(172)[dot]{};
\path (level18) node(181)[biggerdot]{}++(-0.75,0) node(182)[biggerdot]{}++(-0.75,0) node(18k)[biggerdot]{} node(v19)[terminal]{};
\path (level19) node(191)[dot]{}++(-0.75,0) node(192)[dot]{};
\path (level20) node(201)[biggerdot]{}++(-0.75,0) node(202)[biggerdot]{}++(-0.75,0) node(20k)[biggerdot]{} node(v110)[terminal]{};

\draw
(11) -- (21)
(11) -- (22)
(11) -- (2k);
\draw
(12) -- (21)
(12) -- (22)
(12) -- (2k);
\draw (21)--(31)  (22)--(32);
\draw
(31) -- (41)
(31) -- (42)
(31) -- (4k);
\draw
(32) -- (41)
(32) -- (42)
(32) -- (4k);
\draw (41)--(51)  (42)--(52);
\draw
(51) -- (61)
(51) -- (62)
(51) -- (6k);
\draw
(52) -- (61)
(52) -- (62)
(52) -- (6k);
\draw (61)--(71)  (62)--(72);
\draw
(71) -- (81)
(71) -- (82)
(71) -- (8k);
\draw
(72) -- (81)
(72) -- (82)
(72) -- (8k);
\draw (81)--(91)  (82)--(92);
\draw
(91) -- (101)
(91) -- (102)
(91) -- (10k);
\draw
(92) -- (101)
(92) -- (102)
(92) -- (10k);
\draw (101)--(111)  (102)--(112);
\draw
(111) -- (121)
(111) -- (122)
(111) -- (12k);
\draw
(112) -- (121)
(112) -- (122)
(112) -- (12k);
\draw (121)--(131)  (122)--(132);
\draw
(131) -- (141)
(131) -- (142)
(131) -- (14k);
\draw
(132) -- (141)
(132) -- (142)
(132) -- (14k);
\draw (141)--(151)  (142)--(152);
\draw
(151) -- (161)
(151) -- (162)
(151) -- (16k);
\draw
(152) -- (161)
(152) -- (162)
(152) -- (16k);
\draw (161)--(171)  (162)--(172);
\draw
(171) -- (181)
(171) -- (182)
(171) -- (18k);
\draw
(172) -- (181)
(172) -- (182)
(172) -- (18k);
\draw (181)--(191)  (182)--(192);
\draw
(191) -- (201)
(191) -- (202)
(191) -- (20k);
\draw
(192) -- (201)
(192) -- (202)
(192) -- (20k); 

\path (level1) +(-0.5,0.5) node{(gadget for \( w \))};

\path (ogLevel13) ++(6,0) coordinate (level1);
\path (level1)
++(0,-1) coordinate (level2) 
++(0,-0.75) coordinate (level3) 
++(0,-1) coordinate (level4) 
++(0,-0.75) coordinate (level5) 
++(0,-1) coordinate (level6)
++(0,-0.75) coordinate (level7)
++(0,-1) coordinate (level8)
++(0,-0.75) coordinate (level9)
++(0,-1) coordinate (level10)
;

\path (level1) node(11)[dot]{}++(-0.75,0) node(12)[dot]{};
\path (level2) node(21)[biggerdot]{}++(-0.75,0) node(22)[biggerdot]{}++(-0.75,0) node(2k)[biggerdot]{} node(v41)[terminal][label={[xshift=-3pt]\( x_{u1} \)}]{};
\path (level3) node(31)[dot]{}++(-0.75,0) node(32)[dot]{};
\path (level4) node(41)[biggerdot]{}++(-0.75,0) node(42)[biggerdot]{}++(-0.75,0) node(4k)[biggerdot]{} node(v42)[terminal][label={[xshift=-3pt]\( x_{u2} \)}]{};
\path (level5) node(51)[dot]{}++(-0.75,0) node(52)[dot]{};
\path (level6) node(61)[biggerdot]{}++(-0.75,0) node(62)[biggerdot]{}++(-0.75,0) node(6k)[biggerdot]{} node(v43)[terminal][label={[xshift=-3pt]\( x_{u3} \)}]{};

\path (level7) node(71)[dot]{}++(-0.75,0) node(72)[dot]{};
\path (level8) node(81)[biggerdot]{}++(-0.75,0) node(82)[biggerdot]{}++(-0.75,0) node(8k)[biggerdot]{} node(v44)[terminal]{};
\path (level9) node(91)[dot]{}++(-0.75,0) node(92)[dot]{};
\path (level10) node(101)[biggerdot]{}++(-0.75,0) node(102)[biggerdot]{}++(-0.75,0) node(10k)[biggerdot]{} node[terminal]{};

\draw
(11) -- (21)
(11) -- (22)
(11) -- (2k);
\draw
(12) -- (21)
(12) -- (22)
(12) -- (2k);
\draw (21)--(31)  (22)--(32);
\draw
(31) -- (41)
(31) -- (42)
(31) -- (4k);
\draw
(32) -- (41)
(32) -- (42)
(32) -- (4k);
\draw (41)--(51)  (42)--(52);
\draw
(51) -- (61)
(51) -- (62)
(51) -- (6k);
\draw
(52) -- (61)
(52) -- (62)
(52) -- (6k);
\draw (61)--(71)  (62)--(72);
\draw
(71) -- (81)
(71) -- (82)
(71) -- (8k);
\draw
(72) -- (81)
(72) -- (82)
(72) -- (8k);
\draw (81)--(91)  (82)--(92);
\draw
(91) -- (101)
(91) -- (102)
(91) -- (10k);
\draw
(92) -- (101)
(92) -- (102)
(92) -- (10k);

\path (level10) +(-0.5,-0.5) node{(gadget for \( x \))};

\draw (v11) --node[dot][label=\( e^{(1)}_1 \)]{} (v24);
\draw (v12) --node[dot][label=\( e^{(1)}_2 \)]{} (v25);
\draw (v13) --node[dot][label=\( e^{(1)}_3 \)]{} (v26);
\draw (v14) --node[dot][label=\( e^{(2)}_1 \)]{} (v34);
\draw (v15) --node[dot][label=\( e^{(2)}_2 \)]{} (v35);
\draw (v16) --node[dot][label=\( e^{(2)}_3 \)]{} (v36);
\draw (ogv17) --node[dot][label=\( e^{(4)}_1 \)]{} (v41);
\draw (ogv18) --node[dot][label=\( e^{(4)}_2 \)]{} (v42);
\draw (ogv19) --node[dot][label=\( e^{(4)}_3 \)]{} (v43);
\draw (v21) --node[dot][label={[xshift=2pt]\( e^{(3)}_1 \)}]{} (v31);
\draw (v22) --node[dot][label={[xshift=2pt]\( e^{(3)}_2 \)}]{} (v32);
\draw (v23) --node[dot][label={[xshift=2pt]\( e^{(3)}_3 \)}]{} (v33);
\end{tikzpicture}
\caption[Example of Construction~\ref{make:3-acyclic R_swap_n_auto} with \( k=3 \).]{Example of Construction~\ref{make:3-acyclic R_swap_n_auto} with \( k=3 \) (here, \( \lambda(u)=3 \), \( \lambda(v)=1 \), \( \lambda(w)=4 \) and \( \lambda(x)=2 \)). Graph \( G^* \) is displayed large, and graph \( G \) is shown inset (for convenience, a graph of maximum degree 3 rather than 8 is used as \( G \)).}
\label{fig:eg 3-acyclic R_swap_n_auto}
\end{figure}

\begin{proof}[Proof of Guarantee]
First, we construct a surjective mapping \( \phi \) from the set of 3-acyclic colourings of \( G^* \) to the set of 3-colourings of \( G \). 
Then, we show that \( \phi \) gives a bijection from the set of 3-acyclic colourings of \( G^* \) up to automorphisms to the set of 3-colourings of \( G \). 
This proves that \( \phi \) is a bijection from the set of 3-acyclic colourings of \( G^* \) up to colour swaps and automorphisms to the set of 3-colourings of \( G \) up to colour swaps (that is, if \( \mathcal{R}^* \) is the equivalence relation \( \mathcal{R}_{\text{swap+auto}}(G^*,3) \) restricted to the set of 3-acyclic colourings of \( G^* \), then \( \phi \) is a bijection from the set of equivalence classes of \( \mathcal{R}^* \) to the set of equivalence classes of \( \mathcal{R}_{\text{swap}}(G,3) \)). 
Before constructing \( \phi \), we discuss the structure of automorphisms of \( G^* \).

Let \( \text{Aut}(G^*,v) \) denote the set of automorphisms \( \psi \) of \( G^* \) such that \( \psi \) fixes all vertices not in \( \text{chain}(v) \); i.e., \( \psi(x)=x \) for all \( x\in V(G^*)\setminus V(\text{chain}(v)) \). 
Since no two chain gadgets have the same number of terminals, each automorphism of \( G^* \) is in a sense composed of automorphisms of the chain gadgets. 
This is formally expressed as Claim~\ref{clm:aut G' made of auts of chains} below (see Section~5 in the supplementary material for proof). 
\setcounter{claim}{0}
\begin{claim}\label{clm:aut G' made of auts of chains}
For every automorphism \( \psi \) of \( G^* \), there exists an automorphism \( \psi_v \) for each \( v\in V(G) \) such that \( \psi \) equals the function composition of \( \psi_v \)'s\\
(i.e., if \( V(G)=\{v_1,v_2,\dots,v_n\} \), then for every automorphism \( \psi \) of \( G^* \), there exist\\ \( \psi_1 \in \text{Aut}(G^*,v_1) \), \dots, \( \psi_n \in \text{Aut}(G^*,v_n) \) such that \( \psi=\psi_1\circ \psi_2\circ \dots \circ \psi_n \)). 
\end{claim}

Moreover, we have the following claim (for a proof, see Claim~1.6 in Section~5 of the supplementary material). 
\begin{claim}\label{clm:aut of G* maps level j of chain(v) to level j of chin(v)}
For each vertex \( v \) of \( G \), each automorphism \( \psi \) of \( G^* \) maps each vertex in Level~\( j \) of \( \text{chain}(v) \) to some vertex in Level~\( j \) of \( \text{chain}(v) \), where \( j\in \mathbb{N} \). 
\end{claim}

Note that the chain gadget used here is the special case \( k=3 \) of the chain gadget in Lemma~\ref{lem:acyclic colouring chain gadget properties}. 
By Lemma~\ref{lem:acyclic colouring chain gadget properties}, this chain gadget has exactly one 3-acyclic colouring up to colour swaps and automorphisms. 
In particular, the terminals of a chain gadget get the same colour under a 3-acyclic colouring (and we shall call this colour as the \emph{colour of the chain gadget}). 
Note that if \( G \) is given as input, the output graph \( G' \) in Construction~\ref{make:acyclic} with \( k=3 \) is a subgraph of \( G^* \)  (compare Figure~\ref{fig:eg make acyclic v2} with Figure~\ref{fig:eg 3-acyclic R_swap_n_auto}). 
Hence, as in Construction~\ref{make:acyclic}, for each edge \( uv \) of \( G \), the colour of \( \text{chain}(u) \) differs from the colour of \( \text{chain}(v) \) under each 3-acyclic colouring \( f^* \) of \( G^* \). 
Hence, for every 3-acyclic colouring \( f^* \) of \( G^* \), there exists a corresponding 3-colouring \( f \) of \( G \) such that \( f(v) \) equals the colour of \( \text{chain}(v) \) for each \( v\in V(G) \).

Let \( \phi \) be the function that maps each 3-acyclic colouring \( f^* \) of \( G^* \) to this corresponding 3-colouring \( f \) of \( G \). 
The next claim shows that \( \phi \) is onto.
\begin{claim}\label{clm:phi is onto}
Every 3-colouring \( f \) of \( G \) has a preimage under \( \phi \).
\end{claim}
Let \( f \) be a 3-colouring of \( G \). 
We claim that the colouring \( f^* \) of \( G^* \) defined as follows is a preimage of \( f \) under \( \phi \): for each vertex \( v \) of \( G \), colour the chain gadget for vertex \( v \) by assigning the colour \( f(v) \) on vertices of Level \( 2j \) and the remaining two colours on vertices of Level \( 2j-1 \) for each \( j \); whenever \( e=uv \) is an edge in \( G \), colour the vertices \( e_1,e_2 \) and \( e_3 \) by the only colour different from both \( f(u) \) and \( f(v) \). 
Since the paths of the form \( u_{v\,j},e_j,v_{u\,j} \) are tricoloured, any bicoloured cycle in \( G^* \) must be entirely within a chain gadget. 
But, an acyclic colouring scheme is used on each chain gadget. 
Therefore, there is no cycle in \( G^* \) bicoloured by \( f^* \). 
This proves Claim~\ref{clm:phi is onto}. 

\begin{claim}\label{clm:same up to swaps and auto imply image is same}
If two 3-acyclic colourings \( f_1^* \) and \( f_2^* \) of \( G^* \) are the same up to automorphisms, then \( \phi(f_1^*)=\phi(f_2^*) \).
\end{claim}
Let \( f_1^* \) and \( f_2^* \) be 3-acyclic colourings of \( G^* \) which are the same up to automorphisms.
That is, there exists an automorphism \( \psi \) of \( G^* \) such that \( f_1^*(\psi(x))=f_2^*(x) \) for all \( x\in V(G^*) \). 
By the special case \( k=3 \) of Lemma~\ref{lem:acyclic colouring chain gadget properties}, the chain gadget has exactly one 3-acyclic colouring up to colour swaps and automorphisms, namely the colouring in Figure~\ref{fig:k-acyclic colouring of chain gadget}. 
Observe that all even-level vertices have the same colour in Figure~\ref{fig:k-acyclic colouring of chain gadget}. 
Also, observe that every automorphism of the chain gadget maps vertices on Level~\( j \) to vertices on Level~\( j \) for each \( j\in \mathbb{N} \) (for a proof, see Claim~1.5 in Section~5 of the supplementary material). 
Hence, for every 3-acyclic colouring of the chain gadget, there is a colour \( c \) such that all even-level vertices of the chain gadget are coloured~\( c \). 
In particular, for each \( v\in V(G) \) and each \( i\in \{1,2\} \), there is a colour \( c_v^{(i)} \) such that \( f_i^* \) assigns colour \( c_v^{(i)} \) on all even-level vertices of \( \text{chain}(v) \). 
Consider an arbitrary vertex \( v \) of \( G \) and an arbitrary vertex \( x_v \) at Level~\( 2j \) of \( \text{chain}(v) \) for some \( j\in \mathbb{N} \). 
We know that \( f_i^* \) assigns colour \( c_v^{(i)} \) on even-level vertices of \( \text{chain}(v) \) for \( i\in \{1,2\} \). 
In particular, \( f_1^*(x_v)=c_v^{(1)} \). 
Since \( \psi \) maps \( x_v \) to a vertex in Level~\( 2j \) of \( \text{chain}(v) \), \( f_2^*(\psi(x_v))=c_v^{(2)} \). 
Since \( f_2^*(\psi(x))=f_1^*(x) \) for all \( x\in V(G^*) \), we have \( c_v^{(1)}=f_1^*(x_v)=f_2^*(\psi(x_v))=c_v^{(2)} \). 
For \( i\in \{1,2\} \), \( f_i^* \) assigns colour~\( c_v^{(i)} \) on even-level vertices of \( \text{chain}(v) \)  and in particular terminals of \( \text{chain}(v) \). 
Thus, \( c_v^{(i)} \) is the colour of \( \text{chain}(v) \) under \( f_i^* \) for \( i\in \{1,2\} \). 
Hence, \( (\phi(f_i^*))(v)=c_v^{(i)} \) for \( i\in \{1,2\} \). 
Since \( c_v^{(1)}=c_v^{(2)} \) , we have \( (\phi(f_1^*))(v)=(\phi(f_2^*))(v) \). 
Since \( v \) is arbitrary, \( \phi(f_1^*)=\phi(f_2^*) \). 
This proves Claim~\ref{clm:same up to swaps and auto imply image is same}.

\begin{claim}\label{clm:image is same imply same up to swaps and auto}
If \( \phi(f_1^*)=\phi(f_2^*) \) for two 3-acyclic colourings \( f_1^* \) and \( f_2^* \) of \( G^* \), then \( f_1^* \) and \( f_2^* \) are the same up to automorphisms.
\end{claim}
Let \( f_1^* \) and \( f_2^* \) be two 3-acyclic colourings of \( G^* \), and let \( \phi(f_1^*)=\phi(f_2^*) \). 
Consider an arbitrary vertex \( v \) of \( G \). 
Suppose that \( \text{chain}(v) \) has \( \ell \) terminals. 
Since \( \phi(f_1^*)=\phi(f_2^*) \), the colour of \( \text{chain}(v) \) under \( f_1^* \) is equal to the colour of \( \text{chain}(v) \) under \( f_2^* \). 
By the special case \( k=3 \) of Lemma~\ref{lem:acyclic colouring chain gadget properties}, the chain gadget has exactly one 3-acyclic colouring up to colour swaps and automorphisms, namely the colouring in Figure~\ref{fig:k-acyclic colouring of chain gadget}. 
Observe that in Figure~\ref{fig:k-acyclic colouring of chain gadget}, (i)~all even-level vertices have colour~0, and (ii)~for \( k=3 \), for each \( j\in \{1,2,\dots,\ell\} \), vertices at Level~\( 2j-1 \) are assigned colours 1 and 2 from left to right. 
Also, observe that every automorphism of the chain gadget maps vertices on Level~\( j \) to vertices on Level~\( j \) for each \( j\in \mathbb{N} \) (for a proof, see Claim~1.5 in Section~5 of the supplementary material). 
Hence, for every 3-acyclic colouring of \( \text{chain}(v) \), there is a colour \( c \) such that (i)~all even-level vertices (of \( \text{chain}(v) \)) are coloured~\( c \), and (ii)~for each \( j\in \{1,2,\dots,\ell\} \), vertices at Level~\( 2j-1 \) are assigned colours 1 and 2 in some order. 
Thus, for \( i\in \{1,2\} \), \( f_i^* \) assigns the same colour, say colour \( c_i \), on even-level vertices of \( \text{chain}(v) \) and in particular on the terminals of \( \text{chain}(v) \). 
Since the colour on terminals of \( \text{chain}(v) \) under \( f_1^* \) (i.e., the colour of \( \text{chain}(v) \) under \( f_1^* \)) is equal to the colour on terminals of \( \text{chain}(v) \) under \( f_2^* \), we have \( c_1=c_2 \). 
That is, both \( f_1^* \) and \( f_2^* \) assign the same colour, say colour~\( c \), on even-level vertices of \( \text{chain}(v) \). 
Owing to this and the fact that both \( f_1^* \) and \( f_2^* \) assign a permutation of colours \( \{0,1,2\}\setminus \{c\} \) on vertices at Level~\( 2j-1 \) of \( \text{chain}(v) \) for each \( j\in \{1,2,\dots,\ell\} \), \( f_1^* \) restricted to the vertex set of \( \text{chain}(v) \)  (i.e., \( {f_1^*}_{\restriction{V(\text{chain}(v))}} \)) can be obtained from \( {f_2^*}_{\restriction{V(\text{chain}(v))}} \) by applying a permutation of colours on the set of vertices at Level~\( 2j-1 \) for each \( j\in \{1,2,\dots,\ell\} \). 
Since applying a permutation of colours on the set of vertices at Level~\( 2j-1 \) for each \( j\in \{1,2,\dots,\ell\} \) corresponds to an automorphism of \( \text{chain}(v) \) (see Figure~2 in the supplement for a demonstration), there exists an automorphism \( \psi_v^* \) of \( \text{chain}(v) \) such that \( {f_1^*}_{\restriction{V(\text{chain}(v))}}={f_2^*}_{\restriction{V(\text{chain}(v))}}\circ \psi_v^* \). 
Hence, \( f_1^*(x)=f_2^*(\psi_v^*(x)) \) for each \( x\in V(\text{chain}(v)) \). 
Define \( \psi_v\colon V(G^*)\to V(G^*) \) as \( \psi_v(x)=\psi_v^*(x) \) for each \( x\in V(\text{chain}(v)) \) and \( \psi_v(x)=x \) otherwise. 
Clearly, \( f_1^*(x)=f_2^*(\psi_v(x)) \) for each \( x\in V(\text{chain}(v)) \). 
Define \( \psi \) as the function composition of \( \psi_v \)'s (i.e., \( \psi=\psi_{v_1}\circ \psi_{v_2}\circ \dots \circ \psi_{v_n} \) if \( V(G)=\{v_1,v_2,\dots,v_v\} \)). 
Since \( f_1^*(x)=f_2^*(\psi_v(x)) \) for each \( x\in V(\text{chain}(v)) \) and \( \psi_v \) is an automorphism of \( G^* \) that fixes vertices not in \( \text{chain}(v) \) for each \( v\in V(G) \), we have \( f_1^*(x)=f_2^*(\psi(x)) \) for each vertex \( x \) in some chain gadget of \( G^* \). 
To prove that \( f_1^*=f_2^*\circ \psi \), it suffices to show that \( f_1^*(x)=f_2^*(x) \) for each vertex \( x \) of \( G^* \) which is not in any gadget; i.e., \( x \) is a vertex of the form \( e_j^{(t)} \) in Figure~\ref{fig:eg 3-acyclic R_swap_n_auto}, where \( e^{(t)}=uv \) is an edge in \( G \). 
Recall that the vertex \( e_j^{(t)} \) is adjacent to both \( u_{vj} \) and \( v_{uj} \) in \( G^* \). 
For \( i\in \{1,2\} \), \( \phi(f_i^*) \) is a 3-colouring of \( G^* \), and thus \( f_i^*(u_{vj})=(\phi(f_i^*))(u)\neq (\phi(f_i^*))(v)=f_i^*(v_{uj}) \). 
For \( i\in \{1,2\} \), since \( f_i^*(u_{vj})\neq f_i^*(v_{uj}) \), the colour of \( e_j^{(t)} \) under \( f_i^* \) is the unique colour in \( \{0,1,2\}\setminus \{f_i^*(u_{vj}),f_i^*(v_{uj})\} \). 
In other words, the colour of \( e_j^{(t)} \) under \( f_i^* \) is the unique colour in \( \{0,1,2\}\setminus \{(\phi(f_i^*))(u),(\phi(f_i^*))(v)\} \). 
Since \( \phi(f_1^*)=\phi(f_2^*) \), the unique colour in \( \{0,1,2\}\setminus \{(\phi(f_1^*))(u),(\phi(f_1^*))(v)\} \) is the same as the unique colour in \( \{0,1,2\}\setminus \{(\phi(f_2^*))(u),(\phi(f_2^*))(v)\} \). 
That is, \( f_1^*(e_j^{(t)})=f_2^*(e_j^{(t)}) \). 
This completes the proof of \( f_1^*=f_2^*\circ \psi \), and thus proves Claim~\ref{clm:image is same imply same up to swaps and auto}.

By Claim~\ref{clm:same up to swaps and auto imply image is same} and Claim~\ref{clm:image is same imply same up to swaps and auto}, \( \phi \) gives a bijection from the set of 3-acyclic colourings of \( G^* \) up to automorphisms to the set of 3\nobreakdash-colourings of~\( G \). 
Therefore, \( \phi \) also gives a bijection from the set of 3-acyclic colourings of \( G^* \) up to colour swaps and automorphisms to the set of 3-colourings of \( G \) up to colour swaps. 
In particular, the number of 3-colourings of \( G \) up to colour swaps is equal to the number of 3-acyclic colourings of \( G^* \) up to colour swaps and automorphisms. 
This proves the guarantee.
\end{proof}

Let us consider the time complexity of Construction~\ref{make:3-acyclic R_swap_n_auto}. 
Assume that \( V(G)=\{v_1,\allowbreak v_2,\dots,v_n\} \), where \( \deg_G(v_1)\leq \deg_G(v_2)\leq \dots \leq \deg_G(v_n) \) and \( \lambda(v_i)=i \) for \( 1\leq i\leq n \). 
Let \( m=|E(G)| \). 
For \( 1\leq i\leq n \), since \( \text{chain}(v_i) \) has \( \deg_G(v_i)+i \) terminals, there are at most \( 5(\deg_G(v_i)+i)\leq 5\deg_G(v_i)+5n \) vertices and at most \( 6(\deg_G(v_i)+i)+2(\deg_G(v_i)+i-1)\leq 8\deg_G(v_i)+8n \) edges in \( \text{chain}(v_i) \). 
There are \( 3m \) vertices and \( 6m \) edges that are not in any chain gadget in \( G' \). 
Thus, in \( G' \), there are at most \( \left(\sum_{i=1}^{n} (5\deg_G(v_i)+5n)\right)+3m\leq 10m+5n^2+3m\leq 18\,n^2 \) vertices and \( \left(\sum_{i=1}^{n} (8\deg_G(v_i)+8n)\right)+6m\leq 16m+8n^2+6m\leq 30\,n^2 \) edges. 
Hence, Construction~\ref{make:3-acyclic R_swap_n_auto} requires only time polynomial in \( n \). 
Construction~\ref{make:3-acyclic R_swap_n_auto} establishes a reduction from the problem \textsc{Another 3-Colouring}\,[\( \mathcal{R}_{\text{swap}} \)]\,(\( \Delta=8 \)) to the problem \textsc{Another 3-Acyclic Colouring}\,[\( \mathcal{R}_{\text{swap+auto}} \)]\,(\( \Delta=4 \)). 
Since \textsc{Another 3-Colouring}\,[\( \mathcal{R}_{\text{swap}} \)]\,(\( \Delta=8 \)) is NP-complete (see page~\pageref{lnk:another 3-colouring}), we have the following theorem. 
\begin{theorem}\label{thm:3-acyclic R_swap max degree 4}
For bipartite graphs of maximum degree~4, \textsc{Another 3-Acyclic Colouring}\,\( [\mathcal{R}_{\text{swap+auto}}] \) is NP-complete, and thus \textsc{Unique 3-Acyclic Colouring}\,\( [\mathcal{R}_{\text{swap+auto}}] \) is coNP-hard. 
\qed
\end{theorem}

Observe that there is no universal vertex (i.e., a vertex adjacent to every other vertex) in the output graph in Construction~\ref{make:3-acyclic R_swap_n_auto}. 
Hence, we have the following corollary. 
\begin{corollary}\label{cor:3-acyclic R_swap_n_auto no uni}
\textsc{Another 3-Acyclic Colouring}\,\( [\mathcal{R}_{\text{swap+auto}}] \) is NP-complete for graphs without any universal vertex. 
\end{corollary}

Next, we show that for \( k\geq 4 \), \textsc{Another \( k \)-Acyclic Colouring}\,\( [\mathcal{R}_{\text{swap+auto}}] \) is NP-complete. 

\begin{figure}[hbt]
\centering
\begin{minipage}[b]{0.35\textwidth}
\centering
\begin{tikzpicture}
\coordinate (level1) node[left=0.5cm]{Level 1};
\path (level1)++(0,-1) coordinate (level2) node[left=0.5cm]{Level 2}++(0,-0.75) coordinate (level3) node[left=0.5cm]{Level 3}
++(0,-1) coordinate (level4) node[left=0.5cm]{Level 4}
++(0,-2) coordinate (level2n-1) node[left=0.5cm]{Level \( 2t-1 \)}
++(0,-1) coordinate (level2n) node[left=0.5cm]{Level \( 2t \)};
\path (level4)-- node[pos=0.5,xshift=0.4cm,font=\LARGE,rotate=90]{.\,.\,.} (level2n-1);

\path (level1) node(11)[bigdot]{}++(0.75,0) node(12)[bigdot]{};
\path (level2) node(21)[dot]{}++(0.75,0) node(22)[dot]{}++(0.75,0) node(2k)[dot]{} node[terminal]{};
\path (level3) node(31)[bigdot]{}++(0.75,0) node(32)[bigdot]{};
\path (level4) node(41)[dot]{}++(0.75,0) node(42)[dot]{}++(0.75,0) node(4k)[dot]{} node[terminal]{};
\path (level2n-1) node(2n-11)[bigdot]{}++(0.75,0) node(2n-12)[bigdot]{};
\path (level2n) node(2n1)[dot]{}++(0.75,0) node(2n2)[dot]{}++(0.75,0) node(2nk)[dot]{} node[terminal]{};

\draw
(11) -- (21)
(11) -- (22)
(11) -- (2k);
\draw
(12) -- (21)
(12) -- (22)
(12) -- (2k);
\draw (21)--(31)  (22)--(32);
\draw 
(31) -- (41)
(31) -- (42)
(31) -- (4k);
\draw 
(32) -- (41)
(32) -- (42)
(32) -- (4k);
\draw (41)--+(0,-0.5)  (42)--+(0,-0.5);
\draw (2n-11)--+(0,0.5)  (2n-12)--+(0,0.5);
\draw 
(2n-11) -- (2n1)
(2n-11) -- (2n2)
(2n-11) -- (2nk);
\draw 
(2n-12) -- (2n1)
(2n-12) -- (2n2)
(2n-12) -- (2nk);

\end{tikzpicture}
\captionsetup{width=0.8\textwidth}
\caption{Chain gadget in Construction~\ref{make:3-acyclic R_swap_n_auto}.}
\label{fig:3-acyclic chain gadget}
\end{minipage}\hfill
\begin{minipage}[b]{0.6\textwidth}
\centering
\begin{tikzpicture}
\coordinate (level1);
\path (level1)++(0,-1) coordinate (level2) ++(0,-0.75) coordinate (level3)
++(0,-1) coordinate (level4) 
++(0,-2) coordinate (level2n-1) 
++(0,-1) coordinate (level2n);

\path (level2) node(u1)[dot][label=below:\( u_1 \)]{};
\path (level4) node(u2)[dot][label=below:\( u_2 \)]{};
\path (level2n)+(0,0.25) node(uq)[dot][label=below:\( u_q \)]{};
\path (u2)-- node[font=\LARGE,rotate=90]{.\,.\,.} (uq);

\path (u1)--+(3.5,0) node(v1)[dot][label={[name=v1Label]below:\( v_1 \)}]{};
\path (u2)--+(3.5,0) node(v2)[dot][label=below:\( v_2 \)]{};
\path (level2n)--+(3.5,0) node(vn)[dot][label={[name=vnLabel]below:\( v_n \)}]{};
\path (v2)-- node[font=\LARGE,rotate=90]{.\,.\,.} (vn);

\node (mrFit)[fit={(v1Label)  (vn)},ellipse,inner sep=0pt,outer sep=0pt]{}; \path (mrFit) node[ellipse,draw,dotted,minimum width=1.25cm,minimum height=6cm][label=below:\( G \)]{};
\node (mrFit2)[fit={(u1)  (uq)},ellipse,inner sep=0pt,outer sep=0pt]{}; \path (mrFit2) node[ellipse,draw,dotted,minimum width=3.0cm,minimum height=6cm][label=below:\( K_q \)]{};
\draw (u1)--(v1)  (u1)--(v2)  (u1)--(vn);
\draw (u2)--(v1)  (u2)--(v2)  (u2)--(vn);
\draw (uq)--(v1)  (uq)--(v2)  (uq)--(vn);
\draw (u1) to[bend right=80] (u2);
\draw (u2) to[bend right=80] (uq);
\draw (u1) to[bend right=80] (uq);
\end{tikzpicture}
\caption[Graph \( G' \) in Construction~\ref{make:k-acyclic R_swap_n_auto}]{Graph \( G' \) in Construction~\ref{make:k-acyclic R_swap_n_auto} (edges within the copy of \( G \) are not displayed).}
\label{fig:G' k-acyclic R_swap_n_auto}
\end{minipage}\end{figure}

\begin{construct}\label{make:k-acyclic R_swap_n_auto}
\emph{Parameter:} A positive integer \( q \).\\
\emph{Input:} A graph \( G \) without any universal vertex.\\
\emph{Output:} A graph \( G' \).\\
\emph{Guarantee:} \( G \) has a unique 3-acyclic colouring up to colour swaps and automorphisms if and only if \( G' \) has a unique \( (q+3) \)-acyclic colouring up to colour swaps and automorphisms.\\
\emph{Steps:}\\
Let \( v_1,v_2,\dots,v_n \) be the vertices in \( G \). 
Let \( H \) be a graph isomorphic to \( K_q \) with vertex set \( \{u_1,u_2,\dots,u_q\} \). 
To construct \( G' \), introduce a copy of \( G \) and a copy of \( H \), and join each \( u_j \) to each \( v_i \) for \( 1\leq i\leq n \) and \( 1\leq j\leq q \) (i.e., \( G' \) is the graph join of \( G \) and \( K_q \); see Figure~\ref{fig:G' k-acyclic R_swap_n_auto}). 
\end{construct}
\begin{proof}[Proof of Guarantee]
Let \( U=\{u_1,u_2,\dots,u_q\} \) and \( V=\{v_1,v_2,\dots,v_n\} \). 
Since there is no universal vertex in \( G \), the set \( U \) is precisely the set of universal vertices in \( G' \).

First, let us consider the structure of automorphisms of \( G' \). 
Let \( \psi' \) be an automorphism of \( G' \). 
Since automorphisms preserve the vertex degrees \cite[Lemma~1.3.1]{godsil_royle}, \( \psi' \) maps each universal vertex in \( G' \) to a universal vertex in \( G' \) (i.e., \( \psi'(u_j)\in U \) for all \( u_j\in U \)). 
Hence, \( \psi' \) maps each vertex in \( V \) to a vertex in \( V \) (i.e., \( \psi'_{\restriction V} \) is a bijection from \( V \) to itself). 
Since \( \psi' \) is an automorphism of \( G' \) and its restriction to \( V \) is a bijection from \( V \) to itself, \( \psi' \) restricted to \( V \) is an automorphism of \( G'[V] \). 
Since \( G\cong G'[V] \), we have the following. 
\setcounter{claim}{0}
\begin{claim}\label{clm:aut G' restricted to G}
For every automorphism \( \psi' \) of \( G' \), the restriction of \( \psi' \) to \( V(G) \) is an automorphism of \( G \). 
\end{claim}
In the reverse direction, let \( \psi \) be an automorphism of \( G \). 
Define \( \psi'\colon V(G')\to V(G') \) as \( \psi'(v_i)=\psi(v_i) \) for \( 1\leq i\leq n \) and \( \psi'(u_j)=u_j \) for \( 1\leq j\leq q \). 
Clearly, \( \psi' \) is a bijection from \( V(G') \) to itself. 
Next, we show that \( \psi' \) preserves adjacency as well as non-adjacency. 
Since \( \psi=\psi'_{\restriction V(G)} \), it is easy to verify that for \( 1\leq i<j\leq n \),\; \( \psi'(v_i)\psi'(v_j)\in E(G') \) if and only if \( v_iv_j\in E(G') \). 
Each \( u_j\in U \) is a universal vertex in \( G' \) and \( \psi'(u_j)=u_j \); as a result, \( u_jx\in E(G') \) and \( \psi'(u_j)\psi'(x)\in E(G') \) for all \( x\in V(G') \). 
Therefore, \( \psi' \) preserves adjacency as well as non-adjacency, and thus \( \psi' \) is an automorphism of \( G' \). 
Thus, we have the following claim. 
\begin{claim}\label{clm:aut G extension to G'}
For every automorphism \( \psi \) of \( G \), the extension \( \psi' \) of \( \psi \) into \( V(G') \) defined as \( \psi'(v_i)=\psi(v_i) \) for \( 1\leq i\leq n \) and \( \psi'(u_j)=u_j \) for \( 1\leq j\leq q \) is an automorphism of \( G' \). 
\end{claim}

We are ready to prove the guarantee. 
Since \( G' \) is the graph join of \( K_q \) and \( G \), we have \( \chi_a(G')=\min\{\chi_a(K_q)+n,q+\chi_a(G)\} \) by \cite[Lemma~2.1]{lyons} and thus \( \chi_a(G')=\min\{q+n,q+\chi_a(G)\}=q+\chi_a(G) \). 
Hence, \( G \) is 3-acyclic colourable if and only if \( G' \) is \( (q+3) \)-acyclic colourable. 
To complete the proof of the guarantee, it suffices to prove the following
claim.
\begin{claim}\label{clm:3-acyclic col to q+3 acyclic col}
\( G \) has two 3-acyclic colourings \( f_1 \) and \( f_2 \) such that \mbox{\( (f_1,f_2)\notin \mathcal{R}_{\text{swap+auto}}(G,3) \)} if and only if \( G' \) has two \( (q+3) \)-acyclic colourings \( f'_1 \) and \( f'_2 \) such that\\ \( (f'_1,f'_2)\notin \mathcal{R}_{\text{swap+auto}}(G',q+3) \).
\end{claim}

To prove Claim~\ref{clm:3-acyclic col to q+3 acyclic col}, suppose that \( G \) admits two 3-acyclic colourings \( f_1 \) and \( f_2 \) which are not the same up to colour swaps and automorphisms (i.e., \( (f_1,f_2)\notin \mathcal{R}_{\text{swap+auto}}(G,3) \)). 
Without loss of generality, assume that \( f_1 \) and \( f_2 \) use colours \( 0,1 \) and~\( 2 \). 
For \( \ell \in \{1,2\} \), define \( f'_\ell \colon V(G')\to V(G') \) as \( f'_\ell(v_i)=f_\ell(v_i) \) for \( 1\leq i\leq n \) and \( f'_\ell(u_j)=j+2 \) for \( 1\leq j\leq q \). 
Clearly, \( f'_1 \) and \( f'_2 \) are \( (q+3) \)-colourings of \( G' \), and \( f_\ell={f'_\ell}_{\,\restriction V(G)} \) for \( \ell \in \{1,2\} \). 
We claim that \( f'_1 \) and \( f'_2 \) are not the same up to colour swaps and automorphisms. 
To produce a contradiction, assume the contrary. 
That is, there exists a permutation \( \sigma' \) of colours \( 0,1,\dots,q+2 \) and an automorphism \( \psi' \) of \( G' \) such that \( f'_1(\psi'(x))=\sigma'(f'_2(x)) \) for all \( x\in V(G') \). 
Thus, \( f'_1(\psi'(u_j))=\sigma'(f'_2(u_j)) \) for \( 1\leq j\leq q \), and \( f'_1(\psi'(v_i))=\sigma'(f'_2(v_i)) \) for \( 1\leq i\leq n \). 
For \( 1\leq j\leq q \),\; \( f'_1(\psi'(u_j))=\sigma'(j+2) \) since \( f'_2(u_j)=j+2 \). 
Since \( \psi' \) is an automorphism of \( G' \) and \( U \) is the set of universal vertices in \( G' \), \( \psi' \) maps vertices in \( U \) to vertices in \( U \). 
Thus, for each \( j\in \{1,2,\dots,q\} \), there exists a unique \( \ell\in \{1,2,\dots,q\} \) such that \( \psi'(u_j)=u_\ell \). 
Since \( f'_1(u_\ell)=\ell+2 \) for \( 1\leq \ell\leq q \), this implies that for each \( j\in \{1,2,\dots,q\} \), there exists a unique \( \ell\in \{1,2,\dots,q\} \) such that \( \sigma'(j+2)=\ell+2 \). 
That is, \( \sigma' \) restricted to \( \{3,4,\dots,q+2\} \) is a permutation of \( \{3,4,\dots,q+2\} \). 
Hence, \( \sigma' \) restricted to \( \{0,1,2\} \) is a permutation of \( \{0,1,2\} \). 
Let \( \sigma \) be the restriction of \( \sigma' \) to \( \{0,1,2\} \). 
Let \( \psi=\psi'_{\restriction V(G)} \). 
By Claim~\ref{clm:aut G' restricted to G}, \( \psi \) is an automorphism of \( G \). 
Since \( \sigma=\sigma'_{\restriction \{0,1,2\}} \), \( \psi=\psi'_{\restriction V(G)} \), \( f_\ell={f'_\ell}_{\,\restriction V(G)} \) and \( f_\ell' \) use only colours \( 0,1,2 \) on \( V(G) \) for \( \ell \in \{1,2\} \), we can rewrite ``\( f'_1(\psi'(v_i))=\sigma\,'(f'_2(v_i)) \)'' as ``\( f_1(\psi(v_i))=\sigma(f_2(v_i)) \)'' for \( 1\leq i\leq n \). 
Since \( \sigma \) is a permutation of \( \{0,1,2\} \) and \( \psi \) is an automorphism of \( G \), \( f_1 \) and \( f_2 \) are the same up to colour swaps and automorphisms. 
This is a contradiction. 
Thus, by contradiction, \( f'_1 \) and \( f'_2 \) are not the same up to colour swaps and automorphisms. 
Therefore, \( G' \) admits two \( (q+3) \)-acyclic colourings which are not the same up to colour swaps and automorphisms.

Conversely, suppose that \( G' \) has two \( (q+3) \)-acyclic colourings \( f'_1 \) and \( f'_2 \) that are not the same up to colour swaps and automorphisms (i.e., \( (f'_1,f'_2)\notin \mathcal{R}_{\text{swap+auto}}(G',q+3) \)). 
Let \( 0,1,\dots,q+2 \) be the colours used by \( f'_1 \) and \( f'_2 \). 
For \( \ell \in \{1,2\} \) and \( 1\leq j\leq q \), since \( u_j \) is a universal vertex in \( G' \), \( u_j \) is the only vertex coloured \( f'_\ell(u_j) \) by \( f'_\ell \). 
Without loss of generality, assume that \( f'_1(u_j)=f'_2(u_j)=j+2 \) for \( 1\leq j\leq q \). 
Since each \( u_j\in U \) is a universal vertex in \( G' \), we have \( f'_\ell(v_i)\in \{0,1,2\} \) for \( \ell \in \{1,2\} \) and \( 1\leq i\leq n \). 
For \( \ell \in \{1,2\} \), let \( f_\ell={f'_\ell}_{\,\restriction V(G)} \). 
Clearly, \( f_\ell(v_i)\in \{0,1,2\} \) for \( 1\leq i\leq n \), and \( f_\ell \) is a 3-colouring of \( G \) for \( \ell \in \{1,2\} \). 
For \( \ell \in \{1,2\} \), since \( f_\ell \) is the restriction of an acyclic colouring to \( V(G) \), \( f_\ell \) is an acyclic colouring of \( G \). 
Hence, \( f_\ell \) is a 3-acyclic colouring of \( G \) for \( \ell \in \{1,2\} \). 
We claim that \( f_1 \) and \( f_2 \) are not the same up to colour swaps and automorphisms. 
To produce a contradiction, assume the contrary. 
That is, there exists a permutation \( \sigma \) of colours \( 0,1,2 \) and an automorphism \( \psi \) of \( G \) such that \( f_1(\psi(v_i))=\sigma(f_2(v_i)) \) for \( 1\leq i\leq n \). 
Define a permutation \( \sigma' \) of colours \( 0,1,\dots,q+2 \) as \( \sigma'(c)=\sigma(c) \) for \( 0\leq c\leq 2 \) and \( \sigma'(c)=c \) for \( 3\leq c\leq q+2 \). 
By Claim~\ref{clm:aut G extension to G'}, the extension \( \psi' \) of \( \psi \) into \( V(G') \) defined as \( \psi'(v_i)=\psi(v_i) \) for \( 1\leq i\leq n \) and \( \psi'(u_j)=u_j \) for \( 1\leq j\leq q \) is an automorphism of \( G' \). 
For \( 1\leq j\leq q \),\; \( f'_1(\psi'(u_j))=\sigma'(f'_2(u_j)) \) because \( \psi'(u_j)=u_j \), \( f'_1(u_j)=f'_2(u_j)=j+2 \) and \( \sigma'(j+2)=j+2 \). 
We know that \( f_1(\psi(v_i))=\sigma(f_2(v_i)) \) for \( 1\leq i\leq n \). 
Since \( \psi \), \( f_1 \) and \( f_2 \) are restrictions of \( \psi' \), \( f'_1 \) and \( f'_2 \) respectively to \( V(G) \), and \( \sigma \) is the restriction of \( \sigma' \) to \( \{0,1,2\} \), it follows that \( f'_1(\psi'(v_i))=\sigma'(f_2'(v_i)) \) for \( 1\leq i\leq n \). 
Thus, \( f'_1(\psi'(x))=\sigma'(f'_2(x)) \) for all \( x\in V(G') \). 
That is, \( f'_1 \) and \( f'_2 \) are the same up to colour swaps and automorphisms. 
This is a contradiction. 
Hence, by contradiction, \( f_1 \) and \( f_2 \) are not the same up to colour swaps and automorphisms. 
Therefore, \( G \) admits two 3-acyclic colourings which are not the same up to colour swaps and automorphisms. 
This completes the proof of Claim~\ref{clm:3-acyclic col to q+3 acyclic col}. 
\end{proof}

Let us consider the time complexity of Construction~\ref{make:k-acyclic R_swap_n_auto}. 
Let \( m=|E(G)| \). 
Clearly, \( G' \) has only \( n+q \) vertices and \( m+nq+\binom{q}{2}\leq m+nq+q^2 \) edges. 
Since \( q \) is fixed, Construction~\ref{make:k-acyclic R_swap_n_auto} requires only time polynomial in the input size. 
By Corollary~\ref{cor:3-acyclic R_swap_n_auto no uni}, \textsc{Another 3-Acyclic Colouring}\,\( [\mathcal{R}_{\text{swap+auto}}] \) is NP-complete for graphs without any universal vertex. 
For each \( q\in \mathbb{N} \), Construction~\ref{make:k-acyclic R_swap_n_auto} establishes a reduction from this problem to \textsc{Another \( (q+3) \)-Acyclic Colouring}\,\( [\mathcal{R}_{\text{swap+auto}}] \). 
Thus, we have the following.
\begin{theorem}\label{thm:k-acyclic R_swap_n_auto k>=4}
For \( k\geq 4 \), \textsc{Another \( k \)-Acyclic Colouring}\,\( [\mathcal{R}_{\text{swap+auto}}] \) is NP-complete, and thus \textsc{Unique \( k \)-Acyclic Colouring}\,\( [\mathcal{R}_{\text{swap+auto}}] \) is coNP-hard.
\qed
\end{theorem}
We have the following theorem by combining Theorem~\ref{thm:3-acyclic R_swap max degree 4} and Theorem~\ref{thm:k-acyclic R_swap_n_auto k>=4}. 
\begin{theorem}\label{thm:k-acyclic R_swap_n_auto}
For \( k\geq 3 \), \textsc{Another \( k \)-Acyclic Colouring}\,\( [\mathcal{R}_{\text{swap+auto}}] \) is NP-complete, and thus \textsc{Unique \( k \)-Acyclic Colouring}\,\( [\mathcal{R}_{\text{swap+auto}}] \) is coNP-hard.
\qed
\end{theorem}

\iftoggle{forThesis}
{}{\section{Open Problems and Related Works}\label{sec:acyclic open}
Many problems related to acyclic colouring are open even for the class of cubic graphs. Gr\"{u}nbaum~\cite{grunbaum} proved that cubic graphs are 4-acyclic colourable (see also \cite{skulrattanakulchai}). So, \( 3\leq \chi_a(G)\leq 4 \) for every cubic graph \( G \) which is not a forest. Yet, it is unknown whether we can distinguish between the cases \( \chi_a(G)=3 \) and \( \chi_a(G)=4 \) in polynomial time.
\begin{problem}[\cite{angelini_frati}]
What is the complexity of \textsc{3-Acyclic Colourability} in cubic graphs?
\end{problem}
We know that there are infinitely many cubic graphs that are 3-acyclic colourable and infinitely many that are not \cite[Lemma~2]{angelini_frati}. 
Cheng et al.~\cite{cheng} designed a polynomial-time algorithm that finds an optimal acyclic colouring of a subcubic claw-free graph. 
They also proved that there are exactly three subcubic claw-free graphs that require four colours for acyclic colouring. 
Zhang and Bylka~\cite{zhang_bylka} proved that every cubic line graph except \( K_4 \) is 3-acyclic colourable. 
According to a conjecture of Zhu et al.~\cite{zhu2016b}, \textsc{3-Acyclic Colourability} is in P when restricted to cubic planar 3-connected graphs. 
\begin{conjecture}[\cite{zhu2016b}]\label{conj:zhu's}
Every cubic planar 3-connected graph except \( K_4 \), \( Q_3 \) and the dual graph of \( P_4 \)\raisebox{-2pt}{\scalebox{0.8}{\rotatebox{45}{\( \boxtimes \)}}}\( K_2 \) \( ( \)see Figure~\ref{fig:3rd in zhu's conjecture}\( ) \) is 3-acyclic colourable. 
\end{conjecture}
\begin{figure}[hbt]
\centering
\begin{tikzpicture}[scale=0.5]
\coordinate (origin);
\draw (origin) node(a)[bigdot]{}--++(-1,-1.5) node(b)[bigdot]{}--++(2,0) node(c)[bigdot]{}--(a);
\draw
(b)--+(0,-2) node(d)[bigdot]{}
(c)--+(0,-2) node(e)[bigdot]{};
\draw (d)--(e);
\draw
(a)--+(90:3) node(f)[bigdot]{}
(d)--+(-150:3) node(g)[bigdot]{}
(e)--+(-30:3) node(h)[bigdot]{};
\draw (f)--(g)--(h)--(f);
\end{tikzpicture}

\caption{The dual graph of \( P_4 \)\raisebox{-2pt}{\scalebox{0.8}{\rotatebox{45}{\( \boxtimes \)}}}\( K_2 \), where \raisebox{-2pt}{\scalebox{0.8}{\rotatebox{45}{\( \boxtimes \)}}} denotes the graph join operation.}
\label{fig:3rd in zhu's conjecture}
\end{figure}

Regarding the class of bounded-degree graphs, the problem of determining the exact value of \( L_a^{(k)} \) for each \( k\geq 3 \) remains open. 
For \( k\geq 4 \), finding the value of \( L_a^{(k)} \) suffices to characterise the values \( d \) for which \textsc{\( k \)-Acyclic Colourability} in \( d \)-regular graphs is NP-complete (see Theorem~\ref{thm:k acyclic regular}). 

\section{Acknowledgement}
We thank  Emil Je\v{r}\'{a}bek for his valuable comments. }

 \bibliographystyle{plain}

\end{document}